\documentclass[12pt]{article}
\usepackage{amssymb,amsmath,amsthm,amsfonts,amscd,latexsym, dsfont, color, xcolor, tipa, tipx, mathrsfs, upgreek, eucal, wasysym, enumerate}
\usepackage{graphics}
\usepackage{hyperref} 
\usepackage{array}
\hypersetup{colorlinks}
\usepackage{tikz}
\usetikzlibrary{matrix,arrows,decorations.pathmorphing}

\definecolor{indigo}{RGB}{51,0,102}
\definecolor{brightpurple}{RGB}{102,0,153}

\definecolor{fuchsia}{RGB}{180,51,180}
\definecolor{jolightpurple}{RGB}{188,171,240}

\hypersetup{colorlinks,
linkcolor=brightpurple,
filecolor=brightpurple,
urlcolor=indigo,
citecolor=fuchsia}

\usepackage[margin=1.25in]{geometry}

\normalbaselines
\newcommand{\C}{\mathbb{C}}
\newcommand{\Ci}{C^\infty}

\newcommand{\Z}{\mathbb{Z}}
\newcommand{\R}{\mathbb{R}}
\newcommand{\N}{\mathbb{N}}
\newcommand{\Q}{\mathbb{Q}}

\newcommand{\g}{\Gamma}

\newcommand{\A}{\alpha}

\newcommand{\we}{\wedge}

\newcommand{\pa}{\partial}
\newcommand{\pas}{\frac{\partial}{\partial s}}
\newcommand{\pat}{\frac{\partial}{\partial t}}

\newcommand{\ptau}{\frac{\partial}{\partial \tau}}

\newcommand{\om}{\omega}

\newcommand{\M}{\widehat{\mathcal{M}}}

\newcommand{\ds}{\dot{\Sigma}}
\newcommand{\ind}{\mbox{ind}}

\newcommand{\cala}{\mathcal{A}}

\newcommand{\calb}{\mathfrak{B}}
\newcommand{\cale}{\mathcal{E}}

\newcommand{\mhat}{{\mathcal{M}}^{\Jt}}

\newcommand{\calm}{\widehat{\mathcal{M}}^{\Jt}}
\newcommand{\ncal}{\hat{\mathcal{N}}^{\Jt}}

\newcommand{\mbar}{\overline{\mathcal{M}}^{\Jt}}

\newcommand{\calpa}{\mathscr{P}({\A;\baar})}
\newcommand{\baar}{c}

\newcommand{\calc}{u}

\newcommand{\ce}{C_{\epsilon}}
\newcommand{\ee}{\epsilon}

\newcommand{\deebar}{\bar{\pa}_{\Jt}}

\newcommand{\lep}{\lambda_{\varepsilon}}

\newcommand{\czm}{\mu_{CZ}}

\newcommand{\id}{\mathds{1}}
\newcommand{\gax}{x}
\newcommand{\up}{y}

\newcommand{\mult}{\mbox{m}}
\newcommand{\emult}{\mbox{\em m}}

\newcommand{\ga}{\gamma}

\newcommand{\gpk}{\gamma_p^k}
\newcommand{\glp}{\gamma^\ell_+}

\newcommand{\gdm}{\gamma^d_-}
\newcommand{\gp}{{\gamma_+}}
\newcommand{\gm}{{\gamma_-}}

\newcommand{\lo}{\lambda_0}

\newcommand{\vepsilon}{\varepsilon}
\newcommand{\veps}{\varepsilon}

\newcommand{\lens}{L(n+1,n)}

\newcommand{\hopf}{S^1 \hookrightarrow S^3  \overset{h}{\longrightarrow} S^2}

\newcommand{\J}{\mathcal{J}}
\newcommand{\Jt}{{J}}

\newcommand{\sign}{\mbox{sign}}

\newcommand{\curve}{pseudoholomorphic curve}
\newcommand{\curves}{pseudoholomorphic curves}
\newcommand{\murves}{pseudoholomorphic maps}
\newcommand{\murve}{pseudoholomorphic map}

\newtheorem{theorem}{Theorem}

\newtheorem{conj}{Conjecture}
\newtheorem{lemma}[theorem]{Lemma}
\newtheorem{corollary}[theorem]{Corollary}
\newtheorem{proposition}[theorem]{Proposition}
\newtheorem{example}[theorem]{Example}
\newtheorem{remark}[theorem]{Remark}

\newtheorem{definition}[theorem]{Definition}
\numberwithin{theorem}{section}
\numberwithin{equation}{subsection}
\numberwithin{figure}{section}

\title{Automatic transversality in contact homology I: Regularity}
\author{Jo Nelson\footnote{Supported by NSF grant DMS-1303903, the Bell Companies Fellowship and the Fund
for Mathematics at the Institute for Advanced Study.}}
\date{}

\begin{document}
\maketitle 
 \begin{abstract}
This paper helps to clarify the status of cylindrical contact homology, a conjectured contact invariant introduced by Eliashberg, Givental, and Hofer in 2000. We explain  how heuristic arguments fail to yield a well-defined homological invariant in the presence of multiply covered curves.  We then introduce a large subclass of dynamically convex contact forms in dimension 3, termed \emph{dynamically separated}, and demonstrate automatic transversality holds, therby allowing us to define the desired chain complex.  The Reeb orbits of dynamically separated contact forms satisfy a uniform growth condition on their Conley-Zehnder index under iteration, typically up to large action; see Definition \ref{taut}.  These contact forms arise naturally as perturbations of Morse-Bott contact forms such as those associated to $S^1$-bundles.  In subsequent work, we give a direct proof of invariance for this subclass and, when further proportionality holds between the index and action, powerful geometric computations in a wide variety of examples.

   \end{abstract}

\tableofcontents

\section{Motivation and results}\label{intro}
An impressionistic outline of what contact homology and other related invariants should be if all the analytic difficulties could be resolved was given in \cite{EGH}.  For many years the severity of transversality problems were not completely understood in cylindrical contact homology; see \cite{BC, BCE, BOnotrans, CH, CHugh, vk, P09, ML}.  In previous literature \cite{B02, B09, U}, it was stated that when transversality could not be achieved by perturbing the almost complex structure that the difficulty could still be resolved via a \emph{delicate} virtual cycle technique involving multivalued perturbations.  However, full details were never given and recent literature by  \cite{CLW, confute, IP, MW, Pa, TZ} suggests that this procedure is even more delicate than previously indicated.

 Standard methods of perturbing the almost complex structure, $J$, as in \cite{FHS} fail to achieve transversality for the moduli spaces involved in defining a chain complex for contact homology.  Breaking the underlying $S^1$-symmetry by making use of a $S^1$-dependent $J$ is not appropriate for contact homology as the usual gluing arguments fail and as a result, one can no longer prove $\pa^2=0$ and the chain homotopy equations \cite{HN2}.  Fixing the failure of gluing leads to the construction of a non-equivariant Morse-Bott theory analogous to \cite{BOduke}, which is not isomorphic to the desired theory. 
 
An $S^1$-invariant extension of time-dependent almost complex structures and Hamiltonians are used in \cite{BOSHCH} to define an $S^1$-equivariant version of symplectic homology which should agree with cylindrical contact homology when the latter is defined, as in this paper.  The forthcoming papers, \cite{HN2, HN3}, make use of an analogous construction in contact homology as well as obstruction bundle gluing \cite{HT2} to obtain invariance of cylindrical contact homology under choices of nondegenerate dynamically convex contact forms and $J$.  Depending on the desired application, one theory is often better suited for computational problems at hand, so it is important to understand the interplay between these theories as well as in what generality they can be defined.

The polyfold theory developed by Hofer, Wysocki, and Zehnder \cite{HWZsc, HWZsec, HWZgw} will ultimately resolve these severe transversality issues via a completely abstract functional-analytic framework.  However, the work of recasting the moduli problems involved in defining contact homology and formulating an appropriate scheme of abstract perturbations into the formal language of polyfolds is still pending. Moreover, the geometric constructions of this paper benefit the polyfold framework for contact homology, providing insights into concrete computations and applications.

In this paper, we give a rigorous construction of cylindrical contact homology for contact forms in dimension 3 whose Reeb orbits satisfy a uniform growth condition on their Conley-Zehnder indices.  We call such contact forms dynamically separated and give a precise definition in Definition \ref{taut} and numerous examples.   Further details on invariance and computational methods for the class of dynamically separated contact forms will appear in the paper \cite{jocompute}.  These invariance results are obtained more directly than those to appear in \cite{HN2, HN3} for the class of dynamically convex contact forms.


\begin{remark} [Relationship to the dynamically convex case] \em
The calculations of Section \ref{numerology} also appear in \cite{HNdyn}.  In \cite{HNdyn} we make use of intersection theory to exclude the branch covers of trivial cylinders appearing in Lemma \ref{nontrivialtentacles}, the remaining obstruction to defining cylindrical contact homology for dynamically convex contact 3-manifolds.  However, this intersection theory currently relies on certain technical assumptions, which hold when $\pi_1(M)$ is not torsion;  see (*) in Theorem 1.3 and Remark 1.4 of \cite{HNdyn}.  We expect these assumptions to be removable pending further study of the asymptotics of pseudoholomorphic curves and in the meantime, the dynamically separated case allows us to get around the technical assumptions of (*) in many situations. 
\end{remark}

\noindent \textbf{Organization of the article}.  
The rest of Section \ref{intro} gives a comprehensive overview of cylindrical contact homology, a statement of the regularity results obtained in this paper, and a discussion  of dynamically separated contact forms. The basics of pseudoholomorphic curves and the moduli spaces of interest are provided in Section \ref{freddie}.  Issues of transversality for multiply covered asymptotically cylindrical pseudoholomorphic curves and their branched covers in symplectizations are discussed in Section \ref{transversality}, with the proofs of the necessary regularity results and main theorems appearing in Section \ref{regularity}.

\subsection{Two equivalent cylindrical contact chain complexes}\label{attempt}
For the uninitiated we begin with a brief recollection of how one aims to construct contact homology.   Let $(M, \xi)$ be a co-oriented contact manifold of dimension $2n-1$ and let $\A$ be a nondegenerate\footnote{Nondegeneracy of the contact form means that all the Reeb orbits associated to the Reeb vector field $R_\A$ are nondegenerate; see below.} contact form such that ker $\A = \xi$.  The contact form $\A$ uniquely determines the Reeb vector field $R_\A$ by 
\[
\iota(R_\A)d\A=0, \ \ \ \A(R_\A)=1.   
\]
A Reeb orbit $\ga$ of period $T$ with $T>0$, associated to $R_\A$ is defined to be a map
\[
\ga: \R / T\Z \to M
\]
satisfying
\[
\begin{array}{rll}
\dot{\ga}(t):=\dfrac{d\ga}{dt} &=& R_\A(\ga(t)), \\
\ga(0)&=& \ga(T) \\
\end{array}
\]

Two Reeb orbits are considered equivalent if they differ by reparametrization, i.e. precomposition with any translation of $\R / T \Z$ corresponding to the choice of a starting point for the orbit. A Reeb orbit is said to be \textbf{simple}\footnote{Simple is synonymous with embedded in other literature.} whenever the map $\ga: \R / T\Z \to M$ is injective.   If $\ga\colon \R/T\Z \to M$ is a simple Reeb orbit of period $T$ and $k$ a positive integer, then we denote $\ga^k$ to be the \textbf{$k$-fold cover} or iterate of $\ga$, meaning $\ga^k$ is the composition of $\ga$ with $\R/kT\Z \to \R / T\Z$ and has period $kT$.   Denote the the Reeb flow by $\varphi_t$, i.e. $ \dot{\varphi}_t = R_\A(\varphi_t).$  A Reeb orbit is said to be \textbf{nondegenerate} whenever the linearized return map of the flow along $\gamma$,
\[ 
d\varphi_T: (\xi_{\ga(0)}, d\A) \to (\xi_{\ga(T)=\ga(0)},d\A)
\]
has no eigenvalue equal to 1. 

The linearized flow of a $T$-periodic Reeb orbit $\ga$ yields a path of symplectic matrices given by
\[
d\varphi_t:\xi_{\gamma(0)} \to \xi_{\gamma(t)}, \ t\in[0,T].
\]
One can compute the Conley-Zehnder index of $d\varphi_t, \ t\in[0,T],$ however this index is typically dependent on the choice of trivialization $\Phi$ of $\xi$ along $\ga$ used in linearizing the Reeb flow. However, if $c_1(\xi;\Z)=0$ we can use the existence of an (almost) complex volume form on the \textbf{symplectization}\footnote{The \textbf{symplectization} of $(M,\A)$ is given by the manifold  $\R \times M$ and symplectic form
\[
\omega = e^\tau(d\A - \A \wedge d\tau) = d (e^\tau\alpha).
\]
Here $\tau$ is the coordinate on $\R$ and $\A$ is understood to be the 1-form on $\R \times M$, obtained via pullback under the projection $\R \times M \to M$ and $\Jt$ is an \textbf{$\A$-compatible} almost complex structure; see Definition \ref{complexstruc}.} $(\R \times M, d(e^\tau \A), \Jt)$ to obtain a global means of linearizing the flow of the Reeb vector field, as follows.

\subsubsection{Grading}
 For any $\A$-compatible $J$, the symplectic vector bundle $(\xi, d\A, J)$ has a natural $U(n-1)$ structure.  Since this bundle is a (almost) complex bundle, we can take its highest exterior power, which we denote by $\mathcal{K}^*$ called the \textbf{anticanonical bundle} of $M$.  The dual of such  bundle is called the \textbf{canonical bundle}.   If $c_1(\xi; \Z)=0\in H^2(M;\Z)$  then one can trivialize the anticanonical bundle  $\mathcal{K}^*$.  Let 
 \[
 \widetilde{\Phi}: \mathcal{K}^* \to TM \times \C
 \]
 be a choice of such a trivialization.   Note that this amounts to specifying a global complex volume form on $\R \times M$. If $H^1(M;\Q) =0$ then $\widetilde{\Phi}$ (as well as any complex volume form) is unique up to homotopy.  Now we can insist than any local trivialization $\Phi$ of $\xi$, which can be used to linearize the Reeb flow along $\gamma$ must agree with our ``canonically" determined trivialization $\widetilde{\Phi}$.  This gives rise to an absolute $\Z$-grading on the Reeb orbits.  

In this case one can sensibly refer to the Conley-Zehnder index of a Reeb orbit $\ga$, obtaining a $\Z$-grading on the Reeb orbits given by
 \begin{equation}\label{sftgrading}
 |\gamma|=\czm^\Phi(\gamma)+n-3.
 \end{equation}
 Here $\czm^\Phi(\gamma):=\czm(d\varphi_t)\arrowvert_{t\in[0,T]}$ is the Conley-Zehnder index of the path of symplectic matrices obtained from the linearization of the flow along $\ga$, restricted to $\xi$. We note that if $c_1(\xi;\Z)$ only vanishes on $\pi_2(M)$ then on the homotopy class of contractible loops there is a $\Z$-grading which may or may not be the same as the one obtained through the choice of a complex volume form.
 
If $c_1(\xi; \Z)=0$ and  $H^1(M;\Z)\neq0$ then there is more than one homotopy class of trivializations associated to the complex line bundle that is the canonical representation of $-c_1(\xi)$, resulting in different choices of complex volume forms on  $(\R \times M, d(e^\tau \A), \Jt)$.  If $c_1(\xi; \Q)=0$ one can obtain a fractional $\Z$-grading, see  \cite[\S 3.1]{mclean} \cite{sgrade, biased}.   
 
 It is important to note that our trivializations are fixed up to homotopy; that is trivializations over iterated orbits must be homotopic to the iterated trivializations.  When the trivialization $\widetilde{\Phi}$ is available globally as when $c_1(\xi; \Z)=0$ this is straightforward, otherwise care must be taken in specifying local trivializations.  This point will be addressed further in Section \ref{arbitrary}.
  
When a $\Z$-grading is unavailable, there is a canonical $\Z_2$-grading due to the axiomatic properties of the Conley-Zehnder index \cite{RS1,SZ}.  For $(M^{2n-1},\xi)$ this grading is obtained via 
\begin{equation}\label{Z2grading}
(-1)^{\czm(\gamma)}=(-1)^{n-1}\sign \det (\id - \Psi(T)),
\end{equation}
where $\Psi(t)_{t \in [0,T]} \in \mbox{Sp}(2n-2)$ is the linearized flow restricted to $\xi$ along a $T$-periodic Reeb orbit $\gamma$ with respect to the choice of symplectic trivialization $\Phi$ of $\xi$.  

In dimension 3, one can classify a nondegenerate Reeb orbit $\gamma$ as being one of three types, depending on the eigenvalues $\lambda$, $\lambda^{-1}$ of the linearized flow return map along $\gamma$ restricted to $\xi$.
\begin{itemize}
\item[] $\gamma$ is elliptic if $\lambda, \lambda^{-1} := e^{\pm2\pi i \theta}$;
\item[] $\gamma$ is positive hyperbolic if  $\lambda, \lambda^{-1} > 0$; 
\item[] $\gamma$ is negative hyperbolic if $\lambda, \lambda^{-1} < 0$.
\end{itemize}
 The parity of the Conley-Zehnder index does not depend on the choice of trivialization and is even when $\gamma$ is positive hyperbolic and odd otherwise, yielding the canonical $\Z_2$ grading in dimension 3. 

 
We will further need to classify Reeb orbits whose Conley-Zehnder index changes parity under iteration, a phenomenon which is always independent of the choice of trivialization.  Such Reeb orbits are said to be \textbf{bad Reeb orbits} and must be excluded from the chain group due to issues involving orientation and invariance.  More details will be given on bad Reeb orbits, including their exclusion from the chain complex in Remarks \ref{goodbad1}, \ref{goodbad2}, and \ref{goodbad}.

\begin{definition}\em
The $m$-fold closed Reeb orbit $\ga^m$ is \textbf{bad} if it is the $m$-fold covering of some simple Reeb orbit $\ga$ such that the difference $\czm(\ga^{m}) - \czm(\ga)$ of their Conley-Zehnder indices is odd.    
 \end{definition}
In dimension 3 the set of bad orbits consists solely of the even coverings of simple negative hyperbolic orbits, as worked out in \cite{Ylong}.  In higher dimensions one can consult \cite{U} to see that bad orbits can only arise from even multiple covers of nondegenerate simple orbits whose linearized return flow has an odd number of pairs of negative real eigenvalues $(\lambda, \lambda^{-1})$.  If a Reeb orbit is not bad then it is a \textbf{good Reeb orbit}.  The set of all Reeb orbits in the free homotopy class $c$ of $R_\A$ is denoted by $\calpa$, and the set of  good Reeb orbits of $R_\A$ in a free homotopy class $c$ is denoted by $\mathscr{P}_{\mbox{\tiny good}}(\alpha;c)$.
 
The \textbf{chain group} $C_*(M,\A)$ is generated by all nondegenerate closed \textbf{good} Reeb orbits of  $R_\A$ over $\Q$-coefficients, with grading determined by (\ref{sftgrading}).  For a more detailed discussion on other choices of coefficients see Remark \ref{coefficients}.     The chain group splits over the free homotopy classes $\baar \in \pi_0(\Omega M)$ of Reeb orbits, 
 \[
 C_*(M,\A)= \bigoplus_{\baar \in \pi_0(\Omega M)} C^{\baar}_*(M,\A).
 \]
  The reason for this splitting is because the \textbf{differentials} $\pa_-$ and $\pa_+$, defined in (\ref{pa1}) and (\ref{pa2}) respectively, are a weighted count of rigid pseudoholomorphic cylinders interpolating between two closed Reeb orbits, defined as follows.

\subsubsection{Differentials}
In order to precisely define a differential $\pa$ we must define the notion of a pseudoholomorphic curve in a symplectization, which involves the choice of an \textbf{$\A$-compatible} almost complex structure $\Jt$ on $(\R \times M, d(e^\tau\A))$, which is $\R$-invariant and determined by any $J$ compatible with $d\A$ on $\xi$.  

A \textbf{finite energy pseudoholomorphic map} with one positive puncture and one negative puncture
\[
u(s,t):=(a(s,t),f(s,t)): (\R \times S^1, j_0) \to (\R \times M, \Jt)
\]
is a solution of the Cauchy-Riemann equation, 
\[
\bar{\pa}_{j,J}u:= du + J \circ du \circ j \equiv 0,
\]
and asymptotic to parametrized nondegenerate Reeb orbits $\gp$ and $\gm$ of periods $T_+,$ and $T_- $ respectively, meaning that
 \begin{equation}\label{cylasymptotics}
\begin{array}{rcll}
\displaystyle \lim_{s \to \pm \infty} f(s,t) &=& \ga_\pm(T_\pm t) & \mbox{ in } C^\infty(M), \\
\displaystyle \lim_{s \to \pm \infty} a(s,t) &=& \pm \infty & \mbox{ in } C^\infty(\R).\\
\end{array}
\end{equation}

We declare two maps to be equivalent if they differ by translation and rotation of the domain $\R \times S^1$. We call such an equivalence class a \textbf{finite energy pseudoholomorphic cylinders} with one positive and one negative puncture, denoted by $\widehat{\mathcal{M}}(\gp;\gm)$.  The moduli space of (unparametrized) \textbf{rigid cylinders} that $\pa$ counts is
\[
\mhat_0(\gp;\gm):=\calm_1(\gp;\gm)/\R,
\]
where $\czm(\gp)-\czm(\gm)=1$.  The $\R$-action we mod out by is given by vertical translation on $(\R \times M, \Jt)$.   In this paper the only finite energy cylinders we care about  are those with one positive and one negative puncture; henceforth all mention of finite energy cylinders refers to those with one positive and one negative puncture.

While ultimately we are only interested in counting rigid cylinders, we still need to consider moduli spaces of genus 0 curves with one positive and an arbitrary number of negative ends asymptotic to Reeb orbits subject to the finite energy condition of \cite{HWZ1}-\cite{HWZ4}; see Section \ref{area-energy} and \ref{asymptotics} for a definition of finite energy and its implications.  We will need to show that the compactification of $\mhat_1(\gp;\gm)$ does not contain buildings with levels given by noncylindrical curves in the sense of \cite{BEHWZ}. This paper  demonstrates their exclusion via geometric reasons, see Section \ref{quandaries} for an overview of these issues.

These noncylindrical curves are examples of \textbf{asymptotically cylindrical curves}, also known as \textbf{finite energy curves}, described as follows.  Let $(\Sigma, j)$  be a closed Riemann surface and $\Gamma$ be a set of points which are the punctures of $\dot{\Sigma}:=\Sigma \setminus \Gamma$.   Asymptotically cylindrical maps are pseudoholomorphic maps 
\[
u: (\dot{\Sigma}, j) \to (\R \times M, J),
\]
subject to the asymptotic conditions involving $\#|\Gamma|$ Reeb orbits defined momentarily in (\ref{ass1}) and (\ref{ass2}).  

 The moduli space of asymptotically cylindrical curves is the space of equivalence classes of asymptotically cylindrical pseudoholomorphic maps; here an equivalence class is defined by the data $(\Sigma, j, \Gamma, u)$, where $\Gamma$ is an ordered set.  This equivalence is defined in terms of $j$ and the punctures in the usual manner,\footnote{Note that when considering genus 0 domains, we can alternatively fix a standard $j_0$ on $\Sigma:=S^2$ and only keep track of the location of the punctures.} with full details given st the end of Section \ref{pseudsymp}.  

The domain of all the curves of interest in this paper is a multiply punctured sphere $(\dot{\Sigma}, j):=(S^2\setminus \{x,y_1,...,y_{s}\}, j_0)$.  We denote the moduli space of genus 0 \textbf{asymptotically cylindrical pseudoholomorphic} curves with 1 positive end and $s$ negative ends limiting on the Reeb orbits $\gp, \ga_1,...,\ga_s$ by
\[
\mhat(\gp;\ga_1,...,\ga_s):=\calm(\gp;\ga_1,...,\ga_s)/\R
\]
The elements of $\calm(\gp;\ga_1,...,\ga_s)$ are equivalence classes of pseudoholomorphic maps
\[
\begin{array}{rll}
u =(a, f)&:&(S^2\setminus \{x, y_1,...,y_s\}, j_0) \to (\R \times M, \Jt) \\
\bar{\pa}_{j,J}u & :=& du + J \circ du \circ j \equiv 0.\\
\end{array}
\]
subject to (\ref{ass1}) and (\ref{ass2}).  After partitioning the punctures into positive and negative subsets $\g_+:= \{x\}$ and $\g_-:=\{y_1,..y_s\} $ respectively, with $\g = \g_+ \sqcup \g_-$,  we choose a biholomorphic identification of a punctured neighborhood of each $z \in \g_\pm$ with a positive or negative half-cylinder respectively, 
\[
Z_+=[0, \infty) \times S^1, \ \ \ \ \ \ Z_-=(-\infty, 0] \times S^1,
\] 
and choose cylindrical coordinates $(s,t)$ for $u$ near the puncture.  Then for $|s|$ sufficiently large, we require that the following asymptotic formula be satisfied,
\begin{equation}\label{ass1}
u \circ \phi(s,t) = \exp_{(Ts,\ga(Tt))}h(s,t) \in E_\pm.
\end{equation}\label{ass2}
Here $(E_-,J) \cong ((-\infty, 0] \times M, \Jt)$ and   $(E_+,J) \cong ([0,\infty) \times M, \Jt)$

As before, $T >0$ is a constant, $\ga : \R \to M$ is a $T$-periodic Reeb orbit of $R_\A$, and the exponential map is defined with respect to any $\R$-invariant metric on $\R \times M$, $h(s,t) \in \xi|_{\ga(Tt)}$ goes to 0 uniformly in $t$ as $s \to \pm \infty$, and $\phi: Z_\pm \to Z_\pm$ is a smooth embedding such that  
\begin{equation}\label{ass2}
\phi(s,t) - (s+s_0, t+t_0) \to 0 \mbox{ as } s \to \pm \infty
\end{equation}
for some constants $s_0 \in \R$ and $t_0 \in S^1$.  

These asymptotics provide a suitable system of weighted Sobolev spaces used in the study of the linearization of the $\bar{\pa}_{j,J}$-operator, described in \cite{Dr, HWZ3, schwarz}.  This allows us to conclude that the virtual dimension of $\calm(\gamma_+;\gamma_1,...,\gamma_s)$ is given by
\begin{equation}\label{vdimintro}
\mbox{ vdim }\calm(\gamma;\gamma_1,...,\gamma_s)  = (n-3)(1-s) +  \czm\left(\gamma\right) -\displaystyle \sum_{i=1}^s\czm(\gamma_i).
\end{equation}
Note $\mbox{ vdim }\mhat(\gamma;\gamma_1,...,\gamma_s)=\mbox{ vdim }\calm(\gamma;\gamma_1,...,\gamma_s)  -1$. We denote 
\begin{equation}\label{mjd}
\mathcal{M}_d^J(\gp;\gm) := \{ u \in \mhat(\gp;\gm) \ | \ \ind(u) := \mbox{ vdim }\mhat(\gp;\gm) =d \}.  
\end{equation}

We also are interested in genus 0 \textbf{finite energy planes\footnote{Note that $(S^2 \setminus \{x\}, j_0)$ is biholomorphic to $(\C, j_0)$, hence the terminology plane.}}, which are pseudoholomorphic curves
\[
u: (S^2 \setminus \{x\}, j_0) \to (\R \times M, \Jt)
\]
asymptotically cylindrical to a single nondegenerate Reeb orbit $\ga$ at the puncture $x$.  Due to the maximum principle, the puncture of a finite energy plane is always positive. 
We write $u \in \calm(\gamma; \emptyset)$ and the above virtual dimension formula (\ref{vdimintro}) holds, yielding $\mbox{vdim }\calm(\gamma; \emptyset) = (n-3) + \czm\left(\gamma\right) = |\ga|.$  
 
Since we allow $x$ and $y$ to be multiply covered Reeb orbits, elements of $\mhat(x;y)$ will typically no longer be \textbf{somewhere injective.}  The non-specialist will find a definition of somewhere injective and multiply covered curves in Section \ref{modulispaces}.  In fact, we must take the \textbf{multiplicities} $\mult$ of the Reeb orbits and elements of $\mhat(x;y)$ into account in the following expression for $\pa$ to ensure that $\pa^2=0$, given the (expected) geometric structure of the compactified moduli space $\mbar(x;z)$; see Theorem \ref{gluing}.

\begin{definition}[Multiplicities of orbits and curves]\label{multiplicities}\em
If $\tilde{\ga}$ is a closed Reeb orbit, which is a $k$-fold cover of a simple orbit $\ga$, then the \textbf{multiplicity of the Reeb orbit} $\tilde{\ga}$ is defined to be
\[
\mult(\tilde{\ga})=k,
\]
with $\mult(\ga)=1$.  The \textbf{multiplicity of a pseudoholomorphic curve} is 1 if it is somewhere injective.  If the pseudoholomorphic curve $\calc$ is multiply covered then it factors through a somewhere injective curve $v$ and a holomorphic covering $\varphi: (\R \times S^1,j_0) \to (\R \times S^1, j_0),$ e.g. $\calc = v \circ \varphi$. The multiplicity of $\calc$ is defined to be
\[
\mult(\calc):=\mbox{deg}(\varphi).
\]
 If $\calc \in \mhat(x;y)$ then $\mult(\calc)$ divides both $\mult(x)$ and $\mult(y)$.  
\end{definition}

Recall we denoted the set of good Reeb orbits of $R_\A$ in the same free homotopy class $c$ by $\mathscr{P}_{\mbox{\tiny good}}(\alpha;c)$.  We define the operators 
\[
\begin{array}{ccll}
\kappa: & C_*^{\baar}(M,\A, J) &\to& C_{*}^{\baar}(M,\A, J)\\
&x& \mapsto &\mult(x)x \\
\end{array}
\]
and
\[
\begin{array}{ccll}
\delta: & C_*^{\baar}(M,\A, J) &\to& C_{*-1}^{\baar}(M,\A, J)\\
&x& \mapsto &\displaystyle \sum_{\substack{ y \in \mathscr{P}_{\mbox{\tiny good}}(\alpha;c) \\ u \in \mathcal{M}^J_0(x; y)}}\dfrac{\epsilon(u)}{\mult(u)}y, \\
\end{array}
\]
which yield two ways to define the \textbf{differential}\footnote{The different ways of defining differential is rather ambiguous in older literature.} as one can encode the multiplicity of either the top or the bottom Reeb orbit. The expressions for the differentials with respect to the splitting of the chain complex over free homotopy classes $\baar$ of Reeb orbits are given respectively by
\begin{equation}\label{pa1}
\begin{array}{ccll}
\pa_{-}:=\kappa \circ \delta \ \colon & C_*^{\baar}(M,\A, J) &\to& C_{*-1}^{\baar}(M,\A, J) \\
&\gax &\mapsto& \displaystyle \sum_{\substack{y \in \mathscr{P}_{\mbox{\tiny good}}(\alpha;c) \\ u \in \mathcal{M}^J_0(x; y)}}\left( \epsilon(u)\frac{\mult({y})}{\mult(u)}\right) \up \\
\end{array}
\end{equation}
and
\begin{equation}\label{pa2}
\begin{array}{ccll}
\pa_{+}:= \delta \circ \kappa \ \colon & C_*^{\baar}(M,\A, J) &\to& C_{*-1}^{\baar}(M,\A, J) \\
&\gax &\mapsto& \displaystyle \sum_{\substack{y \in \mathscr{P}_{\mbox{\tiny good}}(\alpha;c) \\ u \in \mathcal{M}^J_0(x; y)}} \left( \epsilon(\calc)\frac{\mult({x})}{\mult(u)}\right) \up, \\
\end{array}
\end{equation}
where $\epsilon(\calc) = \pm1$ depends on a choice of \textbf{coherent orientations.}    Coherent orientations for symplectic field theory can be found in \cite{BM} as adapted from Floer theory in \cite{FH}, with additional exposition in Section 9 of \cite{HT2}. A different choice of coherent orientations will lead to different signs in the differential, but the chain complexes will be canonically isomorphic.   

\begin{remark}[Existence of Orientations]\label{goodbad1}\em
When the moduli space $\mathcal{M}^{J}_0(x;y)$ of index 1 cylinders is a 0-manifold, then it can only be oriented by a choice of coherent orientations as in \cite{BM}, provided both $x$ and $y$ are good orbits.  
\end{remark}

\begin{remark}[Choices of coefficients]\label{coefficients}\em
The homologies, $H_*(C_*(M, \A, J), \pa_\pm)$ are equivalent over $\Q$-coefficients, provided sufficient transversality holds to define the chain complexes and obtain invariance in the first place.  The isomorphism between these two chain complexes is then given by $\kappa$ because $(\kappa \delta) \kappa = \kappa (\delta \kappa).$  

While one can always define either differential for cylindrical contact homology over $\Z_2$ or $\Z$-coefficients because the weighted expression is always integral, it is expected that the homologies $H_*(C_*(M, \A, J), \pa_\pm)$ are not typically isomorphic, \cite{HN2, HN3}.  In the case of dynamically separated contact forms $\A$ we have $\pa_- \equiv \pa_+$ because for any $u \in \mhat_0(x;y)$, $\mult(x)=\mult(y)$.  In this case the contact homologies are trivially isomorphic over $\Z_2$ and $\Z$-coefficients.
\end{remark}

\begin{remark}[Well-definedness of the differentials]\em
In order to ensure that both of the expressions (\ref{pa1}) and (\ref{pa2}) are meaningful, i.e. that the counts of curves are finite, one must have proven that all moduli spaces of relevance can be cut out transversally.  When there are multiply covered curves to contend with, there are presently no geometric methods to conclude in higher dimensions that $(C_*,\pa_\pm)$ forms a chain complex or $H(C_*(M, \A, \Jt), \pa_\pm)$ is an invariant of $\ker \A$.  Specific details in regards to these issues are given in Section \ref{quandaries}.
\end{remark}

One must exclude bad Reeb orbits from the chain complex as their inclusion obstructs the proof of invariance, assuming sufficient transversality existed in the first place.

\begin{remark}[Apriori exclusion of bad Reeb orbits]\label{goodbad2}\em
The period doubling example explained in \cite[\S 9]{HN2} demonstrates that one cannot define an invariant version of contact homology which is the homology of a chain complex generated by all (good or bad) Reeb orbits.  If we restrict locally to those orbits that wind twice times around a neighborhood, such a chain complex would have 1 generator before the period-doubling bifurcation and 2 generators after the bifurcation, so its Euler characteristic would not be invariant.
 \end{remark}

The original conjecture from \cite{EGH} is as follows.

\begin{conj}\label{EGH1}
Let $(M^{2n-1}, \xi)$ be a co-oriented\footnote{Co-oriented means that there exists a globally defined $\A \in \Omega^1(M)$ such that $\xi = \ker \A$.} contact manifold. Further assume that all closed orbits of the Reeb vector field associated to $\A$ are nondegenerate and that there are no contractible orbits of grading $|\gamma|=-1,0,1.$ Then for every free homotopy class $\baar$, \\
\indent {\em (i)}  $\pa_\pm^2 = 0$, \\
\indent {\em (ii)} $H_*(C_*^{\baar}(M, \A), \Jt )$ is independent of the contact form $\A$ defining $\xi$ and a generic choice of an $\A$-compatible almost complex structure $\Jt$. 
\end{conj}

\subsubsection{Summary of main results}
The following results of this paper, when combined with on details on gluing and invariance appearing \cite{HNdyn, jocompute}, allow us to demonstrate that one can define cylindrical contact homology and prove invariance under the choice of $\Jt$ and \textbf{dynamically separated} contact form.  The results of \cite{HN2, HN3} prove invariance for the larger class of dynamically convex contact forms.  The definition of a \textbf{dynamically separated} contact form appears in Definition \ref{taut} of Section \ref{dynsep}.  

In the main result we obtain regularity for the wider class of \textbf{dynamically convex} contact forms.  The definition of dynamically convex may be found in Definition \ref{dyncon}, taken from \cite{HWZtight}.  Informally speaking a \textbf{dynamically convex} contact form $\A$ is necessarily defined on a 3-manifold and all contractible orbits $\ga$ associated to $R_\A$ satisfy $\czm(\ga) \geq 3$.  We note that dynamically separated contact forms are a more restrictive subclass of dynamically convex contact forms.

\begin{theorem}[Conditions (A) \& (B)]\label{conditions}
Let $(M,\A)$ be a nondegenerate contact 3-manifold equipped with a dynamically convex contact form.  Then after a generic $\A$-compatible choice of $\Jt$ on $(\R \times M, d(e^\tau\A))$ the following holds.
\begin{enumerate}
\item[\em (A)] All finite energy cylinders of index 1 and 2 are regular, including unbranched multiply covered cylinders;
\item[\em (B)] No finite energy cylinders of index $< 0 $ exist, and the only finite energy cylinders of index 0 are trivial cylinders;
\end{enumerate}
\end{theorem}

\begin{remark}[Condition (C)]\em
Moduli spaces of pseudoholomorphic planes must have (virtual) dimension $\geq  2$ in order to define a cylindrical chain complex without making use of abstract perturbations; see Section \ref{quandaries}.  In dimension 3 this is equivalent to requiring that all contractible Reeb orbits $\ga$ satisfy $\czm(\ga)\geq 3$, as in the definition of dynamically convex.  We will refer to this as Condition (C).  

Currently Hutchings and Nelson are investigating in dimension 3 if one can relax this requirement so that contractible $\ga$ need only satisfy $\czm(\ga)\geq 2$, provided the counts of index zero holomorphic planes are taken to be zero.
\end{remark}

\begin{theorem}[Condition (D)]\label{conditions2} Let $J$ be a generic $\A$-compatible almost complex structure associated to $M^3$ and $\A$ be a nondegenerate contact form which is either dynamically separated or which admits no contractible orbits. Then the following compactness result holds.
\begin{enumerate}
\item[\em (D)] If $x,z \in \calpa$ with $\czm(x)-\czm(z)=2$ then
\begin{equation}\label{compactness}
\mbar_1(x; z) \subseteq \mhat_1(x; z) \  \ \cup \bigcup_{\substack{ y \in \calpa \\ \czm(y) = \czm(x)-1 }}  \mhat_0(x; y) \times \mhat_0(y; z).
\end{equation}
\end{enumerate}
\end{theorem}

\begin{remark}[Baffled by bad orbits]\em
 The index calculations used to obtain Condition (D) do not immediately exclude broken cylinders limiting on bad Reeb orbits.  Remark \ref{goodbad} explains why in Condition (D) $y \in \calpa$ is not problematic in proving $\pa^2$=0, even though the chain complex is generated only by good Reeb orbits.  As explained in the remark, the slightly more refined arguments involving both index and multiplicity of \cite{HNdyn} allow us to conclude that  bad orbits can never show up as asymptotic limits of the broken cylinders.
 \end{remark}

Conditions (A)-(D) are sufficient to ensure that cylindrical contact homology is well-defined, after appealing to the following analogue of Floer's gluing theorem.  This analogue is restated below, with its proof corresponding to Lemma 4.3 and Theorem 1.3 of \cite{HNdyn}, and is a refinement of the gluing theorem appearing in in \cite{jothesis}.  A brief explanation of the proof of Theorem \ref{gluing}(ii) is also given after its statement. 

\begin{theorem}[Gluing]\label{gluing}
Let $(M,\A)$ be a nondegenerate dynamically separated contact 3-manifold and $J$ be a generic $\A$-compatible almost complex structure.  If $x, z \in \mathscr{P}_{\mbox{\scriptsize good}}(\A;c)$
then the following holds:
\begin{enumerate}[\em (i)]
\item For $u \in \mhat_0(x;y)$ and $v \in \mhat_0(y;z)$ there are exactly $\dfrac{ \emult (y)}{\mbox{\em lcm}( \emult(u), \emult(v))}$ ends of the moduli space $\mhat_1(x;z)$ that converge to the building $(u,v)$ and each such end consists of cylinders of multiplicity $\mbox{\em gcd}\left({\emult}(u), \emult(v) \right).$
\item  Let $y$ be a bad orbit and $u \in \mhat_0(x;y)$ and $v \in \mhat_0(y;z)$.  Then the signed count of the ends of $\mhat_1(x;z)$   
which limit to the broken curve $(u,v)$ is 0.
\end{enumerate}
\end{theorem}
In Theorem \ref{gluing}(i) the intermediary orbit $y$ is allowed to be bad.  To see why one obtains (ii), we refer to the following remark, which reproves (ii), assuming (i) holds.  The statement and proof of Theorem \ref{gluing}(i) appears as Lemma 4.3 in \cite{HNdyn}.  Note that Theorem \ref{gluing}( ultimately allows us to prove that $\pa_\pm^2=0$.

\begin{remark}[Proof of Theorem \ref{gluing}(ii)]\label{goodbad}\em
The astute reader likely wonders why Theorem \ref{gluing}(ii) is not belied by Theorems \ref{gluing}(i) and \ref{conditions2}, due to the potential appearance of bad orbits as in (\ref{compactness}). To see why this holds, we repeat the argument given in Section 4.4 of \cite{HNdyn}.

Let $x, z$ be good Reeb orbits such that $\czm(x)-\czm(z)=2$.  From Conditions (A) and (B) and Lemma 4.2(b) of \cite{HNdyn} we know that the moduli space of index 2 cylinders, $\mathcal{M}^J_1(x;z)$ is an oriented 1-manifold and that the covering multiplicity $m$ is constant on each component\footnote{From \cite{Wtrans} Theorem 0, regularity achieved in Condition (A) induces an orbifold structure.  This is upgraded to a manifold structure in Lemma 4.2(b) of \cite{HNdyn} and mentioned in the text after this remark.}.  In order to prove Theorem \ref{gluing})(ii) and obtain that $\partial_\pm^2=0$ we must show the following.
\begin{enumerate}
\item  The moduli space $\mathcal{M}^J_1(x;z)$ has a compactification to a compact oriented 1-manifold $\mbar_1(x; z)$, obtained by attaching one boundary point to each end, such that
\begin{equation}\label{gluestuff}
\sum_{X \in \pi_0 \left( \mbar_1(x; z) \right)} \frac{\# \partial X}{\emult(X)} = \langle \delta \kappa \delta x, y \rangle,
\end{equation}
where $\#\partial X$ denotes the signed count of boundary points of the component $X$, which of course is zero.  Equation (\ref{gluestuff}) implies that $\delta \kappa \delta=0$

\item The pair $(u,v)$ where $u \in \mathcal{M}^J_0(x, y)$ and $v \in \mathcal{M}^J_0(y,z)$ for $y$ good contributes 0 to the left hand side of (\ref{gluestuff}) when $y$ is bad and $\dfrac{\epsilon(u)\epsilon(v)\emult(y)}{\emult(u)\emult(v)}$ when $y$ is good.
\end{enumerate}
The first part of (1.) follows by Conditions (A)-(D), as each end of  $\mathcal{M}^J_2(x;z)/\R$ limits to a building $(u, v)$ where $u \in \mathcal{M}^J_0(x, y)$  and $v \in \mathcal{M}^J_0(y,z)$ for some Reeb orbit $y$.  The second part of (1.) follows from (2.). The proof of (2.) is as follows.

If $y$ is bad, then by Lemma 2.5(b) of \cite{HNdyn} we have $\emult(u),\emult(v) = 1$. Consequently, by Theorem \ref{gluing}(i), there are $\emult(y)$ ends of the moduli space of index 2 cylinders converging to $(u,v)$ each of which has $\emult = 1$. According to \cite{BM}, half of these ends have positive sign and half have negative sign. Here the ÒsignÓ of an end means the sign of the corresponding boundary point of the index two moduli space. Thus there is no contribution when $y$ is bad.

If $y$ is good then by Theorem \ref{gluing}(i), the number of ends of the moduli space of index 2 cylinders converging to $(u, v)$, divided by their multiplicity, is $\dfrac{\emult(y)}{\emult(u)\emult(v)}.$ By \cite{BM}, each end has sign $\epsilon(u)\epsilon(v)$. This completes the proof of Theorem  \ref{gluing}(ii).
\end{remark}

Conditions (A) and (B) ensure that multiply covered curves cannot have smaller index than their underlying curves.  These allow us to give $\mhat_d(\gp;\gm)$ the structure of a smooth manifold for $d=0,1$, by Lemma 4.2(b) of \cite{HNdyn} which in these cases refines Wendl's automatic transversality result, Theorem \ref{folk0}, so that one obtains a manifold structure, rather than an orbifold structure.  This results in well-defined expressions for the differentials in (\ref{pa1}) and (\ref{pa2}).  These results are proven in Section \ref{arbitrary} by appealing to regularity results in \cite{HWZ3, HT2, Wtrans} and index calculations in dimension 3.  The proof of (D) is contained in Section \ref{numerology}, which relies on (A)-(C) along with additional index calculations.

\subsection{Dynamically separated contact forms}\label{dynsep}
By restricting ourselves to the following class of  \textbf{dynamically separated} contact forms in dimension 3, we will be able to obtain the requisite compactness and regularity results of Condition (D) in symplectizations.  In \cite{jocompute} we prove the analogous regularity and compactness results needed for obtaining invariance of cylindrical contact homology under choices of dynamically separated contact forms and $\A$-compatible $J$. 

The definition of dynamically separated necessitates that $c_1(\ker \A)=0$ so that a $\Z$-grading is available.  For there to be an absolute integral grading one must further require that $H^1(M;\Z)=0$.  We begin formulating the dynamically separated condition when all Reeb orbits are contractible, and then explain how to adapt this in the presence of noncontractible orbits.  

\begin{definition} \em
Let $(M,\A)$ be a nondegenerate 3-dimensional contact manifold with $c_1(\ker \A)=0$ such that all the Reeb orbits of $R_\A$ are contractible. Then $\A$ is said to be \textbf{dynamically separated} whenever the following conditions hold.
\begin{enumerate}
\item[{(I)  }]  If $\ga$ is a closed simple Reeb orbit then $ 3 \leq  \czm({\ga}) \leq 5$;
\item[{(II)}] If $\ga^k$ is the $k$-fold cover of a simple orbit $\ga$ then $\czm(\ga^k) = \czm(\ga^{k-1})+4.$
\end{enumerate}
\end{definition}

In order to define a dynamically separated contact form in the presence of noncontractible orbits, we must introduce the following notation to keep track of the free homotopy class of a Reeb orbit after each iteration of the underlying simple orbit.  This will be used to define an analogue of Condition II with respect to a free homotopy class $\baar \in \pi_0(\Omega M)$, as follows.

\begin{definition}\label{taut} \em
Let $(M,\A)$ be a 3-dimensional contact manifold with $c_1(\ker \A)=0$. Let $\gamma$ be a simple Reeb orbit. For each free homotopy class $\baar$, let 
\[
1 \leq k_1(\baar, \ga) < k_2(\baar, \ga) < ... <k_i(\baar, \ga) <...
\]
 be the (possibly empty or infinite) list of all integers such that all the $k_i(\baar,\ga)$-fold covers of $\ga$ lie in the same free homotopy class $\baar$.  We will use $\baar=0$ to represent the class of contractible orbits. A nondegenerate contact form $\A$ is said to be \textbf{dynamically separated} whenever the following conditions are satisfied. 
\begin{enumerate}
\item[{(I.i)  }]   For the class of contractible orbits, $\baar=0$, we have $ 3 \leq  \czm({\ga^{k_1(0, \ga)}}) \leq 5$;
\item[{(I.ii)}]   For each $\baar \neq 0$ there exists  $m(\baar, \ga) \in \Z_{>0}$ such that  $2m-1 \leq \czm(\ga^{k_1(\baar, \ga)}) \leq 2m+1$;
\item[{(II)  }]  For each free homotopy class $\baar$ we have $\czm(\ga^{k_{i+1}(\baar, \ga)})=\czm(\ga^{k_{i}(\baar, \ga)})+4.$
\end{enumerate}
\end{definition}

For computational methods it is often practical to consider contact forms which will be \textbf{dynamically separated up to (large) action}, which is proportional to the index.  This modification is explained in the following definition and we note that many Morse-Bott contact forms can be made dynamically separated up to large action by a small perturbation, allowing one to include the standard contact forms on $T^3$ and $S^3$.

\begin{definition}\em
A nondegenerate contact form $\A$ is said to be \textbf{dynamically separated up to action $A$} whenever the following conditions are satisfied. 
\begin{enumerate}
\item[{(I.i)  }]   For the class of contractible orbits, $\baar=0$, we have $ 3 \leq  \czm({\ga^{k_1(0, \ga)}}) \leq 5$ and 
\[
 \mathcal{A}(\gamma^{k_1(0, \ga)};\alpha):=\int_{\gamma^{k_1(0, \ga)}} \alpha < {A};
\]
\item[{(I.ii)}]   For each $\baar \neq 0$ there exists  $m(\baar, \ga) \in \Z_{>0}$ such that  $2m-1 \leq \czm(\ga^{k_1(\baar, \ga)}) \leq 2m+1$ and
\[
\mathcal{A}(\gamma^{k_1(\baar, \ga)};\alpha):=\int_{\gamma^{k_1(\baar, \ga)}} \alpha < {A};
\]
\item[{(II)  }]  For each free homotopy class $\baar$ we have $\czm(\ga^{k_{i+1}(\baar, \ga)})=\czm(\ga^{k_{i}(\baar, \ga)})+4,$ whenever 
\[
\mathcal{A}(\gamma^{k_{i+1}(\baar, \ga)};\alpha):=\int_{\gamma^{k_{i+1}(\baar, \ga)}} \alpha < {A}.
\]
\end{enumerate}

\end{definition}

\begin{remark}[Regularity up to action A] \em
If we assume that $\A$ is a nondegenerate dynamically separated contact form up to action ${A}$, then Theorem \ref{conditions2} ``holds up to action ${A}$."  Namely, if we denote 
\[
\mathscr{P}^{< {A}}(\alpha;c) :=\{\gamma \in \calpa \ \arrowvert \ \mathcal{A}(\gamma;\alpha):=\int_\gamma \alpha < {A} \},
\]
then condition (D) is reformulated as follows, assuming $J$ is a generic $\A$-compatible almost complex structure:
\begin{enumerate}
\item[ (D)] If $x,z \in \mathscr{P}^{< {A}}(\alpha;c)$ with $\czm(x)-\czm(z)=2$ then
\begin{equation}\label{compactnessaction}
\mbar_1(x; z) := \mhat_1(x; z) \  \ \cup \bigcup_{\substack{ y \in \mathscr{P}^{< \mathcal{A}}(\alpha;c) \\ \czm(y) = \czm(x)-1 }}  \mhat_0(x; y) \times \mhat_0(y; z).
\end{equation}
\end{enumerate}
Similarly, the gluing result of Theorem \ref{gluing} can be restated to hold up to action $A$.
\end{remark}

\medskip
Before giving some examples of dynamically separated contact forms, we make a remark about the relation between iteration properties of the Conley-Zehnder indices of closed Reeb orbits in dimension 3 and the choice of trivialization of $\xi$ along a Reeb orbit $\gamma$ and and it's iterate $\gamma^k$.

\begin{example}[Ellipsoids]\em
The dynamically separated condition is further explained in relation to irrational ellipsoids in Examples \ref{ellipsoid}-\ref{ellipsoidnoinvariance}.  Examples \ref{ellipsoid} and \ref{ellipsoidnoinvariance} are irrational ellipsoids which are not dynamically separated.  We explain in these cases how the heuristic proofs of $\pa^2=0$ and invariance fail in the absence of the analytic techniques of \cite{HNdyn, HN2, HN3}.   In Example \ref{ellipsoiddynsep} we show for specific values of $a$ and $b$ that we can make $E(a,b)$ irrationally separated up to large action.
\end{example}

Proofs of the following examples and computations will appear in \cite{jocompute}.

\begin{example}[Prequantization]\label{pre quant}\em
The contact 3-sphere $(S^3, \xi_{std}=\ker \lambda_0)$ can be realized as an example of a prequantization space via the Hopf fibration 
\[
\begin{array}{c}
\hopf \\
h(u,v)=(2u \bar{v}, |u|^2- |v|^2), \ (u,v) \in S^3 \subset \C^2 \\
\end{array}
\]
over the standard symplectic 2-sphere $(S^2, \omega_0)$.

This construction lends itself to a natural perturbation of $(S^3, \lo)$ and holds for any prequantization space.  It is comprised of adding a small lift of a Morse-Smale function on the base $(S^2, \om_0)$ to the original contact form 
\begin{equation}
\label{perturbedform1}
\lep=(1+\vepsilon h^*H)\lo.
\end{equation}
Since $(1+\vepsilon h^*H)>0$ for small $\vepsilon>0$, the contact structure remains unchanged as ker $\lep =$ ker $\lo = \xi_{std}$. The perturbed Reeb dynamics are given by
\begin{equation}
\label{perturbedreeb1}
R_{\vepsilon}=\frac{R}{1+\vepsilon h^*H} + \frac{\veps \tilde{X}_H}{{(1+\vepsilon h^*H)}^{2}}.
\end{equation}
Here $X_{H}$ is a Hamiltonian vector field\footnote{We use the convention $\omega(X_H, \cdot) = dH.$} on $S^2$ and $\tilde{X}_{ H}$ its horizontal lift,
\[
\mbox{e.g.} \ \ dh(q)\tilde{X}_{ H}(q) = X_{ H}(h(q)) \ \ \mbox{ and } \ \ \lo(\tilde{X}_{ H})=0.
\]

The only fibers that remain Reeb orbits of this perturbed contact form are iterates of fibers over the critical points $p$ of $H$.  For sufficiently small $\vepsilon$ the surviving $k$-fold covers of simple orbits in the fiber, denoted by $\gamma_p^k$, have action $A \sim 1/\vepsilon$, are non-degenerate, and satisfy
\begin{equation}
\label{czprequant}
\czm(\gpk)=4k-1+\mbox{\emph{index}}_p(H).
\end{equation}

However, we may obtain additional Reeb orbits that cover closed orbits of $X_{ H}$ but since $\vepsilon H$ and $ \vepsilon dH$ are small, these Reeb orbits all have periods much greater than $2 k \pi$ when $\veps \sim 1/k$.  In \cite{jothesis, jocompute} we demonstrate that a natural filtration on both the action and the index exists by letting $\vepsilon \to 0$. This gives rise to a formal version of filtered cylindrical contact homology to which the Reeb orbits covering $X_H$ do not contribute.

The proportionality between the action and index of the Reeb orbits, permits the use of direct limits to recover the full cylindrical contact homology from the truncated chain groups, consisting of
\[
 C_*^{<A_k}(S^3, \lambda_{\vepsilon_k}, H)=\{ \gamma_p^j \ | \ j \in [1, k] \mbox{ and } p \in \mbox{Crit}(H)\}.
\]  

 \begin{figure}
 \begin{center}
    \includegraphics[width=0.63\textwidth]{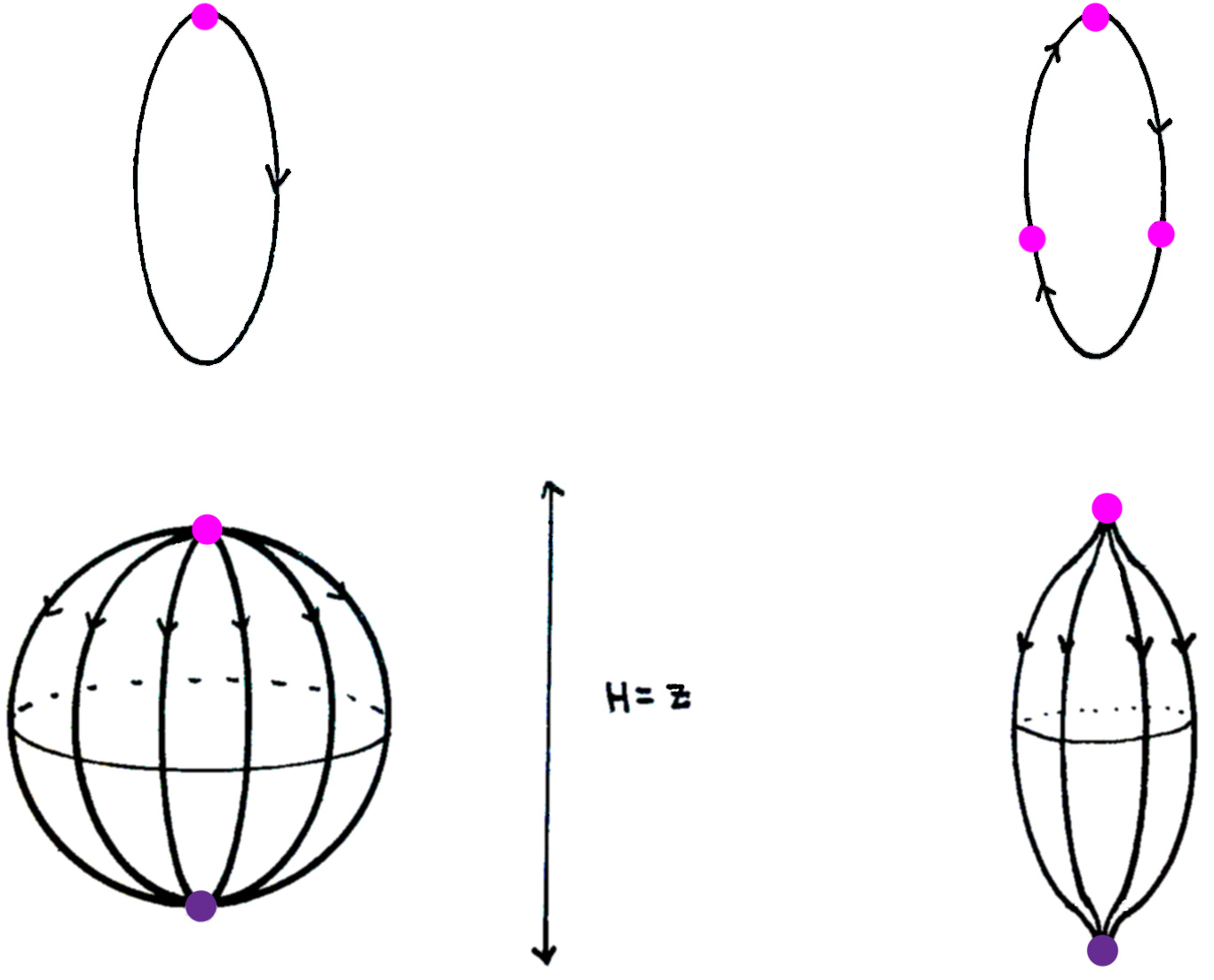}
\end{center}
  \caption{ $-\nabla H$ for $H=z$  with a fiber over $S^2$ and $S^2/\Z_3$ respectively.}
  \label{figlens}
\end{figure}

If we take $H=z$, the height function on $S^2$ as in Figure \ref{figlens} then we obtain a maximum at the north pole (index 2) and a minimum at the south pole (index 0).  Because the index increases by 4 under iteration then $\czm(\gpk)$ in (\ref{czprequant}) is always odd, so the differential vanishes, resulting in the following theorem.  We note that the statement of invariance will be justified by \cite{HN2, HN3, jocompute}.  

\begin{theorem}\label{HCsphere}
The cylindrical contact homology for the sphere $(S^{3},\xi_{std})$ is given by
\[
CH_*(S^{3}, \xi_{std}) = \left\{  \begin{array}{cl}
   \Q & *  \geq 2, \mbox{ even }  \\
    0 & * \ \mbox{ else } \\
\end{array} \right.
\]
\end{theorem}

The resulting differential for different choices of $H$ should behave analogously to the the Morse-Smale differential on the base. 
\end{example}

\begin{example}[Lens space]\em
If  $\pi_1(M)$ is abelian then the $k_i(\baar, \ga)$ form an arithmetic progression because 
\[
\pi_0(\Omega M) = \pi_1(M) / \{ \mbox{conjugacy} \} \cong \pi_1(M).
\]
This applies to the lens space $(L(n+1,n), \xi_{std})$, as each free homotopy class $\baar$ may be represented as an element of $\{ 0, 1, ... n\}$, where 0 represents a contractible class.  As a result, an arbitrary cover of a closed orbit may not be of the same free homotopy class $\baar$.  This will only be the case when the $k_\ell(\baar, \ga)$-th cover is given by 
\[
k_\ell(\baar, \ga) = \ell (n+1)+\baar,\mbox{ for } \baar \neq 0 \mbox{ and } \ell \in \Z_{\geq 0}.
\]
The procedure described in the previous example holds, though some care must be taken in regards to the fact that the base is now a symplectic orbifold.

We note that the Lens spaces $(L(n+1, n), \xi_{std})$ are contactomorphic to the links of the $A_n$ singularities $(L_{A_n},\xi_{A_n})$, with
\[
L_{A_n}:=\{\mathbf{z} \in \C^3 \ | \ z_0^{n+1} + z_1^2+z_2^2 =0 \}\cap S^{5} 
\]
and the canonical contact structure given by
\[
\xi_{A_n}:= T( L_{A_n} )\cap J_0 T( L_{A_n}).
\]
As $(L_{A_n},\xi_{A_n})$ is an example of a Brieskorn manifold, it is well known that $c_1(\xi_{A_n})=0$ \cite{vk}, thus $c_1(\xi_{L(n+1,n)})=0.$  The quotient of $S^3$ with the following cyclic subgroup of $SU(2,\C)$ yields the Lens space $L(n+1,n)$.  This cyclic subgroup is $\Z_{n+1}$, which acts on $\C^2$ by $u \mapsto \varepsilon u, \  v \mapsto \varepsilon^{-1}v$, where $\varepsilon = e^{2\pi i/(n+1)}$, a primitive $(n+1)$-th order root of unity.   The complex volume form $du \wedge dv $ on $\C^2$ can be used to compute the Conley-Zehnder indices associated to Reeb orbits of $S^3$ without local trivializations.  Since $\Z_{n+1} \subset SU(2;C)$, this means that the complex volume form $du \wedge dv $ descends to the quotient, allowing one to compute the Conley-Zehnder indices associated to Reeb orbits of $L(n+1,n)$.  This procedure yields the following formulas for the Conley-Zehnder indices and will be precisely described in the forthcoming paper \cite{jocompute}.

Let $\ga_p$ be the underlying simple orbit over a critical point $p$ of $H$.  For every $\ell \in  \Z_{> 0}$, we obtain a contractible orbit $\ga_p^{\ell(n+1)}$ of index
\begin{equation}
\label{cznoncontractible}
\czm(\ga_p^{\ell(n+1)})=4\ell- 1+\mbox{\emph{index}}_p(H)
\end{equation}

Otherwise for every $\ell \in  \Z_{\geq 0}$ we obtain a noncontractible Reeb orbit $\ga_p^{\ell(n+1)+c}$ in the free homotopy class $c \in \{1, 2, ..., n\}$ of index
\begin{equation}
\label{czcontractible}
\begin{array}{lcl}
\czm(\ga_p^{\ell(n+1)+c})&=&2+ 4 \left \lfloor \frac{\ell(n+1)+c}{n+1}  \right \rfloor - 1+\mbox{\emph{index}}_p(H) \\
&=& 2 + 4 \ell - 1 +\mbox{\emph{index}}_p(H), \\
\end{array}
\end{equation}

When using the height function as in Figure \ref{figlens} the differential vanishes in light of (\ref{cznoncontractible}) and (\ref{czcontractible}), yielding the following theorem.

\begin{theorem}\label{HClens}
The cylindrical contact homology for the lens space $(\lens, \xi_{std})$ is given by
\[
CH_*(\lens, \xi_{std}) = \left\{  \begin{array}{cl}
    \Q^n & *  =0   \\
    \Q^{n+1} & * \geq 2, \mbox{ even } \\
       0 & * \ \mbox{ else } \\
\end{array} \right.
\]
\end{theorem}

\end{example}

\subsection{Quandaries of the multiply covered}\label{quandaries}
In this section we explain the breaking configurations for index 2 cylinders, which are the obstructions to a well-defined chain complex.  If Conditions (A)-(D) are met, as in Theorems \ref{conditions} and \ref{conditions2}, then one can show that an index 2 cylinder in $\mhat_1(x;z)$ can only degenerate into a once broken cylinder, as in Figure \ref{limit}. Combined with the gluing results of Theorem \ref{gluing} one is then able to prove that $\pa_\pm^2=0$; see Section 4 of \cite{HNdyn}.
 \begin{figure}[h!]
  \centering
    \includegraphics[width=.25\textwidth]{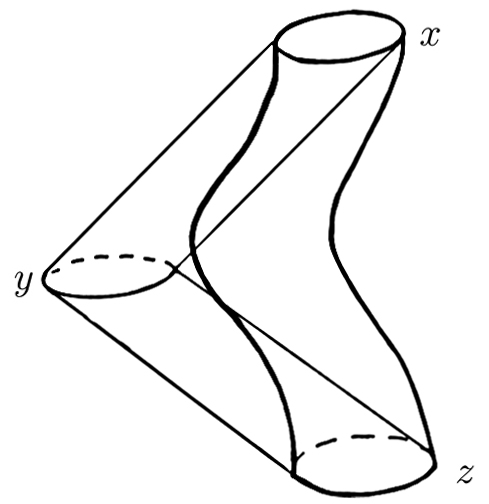}
      \caption{Desired limiting behavior for $u \in \mhat_1(x;z)$, where $y$ satisfies $\czm(y)=\czm(x)-1$. }
      \label{limit}
\end{figure} 

The obstruction to $\pa^2_\pm=0$ stems from the following compactification phenomenon in the symplectization.  The maximum principle, cf. Proposition \ref{max-principle}, implies that the real valued portion of an asymptotically cylindrical curve cannot develop any local maxima of.  However, the general compactness results in \cite{BEHWZ} allow for the real valued portion of an asymptotically cylindrical curve to develop a finite number of minima, as in Figure \ref{min}.   As a result, the compactification of moduli spaces of finite energy cylinders can include buildings\footnote{See Definition \ref{building} for the precise definition of a building.} of arbitrary height, consisting of genus 0 noncylindrical components; see Figure \ref{figbroken}.  In such situations  (\ref{compactness}) no longer holds and one cannot prove $\pa_\pm^2=0$.  Analogous phenomenon also occur in a cobordism, preventing the chain homotopy equations from holding. 

\begin{figure}[ht]
\begin{minipage}[b]{2.9in}
\centering
\includegraphics[width=.45\textwidth]{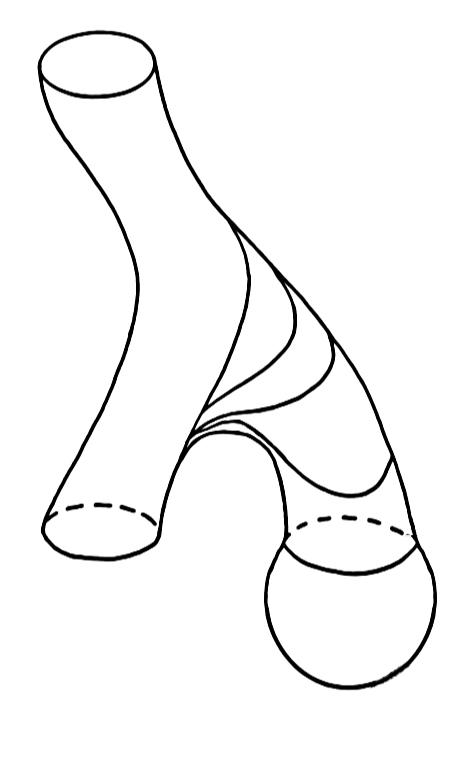}
\caption{Developing a minimum}
\label{min}
\end{minipage}
\begin{minipage}[b]{2.9in}
\centering
\includegraphics[width=.37\textwidth]{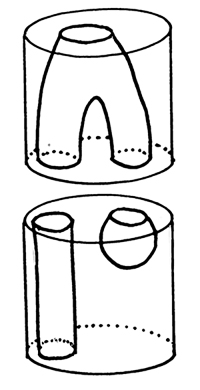}
\caption{A height 2 building, consisting of a pair of pants, plane, \& cylinder.}
\label{figbroken}
\end{minipage}
\end{figure}

Without abstract perturbations or other analytic techniques, the following discussion demonstrates that in order to construct a well-defined cylindrical contact chain complex\footnote{In \cite{EGH}, they only require the exclusion of contractible orbits of degree 1 in order to prove $\pa^2=0$ in Proposition 1.9.1. The exclusion of the contractible orbits of degree 0 and -1 is necessary to obtain the existence of a chain map and the homotopy of homotopies respectively when proving invariance.}, one must further require that all contractible orbits $\ga$ satisfy $|\ga| \geq 2$, e.g. by (\ref{sftgrading}), $\czm(\gamma) \geq 3$ in dimension 3.

\begin{remark} \em
Contractible planes of nonpositive index appear in the symplectization of Brieskorn manifolds of arbitrary dimension; see Remark 3.4 of \cite{vk}.
\end{remark}

In dimension 3, the virtual dimension of the moduli space $\calm := \calm(\gp; \ga_1,...\ga_s)$ is given by $\mbox{vdim }\calm(\gp; \ga_1,...\ga_s) = -(1-s) + \czm(\gp) - \sum_{i=1}^s \czm(\ga_i),$ cf. (\ref{indexformula}).  When regularity for $u \in \calm$ can be achieved then $\ind(u) \geq  1$, where $\ind(u)=\mbox{vdim}(\calm)$, cf. Corollary \ref{regJdim}.  However regularity results have only\footnote{While the result is expected to be true, there appear to be problems with some details of Dragnev's proof in \cite{Dr}, particularly regarding the claimed existence of certain rather special cut-off functions on page 757.  Rather than rely on \cite{Dr} directly, we appeal to  \cite{HWZ3, HT2, Wtrans}. It is not immediately clear whether or not the proof has a simple fix, however a complete proof should follow by modifying arguments in \cite{Btorus, Wtrans, Wnotes}.   A detailed proof of the desired result in the more general setting of stable Hamiltonian structures  appears in Wendl's blog:\\
{\scriptsize https://symplecticfieldtheorist.wordpress.com/2014/11/27/generic-transversality-in-symplectizations-part-1/} \\
{\scriptsize https://symplecticfieldtheorist.wordpress.com/2014/11/27/generic-transversality-in-symplectizations-part-2/} \\
{\scriptsize https://symplecticfieldtheorist.wordpress.com/2014/11/28/an-easy-proof-of-the-pi-du-lemma/} } been rigorously established in \cite{HWZ3}, after a generic choice of $\Jt$, for immersed somewhere injective curves associated to symplectizations of contact 3-manifolds and in limited other cases in \cite{Wtrans}.   

Next we pictorially explain why (D) fails if regularity cannot be achieved.  We work with the numerics of the virtual dimension of $\calm$ instead of $\mhat$ so as to not obscure the elementary properties of addition under consideration\footnote{Otherwise we must keep track of a $-1$ associated to each building component arising from the compactification $\mbar(x;z)$, which is annoying.}.  If $ \mbox{vdim }\calm(x;z) = \czm(x) - \czm(z) =2$, then the building $(u_1,...,u_n) \in\mbar(x;z)$ satisfies $\sum_i \ind(u_i) = 2.$

 \begin{figure}[h!]
  \centering
    \includegraphics[width=.6\textwidth]{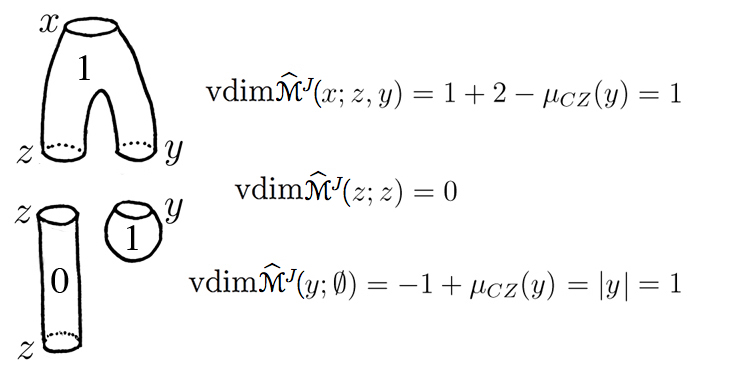}
      \caption{A contractible Reeb orbit $y$ with $\czm(y)=2$ precludes $\pa^2=0$. Here $u \in \calm(z;z)$ is the trivial cylinder and $\calm(y;\emptyset)$ consists of all finite energy planes bounding $y$. Even after a generic choice of $\Jt$, one cannot exclude planes of virtual dimension 1 from appearing.}
      \label{sad1}
\end{figure}

\vfill
\eject

\begin{figure}[h!]
\centering
    \includegraphics[width=.6\textwidth]{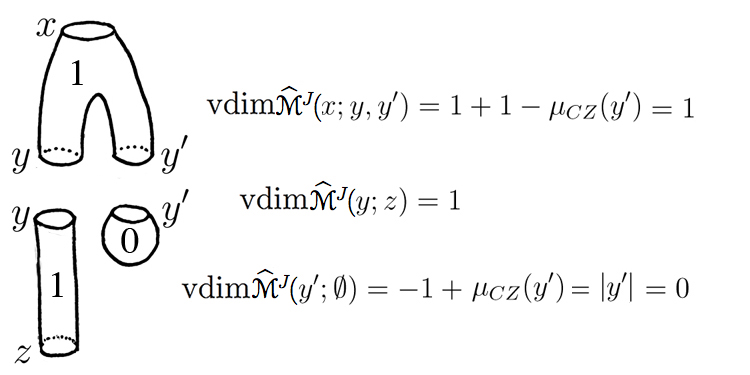}
      \caption{A contractible Reeb orbit $y'$ with $\czm(y')=1$ precludes $\pa^2=0$, as a generic choice of $\Jt$ does not exclude planes of virtual dimension 0 from appearing if $y'$ is multiply covered.}
      \label{sad0}
\end{figure} 

 \begin{figure}[h!]
 \centering
    \includegraphics[width=.6\textwidth]{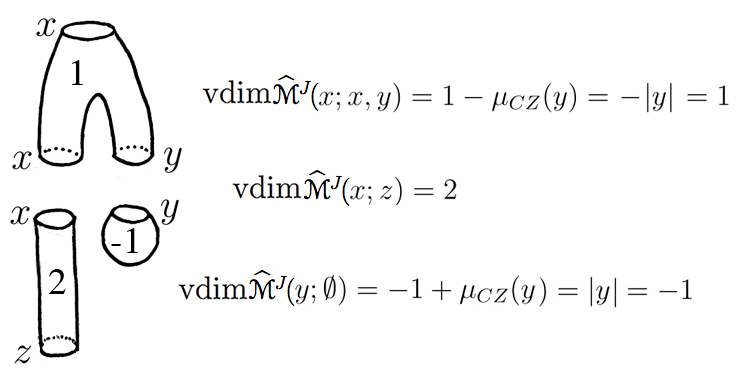}
      \caption{A contractible Reeb orbit $y$ with $\czm(y)=0$ precludes $\pa^2=0$, as a generic choice of $\Jt$ does not exclude planes of virtual dimension -1 from appearing if $y$ is multiply covered.}
      \label{sad-1}
\end{figure}

In addition to the phenomena seen in Figures \ref{sad1}-\ref{sad-1}, there is typically a ``failure of compactness'' whenever the index of $u_i$ is nonpositive, as it is still possible to obtain $\sum_i\mbox{ind }u_i =2$ even when all contractible Reeb orbits have Conley-Zehnder index $\geq3$.  This is demonstrated in the following examples involving the irrational ellipsoid, further underscoring the necessity of Conditions (A)-(D).

\begin{example}[Ellipsoid] \label{ellipsoid} \em
The 3-dimensional ellipsoid is given by $E(a,b):=f^{-1}(1)$,  
\begin{equation}\label{ellipsoideq}
\begin{array}{cccl}
f: &\C^2 &\to & \R \\
& (u, v) & \mapsto & \dfrac{|u|^2}{a} + \dfrac{|v|^2}{b} \\
\end{array}
\end{equation}
and $a,b \in \R_{>0}$.  The standard contact structure for the ellipsoid is
\[
\xi_p=T_pE \cap J_0 (T_pE),
\]
which is the kernel of the 1-form 
\[
\A = - \frac{1}{2} df \circ J_0.
\]
The Reeb vector field associated to $\A$ is
\[
R_\A=\frac{1}{a} \left( u \dfrac{\pa}{\pa u} -  \bar{u} \dfrac{\pa}{\pa \bar{u}} \right) + \dfrac{1}{b} \left( v \dfrac{\pa}{\pa v} -  \bar{v} \dfrac{\pa}{\pa \bar{v}}\right).  
\]
This vector field rotates the $u$-plane at angular speed $\frac{1}{a}$ and the $v$-plane at angular speed $\frac{1}{b}$. If $a/b$ is irrational, there are only two nondegenerate simple Reeb orbits living in the $u=0$ and $v=0$ planes. We denote these by $\ga_1$ and $\ga_2$ respectively.

Their Conley-Zehnder indices are described by
\begin{equation}\label{cz-ellipsoid}
\czm(\ga_i^k) = 2 \lfloor k(1+ \phi_i) \rfloor + 1,
\end{equation}
where $\phi_1=a/b$ and $\phi_2 = b/a$, see \cite{Ylong}.  Note that the Conley-Zehnder indices (\ref{cz-ellipsoid}) never coincide and span all the odd positive numbers as $k \to \infty$.

If we consider the ellipsoid such that $0<a<\frac{1}{2}b$ then $\phi_1< 1/2$ and
\[
\begin{array}{lcl}
\czm(\ga_1) &=& 3 \\
 \czm(\ga_1^{2})&=&5. \\
\end{array}
\]
This means for any $u \in \calm(\ga_1^{2} ; \ga_1, \ga_1)$,  
\[
\ind(u) = 0,
\] 
so the virtual dimension of $\calm(\ga_1^{2} ; \ga_1 , \ga_1)$ is 0.   However this moduli space is never nonempty, since it contains the double branched covers of the trivial cylinder over $\ga_1$, which forms a 2-dimensional family. As a result transversality can not be achieved for this moduli space using ``standard'' techniques of perturbing $J$.  However in Section 3 of \cite{HNdyn} we are able to exclude a building consisting of such an element of $\calm(\ga_1^{2} ; \ga_1 , \ga_1)$ by making use of intersection theory, which allows us to still conclude $\pa^2=0$. \end{example}

We note that certain irrational ellipsoids will be dynamically separated up to a particular action level, as in the following example. However, in order to compute the full cylindrical contact homology from the truncated chain complex one will need methods as in \cite{jothesis, jocompute}, as outlined in Example \ref{pre quant}.
\begin{example}\label{ellipsoiddynsep}\em
If we choose $a,b$ rationally independent such that $\frac{k-1}{k}<\phi_1<1$ then $ 1<\phi_2<1+\frac{1}{k-1}$ and
\[
\begin{array}{lcl}
\czm(\ga_1^k) &=& 4k-1 \\
 \czm(\ga_2^{k})&=&4k+3. \\
\end{array}
\]
As a result such an ellipsoid would be dynamically separated up to the $k$-th iterate of $\ga_1$ and $\ga_2$, which would be up to an action level proportional to $k$.  For example if we take $a=1$, $b=1+\vepsilon$ with an irrational $\vepsilon < 1/(2\cdot 99)$ then $E(1, 1+\vepsilon)$ is dynamically separated up to the $k=100$ iterate.  
\end{example}
The next example demonstrates the obstruction to invariance caused by \textbf{badly iterated cylinders} in a cobordism.  Badly iterated cylinders are multiply covered curves which have index smaller than their underlying cylinder, and cannot be excluded except when working with dynamically separated contact forms.  Details pertaining to invariance with respect to the dynamically separated condition will be given in \cite{jocompute}. 
 
 \begin{example} \label{ellipsoidnoinvariance}\em
Again we take $E(a,b):=f^{-1}(1)$, where 
\begin{equation}\label{ellipsoideq}
\begin{array}{cccl}
f: &\C^2 &\to & \R \\
& (u, v) & \mapsto & \dfrac{|u|^2}{a} + \dfrac{|v|^2}{b} \\
\end{array}
\end{equation}
and $a,b \in \R_{>0}$.  We will consider the following two ellipsoids $E_+:=E(1-\vepsilon, 2+\veps)$ and $E_-:=E(1-\vepsilon, 1+\veps)$, with $f_\pm$ the respective defining plurisubharmonic functions as in (\ref{ellipsoideq}).   Here $\vepsilon$ will be chosen so that $0< \veps << 1$ and in addition both $\frac{1-\vepsilon}{2+\veps}$ and $\frac{1-\vepsilon}{1+\veps}$ are rationally independent.   We equip these ellipsoids with their canonical contact forms, denoted $\A_+$ and $\A_-$, given by
\[
\A_\pm:=-df_\pm \circ J_0,
\]
where $J_0$ is the standard complex structure on $\C^2$.  Moroever the contact forms $\A_+$ and $\A_-$ are nondegenerate provided $\veps$ has been chosen as above. Note that $(E_+, \A_+)$ and $(E_-, \A_-)$ are contactomorphic by Gray's Stability theorem\footnote{The smooth one parameter family of diffeomorphic contact manifolds can be obtained via the flow of the Liouville vector field $u \frac{\pa}{\pa u} + v \frac{\pa}{\pa v}$ on $\C^2 \setminus \{ 0 \}$, which has been appropriately reparametrized so that at time 0 one starts from $E_-$ and lands on $E_+$ at time 1.}.  

As before, each of the Reeb vector fields $R_{\A_\pm}$ admit exactly two nondegenerate simple Reeb orbits living in the $u=0$ and $v=0$ planes.  Denote the two orbits associated to $R_{\A_+}$ by $\delta_1$ and $\delta_2$ and the two orbits associated to $R_{\A_-}$ by $\gamma_1$ and $\gamma_2$. From (\ref{cz-ellipsoid}) we can obtain the following grading on our Reeb orbits.  
For the Reeb orbits $\delta_1$ and $\delta_2$ associated to $E_+:=E(1-\vepsilon, 2+\veps)$ we have
\[
\begin{array}{ccccc}
|\delta_1| = 2, & |\delta_1^2| = 4, & |\delta_2| = 6, & |\delta_1^3| = 8, & ...  \\
\end{array}
\]
For the Reeb orbits $\ga_1$ and $\ga_2$ associated to $E_-:=E(1-\vepsilon, 1+\veps)$ we have
\[
\begin{array}{ccccc}
|\ga_1| = 2, & |\ga_2| = 4, & |\ga_1^2| = 6, & |\ga_2^2|=8, & ...  \\
\end{array}
\]

Consider the \textbf{strong cylindrical cobordism} $(X, \omega)$ defined from $(E_1, \A_1)$ to $(E_2, \A_2)$, which is defined to be a compact symplectic 4 manifold, which is not necessarily exact, with oriented boundary
\[
\partial X = E_- - E_+
\]
such that
\[
\omega\arrowvert_{E_+} = d\A_+ \mbox{ and } \omega\arrowvert_{E_-} = d\A_-.
\]
One may \textbf{complete} $(X, \omega)$ by attaching infinite cones at either end, namely
\[
W:=V_- \ \cup_{E_-} X \ \cup_{E_+}  V_+,
\]
where $(V_-,J) \cong ((-\infty, 0] \times E_-, \Jt_-)$ and   $(V_+,J) \cong ([0,\infty) \times E_+, \Jt_+)$.
The 4-manifold $W$ admits a global symplectic form and a compatible almost complex structure which overlaps with those on $(X, \omega)$ and the conical ends, as it is an example of a stable Hamiltonian structure. We consider finite energy cylinders 
\[
u: (\R \times S^1, j_0) \to (W, \Jt)
\]
limiting on a nondegenerate Reeb orbits of $(E_+, \A_+)$ at $+\infty$ and on a nondegenerate Reeb orbits of $(E_-, \A_1)$ at $-\infty$. 

Supposing the compactness issue in Example \ref{ellipsoid} can be overcome as in \cite{HNdyn}, since the grading on the Reeb orbits is always even for both $(E_+, \A_+)$ and $(E_-, \A_-)$ one might heuristically expect that the cylindrical contact homology of each is equal to its chain complex, by index reasons.  As the two contact manifolds are contactomorphic, one would expect $HC_*(E_+, \A_+, J_+) \cong HC_*(E_-, \A_-, J_-)$. 

However, the usual proof of invariance involving the construction of a chain homotopy fails because of the following compactness issue. We have
\begin{equation}\label{czinv}
\begin{array}{lr}
\mbox{only $\delta_1$ and $\ga_1$ satisfy} & |\delta_1|=|\ga_1|=2; \\
 \mbox{only $\delta_2$ and $\ga_1^2$ satisfy} & |\delta_2|=|\ga_1^2|=6.\\ 
\end{array}  
\end{equation}
For a curve $u \in \mhat_0(\delta_1; \ga_1)$ we can consider its 2-fold unbranched cover $v \in \mhat_0( \delta_1^2;\ga_1^2)$, which has ${\mbox{ind}}(v)=-2$.  As in Example \ref{ellipsoid}, the moduli space $\mhat_0( \delta_1^2;\ga_1^2)$ is never nonempty since it contains the double unbranched covers of $u$.

 But now the elements $v \in \mhat_0( \delta_1^2;\ga_1^2)$  cannot be excluded from appearing in compactification of $\mhat_0(\delta_2; \ga_1^2)$ because the curve $u\in \mhat_0(\delta_2;\ga_1^2)$ can break along $\delta_1^2$ in the upper portion of the cobordism $W$, the $(\R^+ \times E_+)$-component, as illustrated below in Figure \ref{ellipsoidsucks}.  Such a breaking obstructs the usual chain homotopy scheme for proving invariance.
 
 Even more curious is that the relative adjunction formula, cf. Section 4.4 \cite{Hu2}, implies that there can be no cylinder $u \in \mhat_0(\delta_2;\ga_1^2)$, hence the contribution to the chain map going from $\delta_2$ to $\ga_1^2$ must include a broken curve including the index -2 cylinder $v$. 
 
To see why this calculation holds, we will need several of the results pertaining to the asymptotics of pseudoholomorphic curves, found in Section 3.1 of \cite{HNdyn} and Section 4.4 of \cite{Hu2}.  We will only state what is necessary to prove that $\mhat_0(\delta_2;\ga_1^2)$ is empty, and refer the inquisitive reader to the detailed explanations in the above references.   

Let $\gamma$ be an embedded Reeb orbit and $N$ be a tubular neighborhood of $\gamma$. We can identify $N$ with a disk bundle in the normal bundle to $N$, and also with $\xi|_\gamma$. Let $\zeta$ be a braid in $N$; this is defined to be a link in $N$ such that that the tubular neighborhood projection restricts to a submersion $\zeta \to \gamma$.  A trivialization $\Phi$ of $\xi|_\ga$ gives rise to the notion of writhe, $w_\Phi(\zeta) \in \Z.$ The writhe is computed by using the trivialization $\Phi$ to identify $N$ with $S^1 \times D^2$ and then projecting $\zeta$ to an annulus and counting crossings of the projection with (nonstandard) signs. See Section 2.6 of \cite{Hu2} and  Section 3.3 of \cite{Hu} for further details.

Let $u$ be a pseudoholomorphic curve in $(W,J)$ with a positive end at $\gamma^d$ which is not part of a multiply covered component. Corollaries 2.5 and 2.6 of \cite{Sief} show that if $R$ is sufficiently large, then the intersection of the positive end of $u$ with $\{R\} \times N \subset \{R\} \times M$ is a braid $\zeta$, whose isotopy class is independent of $R \in \R^+$.  The same process holds symmetrically for a negative end of $u$.  

The proof of the relative adjunction formula in \cite{Hu2} yields the following result, analogous to Lemma 3.5 of \cite{HNdyn} for  $u \in \mhat_0(\delta_2;\ga_1^2)$,
 \begin{equation}\label{adjunctionfun}
w_\Phi(\zeta_+)-w_\Phi(\zeta_-) = 2\Delta(u) \geq 0
\end{equation}
where $\zeta_+$ is the braid obtained by the intersection of the positive end of $u$ with $\{R+\}\times N^+$, $\zeta_-$ is the braid obtained by the intersection of the negative end of $u$ with $\{R+\}\times N^-$, $w_\Phi(\zeta_\pm)$ is the writhe of $\zeta_\pm$, and $\Delta(u)$ is a count of the singularities of $u$ with positive integer weights in $(W,J)$.  Lemma 4.16 of \cite{Hu2} yields
\[
w_\Phi(\zeta_+) \leq 0,
\]
while Lemma 3.4(b,d) of \cite{HNdyn} yields
\[
w_\Phi(\zeta_-) \geq \mbox{wind}(\zeta_-) \geq \lceil \czm(\ga_1^2)/2\rceil= 3.
\]
Thus
\[
w_\Phi(\zeta_+)-w_\Phi(\zeta_-) \leq -3
\]
contradicting (\ref{adjunctionfun}).

 \begin{figure}[h!]
  \centering
    \includegraphics[width=0.5\linewidth]{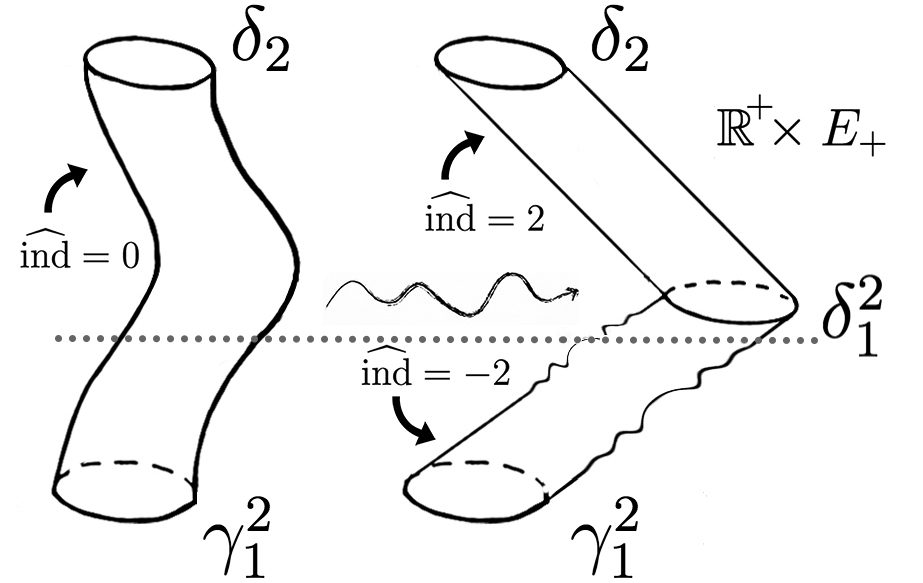}
      \caption{A failure of compactness in a cobordism, obstructing the usual proof of invariance. }
      \label{ellipsoidsucks}
\end{figure} 

 \end{example}

\noindent \textbf{Acknowledgements}. I thank Mohammed Abouzaid, Kai Cieliebak, Helmut Hofer, Michael Hutchings, Dusa McDuff, Joel Fish, Alex Oancea, Katrin Wehrheim, and Chris Wendl for their interest in my work and our assorted insightful discussions.  I would especially like to thank Mohammed Abouzaid for being a wonderful and generous advisor as well as his comments on my thesis, from which this paper has been extracted. I am very grateful to Michael Hutchings for showing me example \ref{ellipsoidnoinvariance} and our discussions of index calculations which gave rise to many of the results in Section \ref{regularity}; these generalized those that appeared in my thesis and spawned our joint projects \cite{HNdyn, HN2, HN3}. Special thanks are due to Michael Hutchings, Janko Latschev, Dusa McDuff,  Andrew McInerney, and the referee for their helpful comments on this paper.   \\

\section{The letter  $J$ is for pseudoholomorphic}\label{freddie}
 In the world of contact homology we will primarily consider pseudoholomorphic curves interpolating between closed nondegenerate Reeb orbits in the symplectization or in a strong symplectic cobordism of a contact manifold.  These curves were first used by Hofer \cite{H1} in this context to prove the Weinstein conjecture for $S^3$.   As many of the results and definitions appear scattered across literature, this section is meant as a survey of the properties of pseudoholomorphic curves appearing in the study of contact homology and may be skipped by the expert.

We begin by precisely stating some basic notions from contact geometry.

\subsection{Pseudoholomorphic curves in symplectizations}\label{pseudsymp}
 Let $(M, \A = \ker \xi)$ to be a contact manifold.  The \textbf{symplectization} of $(M,\A)$ is given by the manifold  $\R \times M$ and symplectic form
\[
\omega = e^\tau(d\A - \A \wedge d\tau) = d (e^\tau\alpha).
\]
Here $\tau$ is the coordinate on $\R$, and it should be noted that $\A$ is interpreted as a 1-form on $\R \times M$, as we identify $\A$ with its pullback under the projection $\R \times M \to M$.  

The other required component of any \curve \ theory is  the notion of an almost complex structure. Recall that any contact structure $\xi$ may be equipped with a complex structure $\bar{J}$ such that $(\xi, \bar{J})$ is a complex vector bundle.

We denote the set of \textbf{compatible almost complex structures} on $\xi$ by
 \[
 \J=\{ \bar{J} : \xi \to \xi \ | \ \bar{J}^2 = -\id, \ d\A(\bar{J}\cdot,\bar{J}\cdot)=d\A(\cdot,\cdot), \ d\A(\cdot, \bar{J}\cdot) > 0 \}.
 \]
 This set is nonempty and contractible, as in the symplectic case which is discussed in \cite{MS1}.  Thus $(\xi, d\alpha, \bar{J})$ is a \textbf{symplectic vector bundle},\footnote{Isomorphism classes of symplectic vector bundles are in a 1-1 correspondence with complex vector bundles.  As a result $(\xi, \bar{J})$ is frequently said to be a \emph{complex vector bundle}, and one suppresses the `almost' in almost complex structure despite the fact that we do not require elements of $\J$ to be integrable.} which admits a Hermitian structure.   There is a unique canonical extension of the almost complex structure $\bar{J}$ on $\xi$ to an $\R$-invariant almost complex structure $\Jt$ on $T(\R \times M)$, whose existence is due to the splitting,
\begin{equation}
\label{decomp}
T(\R \times M) = \R \ptau \oplus \R R_{\A} \oplus \xi.
\end{equation}

\begin{definition}[Canonical extension of $\bar{J}$ to $\Jt$ on $T(\R \times M)$]\label{complexstruc}\em
Let $[a,b;v]$ be a tangent vector where $a, \ b \in \R$ and $v \in \xi$.  We can extend $\bar{J}: \xi \to \xi$ to $\Jt: T(\R \times M) \to T(\R \times M)$ by
\[
\Jt[a,b;v] = [-b,a,\bar{J}v].
\]
Thus $\Jt|_\xi = \bar{J}$ and $\Jt$ acts on $\R \ptau \oplus \R R_{\A}$ in the same manner as multiplication by $i$ acts on $\C$, namely $\Jt \ptau = R_{\A}$. If $\bar{J}$ is compatible with $d\A$ on $\xi$ then $\Jt$ is said to be \textbf{$\A$-compatible}.  \end{definition}

\begin{remark}\em
Note that all $\A$-compatible $\Jt$ are invariant under the external $\R$-action on the symplectization $(\R \times M, d(e^\tau \A))$, and compatible with $d(e^\tau \A)$ by construction.
\end{remark}

Let $(\Sigma, j)$ be a closed Riemann surface and $\g:=\{x, y_1,...y_s\} \subset \Sigma$ be a set of points, which are the punctures of $\ds:=\Sigma \setminus \g$. We denote by $u:=(a,f): (\ds, j) \to (\R \times M, \Jt)$ a \murve \ in the symplectization of $(M,\A)$.  
 For nondegenerate closed Reeb orbits $\gamma,\gamma_1,...\gamma_s$ of periods $T, T_1,...,T_s$, as in the introduction we denote $\M(\gamma;\gamma_1,...\gamma_s) $ to be the \textbf{moduli space of genus 0 asymptotically cylindrical \curves}, with one positive puncture and $s$ negative punctures.  An asymptotically cylindrical pseudoholomorphic curve is an equivalence class of asymptotically cylindrical pseudoholomorphic maps as defined in equations (\ref{ass1}) and (\ref{ass2}).  The equivalence relation is given as follows.   

Denote the finite set of punctures by $\g := \{ x, y_1,...,y_s \}$ which we assume has been ordered into positive and negative punctures, $\g_+:=\{x\}$ and $\g^-:=\{y_1,..y_s\}$ respectively.  An equivalence class $(\Sigma, j, \g, u) \sim (\Sigma', j', \g', u')$ of asymptotically cylindrical pseudoholomorphic maps, $[(\Sigma, j, \g, u)]$, is determined whenever there exists a biholomorphism $\phi: (\Sigma, j) \to (\Sigma', j')$ taking $\g$ to $\g'$ with the ordering preserved, i.e. $\phi(\g_+) =\g'_+$ and $\phi(\g_-) =\g'_-$, such that $u = u' \circ \phi.$
Since $u$ determines $\Sigma$ and $\Gamma$ uniquely we may use notation $u:=(a,f) \in \calm(\gamma;\gamma_1,...\gamma_s)$ to refer to an asymptotically cylindrical pseudoholomorphic curve.

Since $\Jt$ is $\R$-invariant, $\R$ acts on these moduli spaces by \textbf{external translations}
\[
u=(a,f) \to (a + \rho, f),
\]
and we denote the quotient by $\mhat(\gamma;\gamma_1,...\gamma_s):=\calm(\gamma;\gamma_1,...\gamma_s)/\R.$  Before we can define the area and the energy of a pseudoholomorphic curve, we need some results pertaining to the local behavior of solutions to the Cauchy-Riemann equations.

\subsection{Local behavior}\label{local}
The use of local coordinates provides a proof of a maximum principle as well as some other helpful identities.   Let $X$ be a local nowhere vanishing vector field on $\ds$ and define $Y= J \circ X$, yielding a local frame.  Because $\ds$ is a closed Riemannian surface with at least one point removed, we can always find a Hermitian trivialization of $T\ds$.  As a result we may think in terms of a global frame $\{X, Y\}$. 

\begin{remark}\label{global} \em
While there exists a global Hermitizan trivialization of $T\ds$ ensuring that $X$ and $Y$ are globally non-vanishing, we cannot a priori conclude that they come from a coordinate system on $\C$. However, $\{X, Y\}$ is a global frame so the dual of the polyvector field $X \we Y$ yields the global 2-form,
\[
\Omega_{\{X,Y\}} = (X \we Y)^*.
\]
\end{remark}

For the purposes of this section it is preferable to work locally with coordinate $s+it \in \C$ by associating $X$ with $\pas$ and $Y$ with $\pat$.  However in later sections, the above remark will allow us to work with a general global frame.

Recall the projection $\pi$ of the tangent bundle of $M$ along the Reeb vector field, $
\pi: TM \to \xi.$  Then 
\[
u_s:= Du \circ \pas : \Sigma \to T(\R \times M)  \cong \R \frac{\pa}{\pa \tau} \oplus \R R_{\A} \oplus \xi \\
\]
can be written as
\[
u_s(z) = [ a_s(z), \A (f_s(z)); \pi ( f_s(z)) ],
\]
where $u = (a,f)$ and $f_s=Df \circ \pas$.  Similarly we have
\[
u_t(z) := Du \circ \pat = [a_t, \A (f_t); \pi( f_t)].
\]
Locally we have that $u=(a,f)$ is pseudoholomorphic if and only if
\begin{equation}\label{loc1}
u_s + \Jt u_t = 0,
\end{equation}
which is equivalent to
\begin{equation}\label{loc2}
  \left\{
  \begin{array}{lcl}
  a_s& = & \A( f_t), \\
 a_t &=&-\A (f_s) \\
 \pi (f_t) &=& J \pi (f_s) \\
   \end{array} \right. 
\end{equation}

\begin{proposition}[Maximum principle]\label{max-principle} If the real valued portion $a$ of a \murve \ $u:=(a,f)$ assumes a local maximum in the interior of $\ds$ then $u$ is the constant map.
\end{proposition}

\begin{proof}
The first two equations of (\ref{loc2}) can be written as 
\begin{equation}\label{loc3}
f^*\A =  -da \circ j = *d\A, \ \mbox{where $*$ denotes the Hodge star operator}.
\end{equation}
The third equation of (\ref{loc2}) means that the map $\pi \circ df: T\Sigma \to \xi$ is complex linear on each fiber, hence $\pi \circ Df(z)$ is either zero or an isomorphism. Differentiating (\ref{loc2}) yields
\begin{equation}\label{loc3}
f^* d\A = - d( da \circ j) = \Delta a \ ds \we dt, \ \ \ \mbox{where } \Delta a = a_{ss} + a_{tt}.
\end{equation}
Therefore,
\[
\Delta a = f^* d\A\left( \pas, \pat \right ) = d\A(f_s, f_t) = d\A(\pi f_s, J \pi f_s) = |\pi u_s|^2 = |\pi u_t|^2,
\]
as $d\A(\cdot, J \cdot)$ defines a metric on $\xi$. Hence the function $a$ is subharmonic.   As a result of the strong maximum principle, see for example \cite{HK}, we obtain the desired maximum principle, applicable to \murves \ in symplectizations.
\end{proof}

\subsection{Energy and area}\label{area-energy}
The quantities referred to as the \textbf{area} and (Hofer) \textbf{energy} of a \murve \ were introduced in \cite{HWZ1, HWZsmallarea}.  The first serves as a substitute for the notion of area for pseudoholomorphic maps in  closed symplectic manifolds. The finiteness of the latter provides a relationship between the asymptotic behavior of the \murve \ and the Reeb dynamics of a (nondegenerate) contact manifold.  They are defined as follows.  

\begin{definition} \em
The \textbf{area} of the \murve \ $u$ is given by the formula
\begin{equation}\label{area}
A(u) := \int_{\ds} u^*d\A = \int_{\ds} f^*d\A.
\end{equation}
\end{definition}
\begin{remark}\em
In some literature the area is called the $\om$-energy or $d\A$-energy where $\om=d(e^\tau\A)$ is the symplectic form on the symplectization of $(M, \A)$.  In early literature on contact homology this was referred to simply as energy.  We will not use these conventions and refer to this as the area and denote it by $A(u)$.   
\end{remark}

Based on the discussion of the local behavior in the preceding section we have the following result regarding the non-negativity of area.

\begin{proposition}[Non-negativity of area]\label{posarea}
For any finite area \murve \ $u$ we have
\[
A(u) := \int_{\ds} f^*d\A \geq 0
\]
\end{proposition}

\begin{proof}
The local computations (\ref{loc3}) allow us to write
\[
f^*d\A \geq 0 = |\pi u_s|^2 ds \we dt = |\pi u_t|^2 ds \we dt,
\]
where $\pi:TM \to \xi$ is the projection of the tangent bundle of $M$ along the Reeb vector field. The same process works in global coordinates by making use of the global frame $\{X, Y\}$ as discussed in Remark \ref{global}.  
\end{proof}

Combined with the local computations of (\ref{loc2}) we obtain the following corollary.

\begin{corollary}\label{noarea}
If a pseudoholomorphic cylinder $u:=(a,f)$ has $A(u)=0$ then the image $f(\ds)$ is contained in a trajectory of the Reeb vector field $R_{\A}$. 
\end{corollary}

\begin{remark}\label{noarea1}\em
 $A(u)=0$ is equivalent to $\pi u_X =\pi u_Y = 0.$
\end{remark}

To define the \textbf{Hofer energy} of a \murve \ $u$, denoted by $E(u)$, we need to introduce a class of smooth maps, which will be used extend the contact form $\A$ on $M$ to a 1-form on $\R \times M$.    Let
\begin{equation}\label{extend1form}
 \mathscr{S} = \{ \phi \in \Ci(\R, [0,1]) \ | \ \phi' \geq 0\}
\end{equation}
and define for $\phi \in  \mathscr{S}$ the 1-form $\A_\phi$ on  $\R \times M$ by:
\[
\A_\phi(\tau, p)(\rho, v) := \phi(\tau)\A_p(v) \mbox{ for } (\rho, v) \in T_{(\tau, p)}(\R \times M).
\]

\begin{definition} \label{hofer-energie} \em 
The \textbf{Hofer energy} of $u$ is given by
\[
E(u):=\sup_{\phi \in \mathscr{S}} \int_{\dot{\Sigma}} u^*d\A_\phi.
\]
\end{definition}
The Hofer energy of a \curve \ is also referred to as \textbf{$\A$-energy}.  In this paper we will refer to it simply as \textbf{energy}, which is also typical. 

\begin{proposition}[Non-negativity of energy]\label{posarea}
For any \murve \ $u$,
\[
E(u) := \sup_{\phi \in \mathscr{S}} \int_{\dot{\Sigma}} u^*d\A_\phi \geq 0.
\]
\end{proposition}

\begin{proof}
One may compute the integrand in light of the local computations of Section \ref{local} with respect to the local frame $\{\pa_s, \pa_t\}$,
\[
u^*d\A_\phi = (\phi'(a)|\nabla a|^2 + \phi(a) \Delta a) ds \we dt =  (\phi'(a)|\nabla a|^2 + \phi(\tau)|\pi f_s|^2 ) ds \we dt.
\]
As discussed previously in this section, we can convert the above into the following non-negative globally defined expression,  
\begin{equation}\label{zeroenergie}
 u^*d\A_\phi =  (\phi'(a)|\nabla a|^2 + \phi(\tau)|\pi f_X|^2 ) \Omega_{\{X,Y\}} \geq 0.
\end{equation}
\end{proof}

\begin{remark}\label{E0const} \em
As a result of the above expression for the integrand (\ref{zeroenergie}), we see that pseudoholomorphic maps $u$ with $E(u)>0$ are necessarily non-constant.  Moreover, $E(u) =0$ if and only if $u$ is constant.
\end{remark}

Stokes' theorem shows any \murve \ defined on a closed Riemann surface is constant.

\begin{proposition}\label{constant}
A \murve \ $u$ defined on a closed Riemann surface $(\Sigma, J)$ into the symplectization of a contact manifold $(\R \times M, d(e^\tau\A), \Jt)$ is  constant.
\end{proposition}

\begin{proof}
Stokes' theorem yields
\[
 \int_\Sigma u^*d\A_\phi = \int_{\pa \Sigma}u^*\A_\phi = 0.
\]    
for all $\phi \in \mathscr{S}$ hence $E(u)=0$.
\end{proof}

A pseudoholomorphic map $u$ is said to be a \textbf{finite energy map\footnote{In earlier literature these were sometimes referred to as finite energy parametrized surfaces.}} whenever 
\[
0 < E(u) < \infty.
\]
For any \murves \ $u, u'$ in the same equivalence class, we have
\[
E(u) = E(u') \ \ \ \mbox{and} \ \ \ A(u)= A(u').
\]
Thus it is common practice to refer to the energy or area of a curve.

The following examples will illuminate the different controls that area and energy have on a pseudoholomorphic map.

\begin{example}\label{energyex} \em
Define a pseudoholomorphic cylinder over a periodic orbit of the Reeb vector field by
\[
\begin{array}{crcl}
v:& (\R \times S^1, j) &\to& (\R \times M, \Jt) \\
&(s, e^{2\pi i t}) &\mapsto& (Ts, \gamma(Tt)), \\
\end{array}
\]
where the Reeb orbit $\ga$ is $T$-periodic.  

Since 
\[
u_s = T\ptau, \ \ \ u_t= T R_{\A} = \Jt u_s,
\]
the area vanishes, $A(u)=0$.  But $E(u)=T,$ because 
\[
E(u)=\sup_{\phi \in \mathscr{S}} \lim_{R \to \infty} \int_{[-R, R] \times S^1}u^*d\A_\phi = \sup_{\phi  \in \mathscr{S}} \lim_{R \to \infty} (\phi(R)T - \phi(-R)T) = T.
\]
Moreover in this situation we note that
\begin{equation}\label{fineass}
\begin{array}{lcll}
\displaystyle \lim_{s \to \infty} f(s,t) &=& \ga(Tt) & \mbox{ in } C^\infty(M), \\
\displaystyle \lim_{s \to \infty} \frac{a(s,t)}{s} &=& T & \mbox{ in } C^\infty(\R).\\
\end{array}
\end{equation}
In other words, the $M$-part of $u$ converges to a periodic orbit of the Reeb vector field of period $T$, while the $\R$-part is asymptotic to $(s,t) \to Ts$ as $s \to \infty$.   This will also be the situation for general finite energy cylinders; see Theorem \ref{reebasym}.
\end{example}

\bigskip

The following trivial \murve \ satisfies $E(u) = \infty$.
\begin{example}\em
Let
\[
\begin{array}{crcl}
u: &(\C, j_0)& \to &(\R \times M, \Jt) \\
&s+it& \mapsto& (s, \gamma(t)) \\
\end{array}
\]
where $\gamma: \R \to M$ is a closed orbit of the Reeb flow.  Then if we take a function $\phi \in \mathscr{S}$ with $\phi \neq 0$ we compute
\[
\int_\C u^* d\A_\phi  = \int_\C \phi'(s) ds dt = \left(\phi(\infty) - \phi(-\infty)\right)\int_\R dt = \infty.
\]
\end{example}


\subsection{Hofer energy and asymptotics}\label{asymptotics}
The finiteness of Hofer energy is an extremely important distinguishing characteristic of \curves, as it implies asymptotic convergence in $C^\infty$ to nondegenerate periodic orbits of the Reeb vector field at the punctures.  Moreover, the existence of periodic orbits of a Reeb vector field will follow from the existence of a finite energy surface.

The following theorem is an amalgamation of Theorems 1.2 and 1.3 in \cite{HWZ1}as well as a removable singularity result that is not explicitly stated as a theorem in \cite{HWZ2}, but proven within the exposition of the latter's introduction.

\begin{theorem}\label{reebasym}
Let $v=(a, f): ((\rho,\infty) \times S^1 ,j) \to (\R \times M, \Jt)$ be a \murve \ such that, 
\[
0\leq E(v) < \infty.
\]
then the following holds.
\begin{itemize}
\item The following limit exists,
\begin{equation}\label{eqT}
T := \lim_{s\to \infty} \int_{S^1} f^*\A.
\end{equation}
\item If $T=0$ then the corresponding curve $u$ on the punctured disk $D\setminus \{ 0 \}$ defined by $u(e^{2\pi(s+it)})=v(s,t)$ can be extended smoothly to the whole disk $D \subset \C$.
\item If $T \neq 0$ then there exists a $|T|$-periodic orbit $\ga(t)$ of the Reeb vector field.  In addition, there exists a sequence $\rho_k \to \infty$ such that 
\[
\lim_{k \to \infty}f(\rho_k,t) = \ga(tT) \mbox{ in } C^{\infty}(\R)
\]
and 
\[
\lim_{s \to \infty} \frac{a(s,t)}{s} = T.
\]
 If $\ga$ is nondegenerate then $f(s, \cdot)$ converges to a parametrization of $\ga$, namely 
\[
\lim_{s \to \infty}f(s, \cdot) = \ga(tT), 
\]
 with convergence in $C^{\infty}(\R).$ 
 \end{itemize}
 \end{theorem}

\begin{proposition}[Characterization of finite energy cylinders in symplectizations]
Let $u = (a,f) : (\R \times S^1, j) \to (\R \times M, \Jt)$ be a pseudoholomorphic cylinder with $0<E(u) <\infty$.  Then the negative end $u$ has either a removable puncture or it converges to a Reeb orbit in the $M$-component at the $-\infty$ end of $(\R \times M, \Jt)$.  The positive end of $u$ always converges to a Reeb orbit at the $+\infty$ end of $(\R \times M, \Jt)$. 
\end{proposition}
\begin{proof}
From Proposition \ref{constant}, any pseudoholomorphic curve defined on a closed surface is necessarily constant.  As a result we know that both ends cannot be removable punctures unless $u$ is the constant curve.  In addition, at least one of the ends must tend towards $+ \infty$ otherwise we obtain a contradiction with the maximum principle in Proposition \ref{max-principle}.  
\end{proof}

One can interpret Theorem \ref{reebasym} to say that an finite energy pseudoholomorphic map $u:( S^2 \setminus \{ x, y_1, ... y_s \}, j) \to (\R \times M, \Jt)$ converges to vertical cylinders over closed Reeb orbits at $t=\pm \infty$. We illustrate this in Figure \ref{curve}.

 \begin{figure}[h!]
  \centering
    \includegraphics[width=0.5\textwidth]{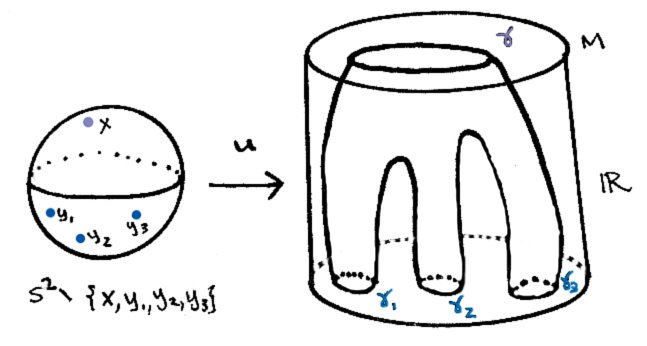}
      \caption{A pseudoholomorphic curve $u$ in $\R \times M$ with $s=3$.}
      \label{curve}
\end{figure}

Recall that the symplectic action of a $T$-periodic Reeb orbit $\ga$ is defined to be 
\[
\cala(\ga) := \int_{\ga} \A = \int_{S^1} \ga^*\A = \displaystyle \int_0^T \A(\dot{\ga}(t))dt = T.
\]
The following lemma shows that the action decreases along finite energy pseudoholomorphic cylinders $u:=(a,f)$, which converge to the Reeb trajectories $\ga_+$ at the positive end and $\ga_-$ at the negative end of the symplectization, as in Example \ref{energyex}, equation (\ref{fineass}). 

\begin{lemma}
Let $u=(a,f)$ be a finite energy pseudoholomorphic cylinder in $\calm(\gp; \gm)$. Then
\[
\cala(\ga_+) \geq \cala(\ga_-),
\]
with equality if and only if $\ga_+ = \ga_-$ and the image of $u$ is an $\R$-invariant pseudoholomorphic cylinder.
\end{lemma}

\begin{proof}
Let $u : (\R \times S^1, j) \to (\R \times M, \Jt)$ be a pseudoholomorphic cylinder which converges to the Reeb trajectory $\ga_+$ at the positive end and to $\ga_-$ at the negative end.  By Stokes' theorem,
\[
\cala(\gp) - \cala(\gm) =  \int_{\R \times S^1} f^* d\A
\]
The asymptotics of $u$ imply that the above integral converges.  We obtain
\[
A(u) = \int_{\R \times S^1} f^* d\A = \int_{\ga_+} \A - \int_{\ga_-} \A = \cala(\ga_+) - \cala(\gm) = T_+ - T_-.
\]
From the condition that $\Jt$ is a compatible almost complex structure we know that $f^*d\A \geq 0$ in $\R \times S^1$.  We obtain equality only in the case that $u$ is tangent to $\R \times R_\A$, i.e. when $\gp = \gm$ and $u$ is as in Example \ref{energyex}.
\end{proof}

\section{Traversing transversality troubles}\label{transversality}
Regularity results applicable to the moduli spaces necessary to the conjectures of \cite{EGH} have only been given under specialized circumstances \cite{HWZ2, HWZ3, HT2, Wtrans}.  In dimension 3, we are able to combine these  results with the Conley-Zehnder index computations of Section \ref{regularity} to construct a well-defined chain complex without the use of virtual chains.  We begin this section with a recollection of the results of \cite{HWZ3, HT2} for immersed somewhere injective curves and end it by providing the results of \cite{HWZ2, Wtrans}, which will be used in Section \ref{regularity}.

\subsection{Somewhere injective and immersed \curves}\label{modulispaces}
The transversality results of \cite{HWZ3} apply to moduli spaces consisting of immersed somewhere injective asymptotically cylindrical curves.  After appealing to \cite{HT2, Wtrans} we conclude that all cylinders\footnote{This includes unbranched multiply covered cylinders.} of index $\leq 2$ are immersed.  We also state the theorem that a non-constant asymptotically cylindrical pseudoholomorphic map factors through a somewhere injective one, which follows from the corresponding result for finite energy planes, proven in the appendix of  \cite{HWZ2}. 

\begin{definition}\em
An asymptotically cylindrical \curve \ 
\[
u: (\ds:=\Sigma \setminus ( \g_+ \sqcup \g_- ), j) \to (W, J)
\]
is said to be \textbf{multiply covered} whenever there exists a \curve \ 
\[
v: (\ds':=\Sigma' \setminus ( \g'_+ \sqcup \g'_-), j') \to (W, J),
\]
 and a holomorphic branched covering $\varphi: (\Sigma, j) \to (\Sigma', j') $ with $\g'_+ = \varphi(\g_+)$ and  $\g'_- = \varphi(\g_-)$ such that
\[
u = v \circ \varphi, \ \ \ \mbox{deg}(\varphi)>1,
\]
allowing for $\varphi$ to not have any branch points.
\end{definition}
The expression of both cylindrical contact differentials (\ref{pa1}, \ref{pa2}) involved the \textbf{multiplicity} of a finite energy pseudoholomorphic cylinder $\calc$, defined by $\mult(\calc) := \mbox{deg}(\varphi).$

An asymptotically cylindrical \curve \ $u$ is called \textbf{simple} whenever it is not multiply covered.  From \cite{HWZ2} we can conclude that every simple asymptotically cylindrical curve  is \textbf{somewhere injective}, meaning for some $z \in \ds$,
\[
du(z) \neq 0 \ \  \ u^{-1}(u(z))=\{z\}.
\]
A point $z \in \ds$ with this property is called an \textbf{injective point} of $u$.  

An \textbf{immersed} pseudoholomorphic curve (with one positive puncture) is an equivalence class of tuples $(\Sigma, j, \Gamma, u)$, as defined in Section \ref{pseudsymp}, such that $u$ is an immersion. Theorem \ref{HT2thm} and Proposition \ref{Z} imply that after a generic choice of $\Jt$ somewhere injective curves of sufficiently small index are immersed because generically the order of the critical points of a curve, defined in (\ref{w-Z}), provides a lower bound on the index of a curve.

Before we can state the results of \cite{HWZ3}, we must define a suitable Banach space on which we can vary $\Jt$.  To accomplish this we follow the approach of Floer, as in \cite{Fgrad} and introduce Floer's $\ce$-space, as follows.  Let $\bar{J}_0: \xi \to \xi$ be compatible with $d\A$.  Let $\Jt_0$ be the corresponding $\A$-compatible extension to  $T(\R \times M)$, as in Definition \ref{complexstruc}.  Consider the space of all smooth bundle maps $\Psi(p): \xi_p \to \xi_p$ satisfying
\begin{equation}
\begin{array}{lcl}
\Psi(p)J_0(p) + J_0(p)\Psi(p)& = &0 \\
d\A(\Psi X, Y) + d\A(X, \Psi Y) &= & 0 \ \ \ \mbox{for } X, Y \in \xi.
\end{array}
\end{equation}
Let $\ee = \{ \ee_n\}_{n=1}^\infty$ be a sequence of positive numbers such that $\lim_{n \to \infty} \ee_n =0$.  The $\ce$-space consists of $C^\infty$ homomorphisms of $\xi$ whose sums of  weighted $C^k$ norms decay sufficiently fast.  It is defined as
\[
\ce = \left\{ \Psi \in \mbox{Hom}_\R(\xi), \ \Psi \in \Ci \ | \ ||\Psi||_\ee := \sum_{n=1}^\infty \ee_n ||\Psi||_n < \infty \right\},
\]
where $||\Psi||_k$ is the $C^k$ norm with respect to a metric on $W$.  
If $\ee_n \to 0$ sufficiently fast then $(\ce, ||\cdot||_\ee)$ is a separable Banach space, which is dense
in $\Ci$; see \cite{Fgrad}.  For $\Delta > 0$, denote by
\[
U_\Delta = \left\{ \Jt \ | \ \bar{J} = \bar{J}_0 \exp (-J_0 \Psi), \ \Psi \in \ce, \ ||\Psi||_\ee < \Delta  \right\}.
\]
Here it is understood that we extend $\Jt$ from $\bar{J}$ in the usual manner, that $\Jt \ptau = R_{\A}$. The inverse map given by $\Psi \mapsto \Jt \in U_\Delta$ provides a global chart for $U_\Delta$, endowing $U_\Delta$ with the structure of a Banach manifold diffeomorphic to an open subset in a  separable Banach manifold.  

\begin{remark}\em
The expression, ``choosing $\Jt$ generically,'' means that $\Jt$ has been chosen from the residual subset $\mathcal{S} \subset U_{\Delta}$ such that Corollary \ref{bigdim} holds.  
\end{remark}

Next we state Theorem 1.10 of \cite{HWZ3}, from which the subsequent corollaries follow for the usual reasons described in the text of Section 1 in \cite{HWZ3}.

\begin{theorem}[Hofer-Wysocki-Zehnder \cite{HWZ3}]
Let \[
u: (\ds := \Sigma \setminus \{ x; y_1, ... y_s\}, j_0) \to (\R \times M, \Jt)
\]
 be an immersed somewhere injective pseudoholomorphic curve asymptotic to the nondegenerate Reeb orbits $\gamma$ at the positive puncture $x$ and $\gamma_1,...,\gamma_s$ at the negative punctures $y_1, ..., y_s$ and $\Jt$ an $\A$-compatible almost complex structure.  Then the set $\ncal(\gamma;\gamma_1,...\gamma_s)$ of all pairs $(C, \Jt)$, consisting of an equivalence class $C$ of tuples $(\Sigma,  j, \Gamma, u)$ and an $\A$-compatible $\R$-invariant almost complex structure $J$, carries the structure of a separable manifold and the projection map
 \[
 pr: \ncal(\gamma; \gamma_1,...\gamma_s) \to U_{\Delta}
 \]
 \[
pr(C, \Jt) = \Jt
\]
is a Fredholm map with Fredholm index near $u$ given by
\begin{equation}\label{dim}
\mbox{\em ind}(u) = (s+1 - \chi(\Sigma)) + \czm(\gamma) - \displaystyle \sum_{i=1}^s\czm(\gamma_i).
\end{equation}
\end{theorem}


We obtain the following two useful corollaries from the above result.
\begin{corollary}\label{regJdim}
For regular values $\Jt$ of $pr$, $pr^{-1}(\Jt)$ is a smooth finite dimensional manifold, whose dimension is given by
\[
\mbox{\em dim } \ncal(\gamma; \gamma_1,...\gamma_s) := (s+1 - \chi(\Sigma)) + \czm(\gamma) - \displaystyle \sum_{i=1}^s\czm(\gamma_i). 
\]
\end{corollary}

\begin{corollary}\label{bigdim}
There exists a dense subset $\mathcal{S} \subset U_{\Delta}$ such that for every $\Jt \in \mathcal{S}$ if $u$ is a somewhere injective immersed $\Jt$-curve in $\ncal(\gamma;\gamma_1,...\gamma_s)$,  then
\[
(s+1 - \chi(\Sigma)) + \czm(\gamma) - \displaystyle \sum_{i=1}^s\czm(\gamma_i)  \geq 1
\]
provided that $\pi \circ Du$ does not vanish identically.  Recall that $\pi: TM \to \xi$ is the projection along the Reeb vector field $R_\A$.
\end{corollary}




The next result that we need is Theorem 4.2 from \cite{HT2}, which will be used to obtain regularity for all finite energy cylinders of index $\leq 2$ in Section \ref{arbitrary}.

\begin{theorem}[Hutchings-Taubes \cite{HT2}]\label{HT2thm}
Let $(M, \A)$ be a nondegenerate contact 3-manifold.   If the $\A$-compatible almost complex structure $\Jt$ on $\R \times M$ is generic, then all somewhere injective asymptotically cylindrical pseudoholomorphic curves of index $\leq 2$ are immersed.
\end{theorem}

\subsection{Factoring multiply covered curves through simple curves}

In this section we provide a precise statement and proof of the folk theorem that any non-constant asymptotically cylindrical curve factors through a somewhere injective curve.    As a result we can conclude that if an asymptotically cylindrical \curve \ is not somewhere injective then it is necessarily multiply covered.  The author learned this proof from Chris Wendl.\footnote{The proof of Theorem \ref{multiplycovered} is modeled on the blog post by Chris Wendl, \\
$\mbox{\scriptsize https://symplecticfieldtheorist.wordpress.com/2014/12/10/somewhere-injective-vs-multiply-covered/}$ \\
This was in turn modeled on the ``shorter'' proof of Proposition 2.5.1 \cite{MS2}.  The ``longer'' proof of Proposition 2.5.1 was adapted to give a proof of Theorem \ref{multiplycovered} for finite energy planes,  subsuming the  11-page appendix of \cite{HWZ2}.  This longer proof extends to asymptotically cylindrical pseudoholomorphic curves we are considering, but this paper is already long enough. }  Our proof follows along the proof given for the closed case in \cite{MS2}, appealing to the local results of \cite{MiWh} and the asymptotic behavior of nonconstant finite energy curves near a puncture in \cite{Sief}.

\begin{theorem}\label{multiplycovered}
Assume $(W,J)$ is a symplectic cobordism with cylindrical ends and 
\[
u : (\dot{\Sigma},j) \to (W,J) 
\]
is a nonconstant asymptotically cylindrical $J$-holomorphic curve asymptotic to nondegenerate Reeb orbits. Let $(\Sigma,j)$ denote the closed Riemann surface from which $(\dot{\Sigma},j)$ is obtained by deleting finitely many points. Then there exists a factorization 
\[
u = v \circ \varphi, 
\]
where
\begin{itemize}
\item $v : (\dot{\Sigma}',j') \to (W,J)$ is an asymptotically cylindrical $J$-holomorphic curve that is embedded outside a finite set of critical points and self-intersections, and
 
\item $\varphi : (\Sigma,j) \to (\Sigma',j')$ is a holomorphic map of positive degree, where $(\dot{\Sigma}',j')$ is obtained from the closed Riemann surface $(\Sigma',j')$ by deleting finitely many points.
\end{itemize}
\end{theorem}

\begin{remark} \em
  In the statement of the above theorem we can take $(W,J)$ to be an almost complex manifold with cylindrical ends adapted to stable Hamiltonian structures, cf. \cite{Wtrans} for the definition of such a manifold. We can also relax the assumption that the Reeb orbits are nondegenerate, and merely stipulate that they  be Morse-Bott \cite{Wtrans}.  This result does not apply to asymptotically cylindrical curves with boundary \cite{LL}, e.g. $\partial \Sigma \neq \emptyset$, which are not of interest in this paper.  
  \end{remark}
  
As in the closed case of \cite{MS2}, we construct $\ds'$ (minus some extra punctures) explicitly from $\mbox{im}(u)$ by removing finitely many singular points.   This allows us to take $v$ to be the inclusion,
\[
v: \ds \hookrightarrow W.
\]
As a result the map
\[
\varphi: \ds \to \ds'
\]
is uniquely determined and extends holomorphically over all punctures because of a removal of singularities theorem.  To proceed we need to know how $\mbox{im}(u)$ behaves near each of its singularities.  The demeanor of the curve near its singularities can be classified into the following three types: branchings, intersections, and asymptotics.  The following three lemmas elucidate these behaviors.  Throughout $(\Sigma, j)$ is a closed Riemann surface.

\begin{lemma}[Branchings of curves]\label{lemmabranchings}
Suppose 
\[
u: (\Sigma, j) \to (W, J) 
\]
is a nonconstant pseudoholomorphic curve and $z_0 \in \Sigma$ is a critical point of $u$.  Then a neighborhood $z_0 \in \mathcal{U} \subset \Sigma$ can be biholomorphically identified with the unit disk $\mathbb{D} \subset \C$ such that
\[
u(z)=v(z^k) \mbox{ for } z \in \mathbb{D}=\mathcal{U},
\]
where $k \in \N$ and $v: (\mathbb{D},j_0) \to (W,J)$ is an injective pseudoholomorphic map with no critical points except possibly at the origin.
\end{lemma}

The above result follows primarily via Theorem 6.1 of \cite{MiWh}.   A weaker, less analytically involved version of Lemma \ref{lemmabranchings} is given in Theorem 2.114 in Section 2.13 of \cite{Wtrans}. The main idea in both proofs is that every almost complex structure is locally $C^\infty$-close to an integrable complex structure so locally the behavior of pseudoholomorphic curves should match the integrable case.

The above local singularity formula of Micaleff-White yields a coordinate system in which a pseudoholomorphic curve can be written locally as a complex polynomial.  Theorem 7.1 of \cite{MiWh} allows us to use this local formulation when investigating any intersection point of finitely many pseudoholomorphic curves to obtain the following result.  For proofs and recollection of the Micaleff-White results aimed at a symplectic geometer, we direct the inquisitve to the wonderful exposition found in Appendix E of \cite{MS2}.

\begin{lemma}[Intersections of curves]
Suppose 
\[
\begin{array}{lcl}
u: (\Sigma, j) &\to& (W, J), \\
v:(\Sigma', j') & \to & (W,J) \\
\end{array}
\]
are two nonconstant pseudoholomorphic curves with an intersection
\[
u(z)=v(z').
\]
Then there exist neighborhoods $z\in \mathcal{U} \subset \Sigma$ and $z' \in \mathcal{U'} \subset \Sigma'$ such that either
\[
\begin{array}{l c l}
u( \mathcal{U}) =  v(\mathcal{U'}) & \mbox{or} & u( \mathcal{U}\setminus \{z\}) \cap v( \mathcal{U'} \setminus \{z' \}) = \emptyset. \\
\end{array}
\]
\end{lemma}
A weaker version of the above result is stated and proven as Theorem
 2.116 of \cite{Wnotes}.  The above two lemmas are sufficient to prove the analogue of Theorem \ref{multiplycovered} in the closed setting.  To proceed to work with asymptotically cylindrical curves we need the following generalized version of the aforementioned Micaleff-White results for punctured pseudoholomorphic disks, e.g. half cylinders.  The following lemma follows from the set of relative asymptotic formulas found in Theorems 2.2 and 2.3 of \cite{Sief}.   

\begin{lemma}[Asymptotics of curves]\label{lemmaasymptotics}
Assume 
\[
u: (\ds, j) \to (W, J) 
\]
is an asymptotically cylindrical curve and $z_0 \in \Sigma$ is a puncture at which the asymptotic orbit is nondegenerate.  Then a punctured neighborhood $\dot{\mathcal{U}}$ of $z_0$ in $\ds$ can be biholomorphically identified with the puncture unit disk $\dot{\mathbb{D}}= \mathbb{D} \setminus \{ 0 \}\subset \C$ such that
\[
u(z)=v(z^k) \mbox{ for } z \in \dot{\mathbb{D}}=\dot{\mathcal{U}},
\]
where $k \in \N$ and $v: (\dot{\mathbb{D}},j_0) \to (W,J)$ is an embedded asymptotically cylindrical map.
\bigskip
\noindent If 
\[
w: (\ds', j') \to (W, J) 
\]
is another asymptotically cylindrical curve with a puncture $z_0' \in \Sigma'$, then the images of $u$ near $z_0$ and $w$ near $z_0'$ are either identical or disjoint.
\end{lemma}

\begin{remark}\em
In Lemma \ref{lemmaasymptotics} the second statement is not particularly remarkable provided $u$ and $w$ are asymptotic to different orbits at the two punctures under consideration.  What is remarkable, is that this statement is still true even when both curves are asymptotic to (covers of) the same orbit.  

\bigskip

We note that Siefring's results as well as the above lemma holds when the pseudoholomorphic curve is asymptotic to Reeb orbits which are merely Morse-Bott as opposed to non-degenerate.  These results also hold when we allow $(W,J)$ to be an almost complex manifold with cylindrical ends adapted to stable Hamiltonian structures.  However, the author does not wish to define these notions here, preferring to refer the interested reader to \cite{Wtrans}.
\end{remark}

At last we are ready to legitimize the above folk theorem with the following proof.
\begin{proof}[Proof of Theorem \ref{multiplycovered}]
In this proof we will construct the domain $\ds'$ from the image curve $u(\ds)$ in W.  Let 
\[
\mbox{Crit}(u) = \{ z \in \ds \ | \ du(z) =0\}
\]
denote the set of critical points of $u$.  

\medskip

\noindent Define $\Delta \subset \ds$ to be the set of all points $z \in \ds$ such that there exists
\begin{itemize}
\item $z'\in \ds$ such that $z \neq z'$ and $u(z)=u(z')$, but also
\item neighborhoods $z\in \mathcal{U} \subset \ds$ and $z' \in \mathcal{U'}\subset \ds$ with $ u( \mathcal{U}\setminus \{z\}) \cap u( \mathcal{U'} \setminus \{z' \}) = \emptyset.$
 \end{itemize}
By Lemmas \ref{lemmabranchings}-\ref{lemmaasymptotics} both of these sets are discrete, meaning they have no accumulation points near the punctures.  As a result both are finite sets.  

Let 
\[
\ddot{\Sigma}:=\ds \setminus \{ \mbox{Crit}(u) \cup \Delta \}.
\]
Moreover, the set
\[
\ddot{\Omega}=u(\ds \setminus \{ \mbox{Crit}(u) \cup \Delta \}) \subset W
\]
is then a smooth submanifold of $W$ with $J$-invariant tangent spaces.  Thus   $\ddot{\Omega}$ inherits a natural complex structure $j'$ for which the inclusion
\[
\iota: (\ddot{\Omega},j') \hookrightarrow (W,J)
\]
is pseudoholomorphic.  

\medskip

Next we explain how to construct the desired Riemann surface $(\dot{\Sigma'}, j')$, from $(\ddot{\Omega}, j')$ by appropriately adding in points to $\ddot{\Omega}$ and extending the almost structure.  Let
\[
\widehat{\Delta} =  \{ \mbox{Crit}(u) \cup \Delta \} / \sim
\]
where two points in  $\{ \mbox{Crit}(u) \cup \Delta \}$ are defined to be equivalent whenever they have neighborhoods in $\ds$ with identical images under $u$.   Then for each $[z] \in \widehat{\Delta}$, the branching lemma, Lemma \ref{lemmabranchings}, provides an injective $J$-holomorphic map
\[
u_{[z]} :(\mathbb{D}, j_0) \to (u(\mathcal{U}), J)
\]
where $\mathcal{U}$ is an appropriate neighborhood of $z$.  We define  $(\dot{\Sigma'}, j')$ as follows.

\medskip

Define
\[
\ds':= \ddot{\Omega} \bigcup_\Phi \left( \bigsqcup_{[z] \in \widehat{\Delta}} \mathbb{D} \right),
\]
where the gluing map $\Phi$ is the disjoint union for each $[z] \in \widehat{\Delta}$ of the maps
\[
u_{[z]}: \mathbb{D}\setminus \{0\} \to \ddot{\Omega}. 
 \]
Moreover, since $u_{[z]}$ is holomorphic, the complex structure $j'$ extends from $\ddot{\Omega}$ to $\dot{\Sigma'}.$

We have already established that $u_{[z]}$ extends over the whole disk,
\[
u_{[z]}: (\mathbb{D}, j_0) \to (W, J).
\]
Thus the inclusion
\[
\iota: (\ddot{\Omega}, j') \hookrightarrow (W, J)
\]
must extend to a pseudoholomorphic map
\[
v: (\ds', j') \to (W,J)
\]
which restricts to $\ddot{\Omega}$ as an embedding and otherwise has at most finitely many critical points and double points.  Moreover, the following restriction of $u$ defines a holomorphic map
\[
u\arrowvert_{\ds \setminus \{ \mbox{\scriptsize Crit}(u) \cup \Delta \}} \to (\ddot{\Omega}, j').
\] 
By the removal of singularities the map $u\arrowvert_{\ds \setminus \{ \mbox{\scriptsize Crit}(u) \cup \Delta \}}$ extends to a proper holomorphic map
\[
\varphi: (\ds, j) \to (\ds', j')
\]
such that $u =v \circ \varphi$.  

To finish proving the last part of the theorem we observe that the first statement in the asymptotics lemma implies that the complements of certain compact subsets in $(\dot{\Sigma},j)$ and $(\dot{\Sigma}',j')$ can each be identified biholomorphically with punctured disks $\dot{\mathbb D}$ on which $\varphi(z) = z^k$ for various $k \in {\mathbb N}.$  As a result we can glue disks to $\dot{\Sigma}'$ so that it becomes the complement of a finite set of punctures in some closed Riemann surface $(\Sigma',j'). $ Thus $\varphi$ extends to a nonconstant holomorphic map $(\Sigma,j) \to (\Sigma',j')$.
\end{proof}

\subsection{Automatic transversality}\label{trans}
The criterion for automatic transversality of asymptotically cylindrical curves in 4-dimensional symplectic cobordisms $(W,J)$ with cylindrical ends is expressed in terms of the asymptotic data, homological properties and number of critical points of the curve in \cite{Wtrans}.  When these numerical conditions are met one can demonstrate that multiply covered curves will be cut out transversally by the Cauchy-Riemann equations without genericity assumptions on $J$, providing these geometrically natural moduli spaces the structure of globally smooth orbifolds. Before stating the key results we briefly review the setting and notation of interest.  In the simpler setting of the symplectization $\R \times M$ equipped with an $\A$-compatible $J$, proofs of the automatic transversality results are substantially shorter and appear in Section 4 of \cite{HNdyn}.

If the asymptotic orbits of a curve $u \in \calm(\ga; \ga_1,...,\ga_s)$ are all nondegenerate, then the \textbf{virtual dimension} of $\calm(\ga; \ga_1,...,\ga_s)$ is equal to the \textbf{index}, which is given by
\begin{equation}\label{w-index}
{\ind}(u) = -\chi(\ds)   + \mu^\Phi(u) + 2c_1^\Phi(u^*TW),
\end{equation}
as in \cite{Wtrans}, with $\chi(\ds) =(2-2g - \#\g^+ - \#\g^-)$ and $\Phi$ a trivialization of $\xi$ along the asymptotic orbits of $u$.   In particular, $c_1^\Phi(u^*TW)$ is the relative first Chern number of $(u^*T, J) \to \ds $ with respect to a suitable choice of $\Phi$ along the ends and boundary.  

Additionally, with respect to $\Phi$,
  \[
\mu^\Phi(u) + 2c_1^\phi(u^*TW)= \mu_{CZ}^\Phi(\gamma) - \displaystyle \sum_{i=1}^s\mu_{CZ}^\Phi(\gamma_i).
 \]
 
\begin{remark}\label{trivchoice}\em 
We can always choose a trivialization $\Phi$ (fixed up to homotopy) such that $c_1^\Phi(u^*TW)=0$; see Section 1.1.1 and Remarks \ref{trivtpy}-\ref{vanishingtriv}.  More precisely, we choose a trivialization $\Phi$ so that $c_1^\Phi(v^*T(\R \times M))=0$ for a somewhere injective curve genus 0 asymptotically cylindrical curve $v$ with one positive puncture and at least one negative puncture.  This implies for any (branched) cover  $u := \varphi \circ v$,  that $c_1^\Phi(u^*T(\R \times M))=0$. 

We fix such a trivialization $\Phi$ and will write $\czm$ as a shorthand for $\mu_{CZ}^\Phi$.  When $g=0$, as in the setting of interest to this paper, without loss of generality we can work with the following index formula 

   \begin{equation}\label{indexformula}
 {\ind}(u) = -(1-s) + \czm(\gamma) - \displaystyle \sum_{i=1}^s\czm(\gamma_i).
 \end{equation}
 \end{remark}
With this notation understood we move onwards. If $(\Sigma, j) = (S^2, j_0)$ and $\g=\{x\}$ or $\{x, y\}$, then $\ds$ is biholomorphic to the complex plane or the cylinder.  In these cases we must say a little more about the role of the moduli space of these domains before proceeding with the usual functional analytic set up for asymptotically cylindrical pseudoholomorphic curves.   Doing so requires a slight detour to visit some classical results regarding moduli spaces of Riemann surfaces as they are related to the analysis of Cauchy-Riemann type operators, following \cite[\S 3.1]{Wtrans}.

Let $\mathcal{J}(\Sigma)$ denote the space of smooth complex structures on $\Sigma$ that induce the given orientation.  Denote 
\[
\mbox{Diff}_+(\Sigma, \Gamma)
\]
 to be the group of orientation preserving diffeomorphisms on $\Sigma$ that fix $\Gamma$.  Let
\[
\mbox{Diff}_0(\Sigma, \Gamma)\subset\mbox{Diff}_+(\Sigma, \Gamma)
\]
be the subgroup of those diffeomorphisms homotopic to the identity.  Both of these groups act on $\mathcal{J}(\Sigma)$ by
\[
(\varphi, j) \mapsto \varphi^* j.
\]
The \textbf{Teichm\"uller space} of $\ds$ is a smooth finite dimensional manifold defined by
\[
\mathscr{T}(\ds) := \mathcal{J}(\Sigma) / \mbox{Diff}_0(\Sigma, \Gamma),
\]
whose quotient by the mapping class group
\[
\mathscr{M}(\ds) := \mbox{Diff}_+(\Sigma, \Gamma)/\mbox{Diff}_0(\Sigma, \Gamma)
\]
yields the \textbf{moduli space of Riemann surfaces} 
\[
\mathcal{M}(\ds) := \mathscr{T}(\ds) / \mathscr{M}(\ds)  = \mathcal{J}(\Sigma)/ \mbox{Diff}_+(\Sigma, \Gamma).
\]
The moduli space of Riemann surfaces $ \mathcal{M}(\ds) $ of  genus $g$ and $\# \g$ interior marked points and no boundary components is an orbifold and in general has the same dimension as $\mathscr{T}(\ds)$.  We say that $\ds$ is \textbf{stable} whenever $\chi(\ds):=2-2g-\#\g <0$, in which case
\[
\mbox{dim}  \mathcal{M}(\ds) = 6g  + 2 \#\g = -3 \chi(\ds) - \# \g,
\]
and the \textbf{automorphism group} 
\[
\mbox{Aut}(\ds, j) = \{ \varphi \in \mbox{Diff}_+(\Sigma, \Gamma) \ | \ \varphi^* j =j \} 
\]
is finite for any $j \in \mathcal{J}(\Sigma)$, though the order may depend on $j$.  

The cylinder $\R \times S^1$ is non-stable and 
\[
\mathcal{M}(\R \times S^1) = \{ [i] \}, \ \ \ \mbox{dim Aut}(\R \times S^1,i)=2.
\]
The plane $\C$ is also non-stable and
\[
\mathcal{M}(\C) = \{ [i] \}, \ \ \ \mbox{dim Aut}(\C,i)=4.
\]
Otherwise all other domains of asymptotically cylindrical curves of interest in this paper are stable.  Since the mapping class groups of both the cylinder and plane are trivial we know for $\Sigma = S^2$ and $\g = \{x\}$ or $\{x, y\}$
\[
\mathcal{M}(S^2 \setminus \g ) = \mathscr{T}(S^2 \setminus \g).
\]
Additionally for unstable $\ds$ we have that
\[
\mbox{dim Aut}(\ds,j) - \mbox{dim} \mathcal{M}(\ds) = 3\chi(\ds) + \#\g.
\]
Fixing $p > 2$ the latter is the Fredholm index of the standard linear Cauchy-Riemann operator.

Going back to the $\deebar$ operator, recall that after fixing a complex structure $j$ on $\Sigma$ there is a Banach space  bundle $\mathcal{B} \to \cale$ whose fibers are spaces of complex antilinear bundle maps. The nonlinear Cauchy-Riemann operator $\deebar$ can be expressed as a smooth section of this Banach space bundle,
\[
\deebar: \mathcal{B} \to \cale
\]
\[
\deebar(u)=du + \Jt \circ du \circ j.
\]
whose zeros are parametrization of asymptotically cylindrical pseudoholomorphic curves $u: (\ds, j) \to (W,J)$.  The linearization of $\deebar$ at a zero $u$ defines a Cauchy-Riemann type operator
\[
\begin{array}{lrcl}
\mathbf{D}_u : & \Gamma(u^*TW) &\to & \Gamma(\overline{\mbox{Hom}}_\C(T\ds, u^*TW)) ,\\
& v & \mapsto & \Delta v + J \circ \Delta v \circ j + (\Delta_v J ) \circ du \circ j,
\end{array}
\]
where $\Delta$ is any symmetric connection on $W$.  As a bounded linear operator $T_u\mathcal{B} \to \cale_u$ is Fredholm and for dim$(W)=4$ we have that the Fredholm index of $\mathbf{D}_u$ is
\[
\ind(\mathbf{D}_u) = 2\chi(\ds) + 2c_1^\Phi(u^*TW) + \mu^\Phi(u) + \#\g.
\]

The above Teichm\"uller set up and accompanying theory \cite[\S 3.2]{Wtrans} allows one to vary complex structures on the domain so that one can appropriately make sense of the ``total linearization'' at $(j,u)$ of $\deebar^{-1}(0)$, expressed as $D\deebar(j,u)$.   Moreover  since $ \mathbf{D}_u$ is Fredholm and $\mathscr{T}$ is finite dimensional we can conclude that $D\deebar(j,u)$ is also Fredholm with Fredholm index given by
\[
\ind D \deebar (j, u) = \mbox{dim} \mathscr{T} + \ind ( \mathbf{D}_u) = \ind (u) + \mbox{dim Aut}(\ds, j). 
\]   
For  non-stable $\ds$ this immediately  yields (\ref{w-index}).  For the genus 0 stable domains of interest in this paper, e.g. $\Sigma = S^2$ and $\#\g \geq 3$, we have that $\mbox{dim}  \mathcal{M}(\ds) = \mbox{dim}  \mathscr{T}(\ds)$ and again (\ref{w-index}) immediately follows.

Any neighborhood of any non-constant $ u \in \calm(\ga, \ga_1,...,\ga_s)$ is in one-to-one correspondence with $\deebar^{-1}(0) / \mbox{Aut}(\ds,j)$, where the group  $\mbox{Aut}(\ds,j)$
of biholomorphic maps $(\Sigma, j) \to (\Sigma,j)$ fixing $\g$ acts on pairs $(j',u') \in \deebar^{-1}(0)$ by
\[
\varphi \cdot (j', u') = (\varphi^* j', u' \circ \varphi).
\]
Note that in the cases where $\ds$ is not stable and is $\C$ or $\R \times S^1$ then $\mathscr{T}$ contains only $j_0$ and $(j_0,u) \sim (j_0, u')$ if and only if $u' = u \circ \varphi$ for some $\varphi \in \mbox{Aut}(\ds, j_0).  $

\begin{definition}\em
One says that  $u \in \calm$ is \textbf{regular} whenever it represents a transverse intersection with the zero-section.  This is equivalent to requiring that the linearization
\[
D\deebar(u): T_u\mathcal{B} \to T_{(u,0)}\cale
\]
be surjective. 
\end{definition}

 If $u$ is a non-constant curve then the action of $\mbox{Aut}(\ds,j)$ induces a natural inclusion of its Lie algebra $\mathfrak{aut}(\dot{\Sigma},j)$ into $\ker D\deebar(u).$ The first theorem in \cite{Wtrans} is the following standard folk theorem, which we have restricted to curves limiting on nondegenerate orbits in symplectizations\footnote{This theorem holds for almost complex manifolds with noncompact cylindrical ends approaching codimension 1 manifolds $M_\pm$ equipped with stable Hamiltonian structures, and allows for Morse-Bott orbits.} of contact 3-manifolds with $\Phi$ chosen as in Remark \ref{trivchoice}.

\begin{theorem}[Theorem 0, \cite{Wtrans}]\label{folk0}
Assume that $u\colon (\dot{\Sigma},j) \to (\R \times M, \Jt)$ is a non-constant curve in $\calm(\ga; \ga_1,...,\ga_s)$ asymptotic to nondegenerate orbits.  If $u$ is regular, then a neighborhood of $u$ in $\calm(\ga, \ga_1,...,\ga_s)$ naturally admits the structure of a smooth orbifold of dimension 
\[
{\mbox{\em ind}}(u) =  -(1-s) + \czm(\gamma) - \displaystyle \sum_{i=1}^s\czm(\gamma_i),
\] 
whose isotropy group at $u$ is given by
\[
\mbox{\em Aut}:=\{ \varphi \in \mbox{\em Aut}(\dot{\Sigma}, j) \ | \ u = u \circ \varphi \}.
\] 
Moreover, there is a natural isomorphism
\[
T_u\calm(\ga; \ga_1,...,\ga_s) = \ker D\deebar(j,u)/\mathfrak{aut}(\dot{\Sigma},j).
\]
\end{theorem}
\noindent In particular, if regularity can be achieved when $u$ is somewhere injective then $\calm$ is a manifold near $u$.  However, when $u$ is multiply covered the isotropy group for an orbifold singularity has order bounded by the covering number of $u$. 

Before we can state Wendl's criterion for automatic transversality, we must review a few numbers that encode certain topological and geometric data.  The subset $\g_0(u) \subset \g$ consists of the \textbf{punctures for which the asymptotic orbit has even Conley-Zehnder index}\footnote{This is in the case that the orbits are nondegenerate. If the orbits are Morse-Bott, the definition is more complicated; see \cite{Wtrans}.  }.   A related quantity is $\g_1(u):=\g \setminus \g_0(u)$, which consists of the \textbf{punctures for which the asymptotic orbit has odd Conley-Zehnder index}.

The \textbf{normal first Chern number}, $c_N(u;c)$, is a half integer, which when $\Sigma$ is closed and of genus 0 is given by
\begin{equation}\label{normalchern}
2c_N(u)={\mbox{ind}}(u) - 2  + \#\g_0(u).
\end{equation}

If $\Sigma$ is closed and there are no punctures then (\ref{w-index}) and (\ref{normalchern}) yields 
\[
c_N(u)=c_1(u^*TW) - \chi(\Sigma).
\]
Thus if $u$ is immersed then $c_N(u;c)$ is the first Chern number of the normal bundle.  

We also need to encode the \textbf{total order of critical points} of a curve $u: \ds \to W$, as follows.  Since a non-constant curve $u$ is necessarily immersed near the ends, it can have at most finitely many critical points.  The bundle 
\[
u^*TW  \to \ds
\]
admits a natural holomorphic structure such that the section
\[
du \in \g(\mbox{Hom}_{\C}(T\ds, u^*TW))
\]
is holomorphic.  Thus its critical points are isolated and have positive order, which we will denvote by $\mbox{ord}(du;z)$ for any $z \in \mbox{Crit}(u).$ This yields the desired quantity,
\begin{equation}\label{w-Z}
Z(du):= \sum_{z\in du^{-1}(0)\cap \mbox{\tiny int}(\ds)}\mbox{ord}(du;z), 
\end{equation}
an integer because $\Sigma$ is closed. 

\begin{remark}\em
Note that $Z(du)=0$ if and only if $u$ is immersed. 
\end{remark} 

In later applications we will often need $Z(du)=0$.  To better understand which somewhere injective curves are immersed, we state the following simple version of the folk theorem which states that generically, spaces of \curves \ with more than a minimum number of critical points have positive codimension.

\begin{proposition}[Corollary 3.17, \cite{Wtrans}]\label{Z}
For generic $\A$-compatible $\Jt$, all somewhere injective curves $u \in \calm$ satisfy
\[
2Z(du) \leq {\mbox{\em ind}}(u).
\]
\end{proposition}
Hence generically if $u$ is somewhere injective with $\ind(u)\leq1$ then $u$ is immersed.  
The following description of a multiply covered curve is relevant to the statement of the following lemmas.  Let $u: (\dot{\Sigma}, j) \to (W,J)$ be a multiple cover of a simple curve $v:(\ds', j') \to (W,J)$.  Then by Theorem \ref{multiplycovered} there exists a holomorphic branched covering $\varphi: (\Sigma, j) \to (\Sigma', j')$ with $\g' = \varphi(\g)$, $deg(\varphi)\geq1$, and a simple \curve \ such that
\[
u = v \circ \varphi. 
\]

\begin{definition}\label{unbanched} \em
We say that $u$ is an \textbf{unbranched} cover of an asymptotically cylindrical pseudoholomorphic curve whenever $u$ can be expressed as 
\[
u = v \circ \varphi. 
\]
where $v$ is a simple asymptotically cylindrical curve and $\varphi: (\Sigma, j) \to (\Sigma', j')$ is a holomorphic covering with no branch points.  
\end{definition}

\begin{lemma}\label{remZ}
For generic $\A$-compatible $J$, if $u$ is a multiple cover of an immersed finite energy cylinder and $\ind(u) \leq 2$ then 
\[
Z(du) = Z(d\varphi).
\]
\end{lemma}

\begin{proof}
We can write any $u$ as the composition $ u=v \circ \varphi,$ where $v$ is an somewhere injective cylinder and $\varphi$ is a holomorphic covering of the source of $u$.  Since $v$ is immersed we have that $Z(dv)=0$. Then it follows from the chain rule that the critical points of $u$ can only arise from branch points of the cover $\varphi$, hence $Z(du)=Z(d\varphi)$.

\end{proof}

Putting this together we obtain the following lemma.

\begin{lemma}\label{Z2}
For generic $\A$-compatible $\Jt$, any unbranched cylinder $u$ of index $\leq 2$ in a symplectization satisfies 
\[
Z(du) =0.
\]
\end{lemma}
\begin{proof}
If $u$ is somewhere injective this follows from Theorem \ref{HT2thm}.   If $u$ is not somewhere injective, then $u$ is the unbranched multiple cover of an immersed cylinder $v$.  After we show that $\ind(u) \leq 2$ forces $\ind(v) \leq 2$ we obtain the desired result.  This is because any unbranched cover of a cylinder has no branch points, hence by Lemma \ref{remZ}, $Z(du)=Z(d\varphi)=0.$

We proceed with a proof by contradiction that $\ind(u) \leq 2$ forces $\ind(v) \leq 2$. Let $u$ be a $k$-fold unbranched covering of $v \in \calm(\gp;\gm)$.   Proposition \ref{almostlinear} states that for any closed Reeb orbit $\gamma$ of a nondegenerate contact 3-manifold the following inequality holds involving $\ga^k$, the $k$-fold cover of $\gamma: $
 \begin{equation}\label{czlove1}
 k\czm(\ga) - k+1 \leq \czm(\ga^k) \leq k \czm(\ga)+k-1.
 \end{equation}

Assume that
\[
\ind(v) = \czm(\gp)-\czm(\gm) > 2.
\]
In combination with (\ref{czlove1}) we compute
\[
\begin{array}{lcl}
\ind(u) &=& \czm(\gp^k)-\czm(\gm^k) \\
&\geq & k \left( \czm(\gp) - \czm(\gm) \right) - 2k + 2\\
&>& 2,\\
\end{array}
\]
a contradiction.  Thus $\ind(u)\leq 2$ forces $\ind(v)\leq 2$, hence $v$ is immersed.

\end{proof}

Lastly we define for given constants $r \in \R$ and $G \geq 0$ the nonnegative integer,
\begin{equation}
K(r,G)= \mbox{min} \{ k + \ell \ | \ k \in \Z_{\geq 0},  \ \ell \in 2\Z_{\geq 0}, \ k \leq G, \mbox{ and } 2k + \ell > 2r \}.
\end{equation}
In most of our applications, it turns out that $r<0$, so $K(r,G)=0$. With everything in place we can state the main automatic transversality result.  In the simple case of a symplectization equipped with a $\A$-compatible almost complex structure, full details of this proof are provided in Section 4.1 of \cite{HNdyn}.  This simpler automatic transversality result is the one used throughout the paper.

\begin{theorem}[Theorem 1 \cite{Wtrans}]\label{Wtrans1}
Suppose that $\mbox{\em dim }W = 4$ and $u \in \calm(\ga; \ga_1,...\ga_s)$ is a non-constant curve asymptotic to nondegenerate orbits.  If
\begin{equation}\label{auttranseq}
{\mbox{\em ind}}(u) > c_N(u)  + Z(du),
\end{equation}
then $u$ is regular.  Moreover, when this condition is not satisfied, we have the following bounds on the dimension of $\ker D\deebar(j,u)$.  If ${\mbox{\em ind}}(u) \leq 2Z(du)$, then
\[
\begin{array}{lcl}
2Z(du) &\leq& \dim ( \ker D\deebar (j,u))/\mathfrak{aut}(\dot{\Sigma},j) \\
& \leq & 2Z(du) + K\left(c_N(u) - Z(du), \# \g_0(u)\right), \\
\end{array}
\]
and if $2Z(du) \leq {\mbox{\em ind}}(u)$, then
\[
\begin{array}{lcl}
 {\mbox{\em ind}}(u)& \leq & \dim (\ker D\deebar(j,u))/\mathfrak{aut}(\dot{\Sigma},j)\\
 &\leq&   {\mbox{\em ind}}(u) +  K\left(c_N(u) +Z(du) - {\mbox{\em ind}}(u), \# \g_0(u)\right).\\
\end{array}
\]
\end{theorem}
\begin{remark}\label{indexremark}
\em
If we plug in the first Chern number of the normal bundle, $c_N(u)$, into the index formula, then condition (\ref{auttranseq}) is equivalent to
\begin{equation}\label{transcyl}
{\mbox{ind}}(u) > 2g + \# \g_0(u)  -2 + 2Z(du),
\end{equation}
or
\[
 2c_1^\Phi(u^*TW) + \mu^\Phi(u) + \# \g_1(u) > 2Z(du),
\]
where $\g_1(u):=\g \setminus \g_0(u)$ consists of the punctures for which the asymptotic orbit has odd Conley-Zehnder index.  
\end{remark}

\begin{remark}
\em
There is an important special case of the dimension bound which we will use in the applications.  Namely, if $c_N(u) < Z(du)$, then 
\[
K\left(c_N(u) - Z(du), \#\g_0(u)\right)=0,
\]
 and $\dim \ker (D\deebar(u))$ becomes $2Z(du)$, which is its smallest possible size.
\end{remark}

\section{Requisite regularity results}\label{regularity}
In the first half of this section we prove regularity results for unbranched multiply covered nontrivial cylinders of index $\leq 2$ in a symplectization of a contact 3-manifold provided the $\A$-compatible almost complex structure has been chosen generically.  This is proven by making use of index calculations and results of \cite{HWZ3, HT2, Wtrans}.

The second half of this section is devoted to demonstrating that moduli spaces of (branched) multiply covered curves, which would obstruct $\pa_\pm^2=0$, are empty in the symplectization of a dynamically separated contact form.   These results are obtained via index calculations, allowing us to complete the proof of Theorem \ref{conditions}.  Moreover, we show that the only obstruction to defining a cylindrical contact chain complex for the class of dynamically convex contact forms on contact 3-manifolds is due to the presence of a specific type of branched cover of a trivial cylinder.  The problematic branch covers are precisely stated in Lemma \ref{nontrivialtentacles}.  

Together these results comprise Theorems \ref{conditions} and \ref{conditions2}, Conditions (A)-(B) and (D).

 \subsection{Index calculations for arbitrary nondegenerate contact 3-manifolds}\label{arbitrary}
 In this section we provide  Conley-Zehnder index calculations for arbitrary nondegenerate contact 3-manifolds.  We briefly recall the classification of  nondegenerate Reeb orbits in dimension 3, as it is relevant to the behavior of the Conley-Zehnder index.  We note that many of the results we obtain in this subsection were proven using alternate methods in Section 2.3 of \cite{Mo}.
 
  Let $(M, \A)$ be a nondegenerate contact 3-manifold with $\xi:=\ker \A$ and $\{\varphi_t\}_{t \in \R}$ be the one-parameter group of diffeomorphisms of $M$ given by the flow of the Reeb vector field $R_\A$. Let $\gamma: \R/T\Z \to M$ be a $T$-periodic Reeb orbit, e.g. a solution to $\gamma'(t) = R_\A(\gamma(t))$ modulo reparametrization. The linearized flow
\[
d\varphi_t:T_{\ga(0)} M \to T_{\ga(t)}M,
\]
induces the symplectic linear map
\begin{equation}\label{slm}
\phi_t : \xi_{\ga(0)} \to \xi_{\ga(t)}.
\end{equation}
Using a trivialization of $\xi$ we can regard (\ref{slm}) as a $2 \times 2$ symplectic linear matrix. By construction, $\phi_0 = \id$, and since $\A$ is nondegenerate, $\phi_T \neq \id$. We call $\phi_T$ the \textbf{linearized return map} along $\gamma$ restricted to $\xi$.   A nondegenerate Reeb orbit $\gamma$ is said to be one of three types, depending on its \textbf{Floquet multipliers}, which are defined to be the eigenvalues $\lambda$, $\lambda^{-1}$ of the linearized return map $\phi_T$: 
\begin{itemize}
\item[] $\gamma$ is elliptic if $\lambda, \lambda^{-1} := e^{\pm2\pi i \theta}$;
\item[] $\gamma$ is positive hyperbolic if  $\lambda, \lambda^{-1} > 0$; 
\item[] $\gamma$ is negative hyperbolic if $\lambda, \lambda^{-1} < 0$.
\end{itemize}

Let $\ga$ be a simple closed Reeb orbit and denote by $\ga^k$ its $k$-th iterate; note that the linearized return map associated to $\ga^k$ is $\phi_{kT}$.  The following proposition gives the properties of the Conley-Zehnder index associated to nondegenerate Reeb orbits.  These formulae are proven in Section 8.1 of \cite{Ylong} for paths in $\mbox{Sp}(2)$; see also the discussion in Section 3.2 of \cite{Hu2}.  
  
     
\begin{proposition}\label{CZs}
Let $(M,\A)$ be a nondegenerate contact 3-manifold.  For any trivialization $\Phi$ the following formula for the Conley Zehnder index holds
\[
\mu_{CZ}^\Phi = \lfloor \theta \rfloor + \lceil \theta \rceil,  
\]
where $\theta$ denotes the \textbf{rotation number} of $\gamma$ with respect to $\Phi$.  
\begin{itemize}
\item[] If $\gamma$ is elliptic then $\theta$ is an irrational number.
\item[] If $\gamma$ is hyperbolic then $\theta$ is the number of times  that the eigenspaces of the linearized return map rotate with respect to $\Phi$ as one goes around $\gamma$.  
\begin{itemize} 
\item If $\gamma$ is positive hyperbolic $\theta \in \Z$.
\item If $\gamma$ is negative hyperbolic $\theta + \dfrac{1}{2} \in \Z$.  
\end{itemize}
\end{itemize}
Changing the trivialization $\Phi$ will shift the rotation number $\theta$ by an integer.  If $k$ is a positive integer and if $\gamma^m$ denotes the Reeb orbit that is the $k$-fold cover of $\gamma$ then 
\[
\mu_{CZ}^\Phi (\gamma^k) =  \lfloor k \theta \rfloor + \lceil k \theta \rceil .
\]
\end{proposition}
Before proceeding with index calculations, we make a few remarks about trvializations  in connection with the Conely-Zehnder index iteration formulas.  


\begin{remark}\label{trivtpy}\em
 The choice of trivialization in the above proposition is implicitly one such that the trivialization over an iterated orbit is homotopic to the `iterated' trivialization.  Namely, the $d\A$-symplectic trivialization $\Phi$ of $\xi := \ker \A$ along $\gamma^k$,  which was used to define $ \mu_{CZ}^\Phi(\gamma^k)$ must be homotopic to the ``$k$-iterated'' trivialization of $\xi$ along $\gamma$, which was used to compute $ \mu_{CZ}^\Phi(\gamma)$.   This is indeed always the case when $c_1(\xi;\Z)=0$, otherwise some care must be taken to obtain a trivialization fixed up to homotopy. 
\end{remark}

After fixing trivializations up to homotopy as in the above Remark, we will want to ``normalize'' them in later Fredholm index calculations as follows.

\begin{remark}\label{vanishingtriv}\em
In our index calculations we will choose a trivialization $\Phi$ so that $c_1^\Phi(v^*T(\R \times M))=0$ for a somewhere injective curve genus 0 asymptotically cylindrical curve $v$ with one positive puncture and at least one negative puncture.  This implies for any (branched) cover  $u := \varphi \circ v$,  that $c_1^\Phi(u^*T(\R \times M))=0$. We fix such a trivialization $\Phi$ and write $\czm$ as a shorthand for $\mu_{CZ}^\Phi$.  
\end{remark}
  
Without loss of generality, after fixing the trivialization up to homotopy as in Remark \ref{trivtpy} and selecting one such that Remark \ref{vanishingtriv} we can write the iteration formulas for the Conley-Zehnder index as follows.

 \begin{list}{\labelitemi}{\leftmargin=2em }
\item[\textbf{Elliptic case:}]
Take $\Phi$ to be a trivialization such that each $\phi_t$ is rotation by the angle $2 \pi \vartheta_t \in \R$ where $\vartheta_t$ is a continuous function of $t \in [0, T]$ satisfying $\vartheta_0=0$ and $\vartheta:=\vartheta_T \in \R \setminus \Z $.  The number $\vartheta \in \R \setminus \Z $ the \textbf{rotation angle} of $\gamma$ with respect to the trivialization and
\[
\mu_{CZ}(\ga^k) = 2 \lfloor k \vartheta \rfloor + 1. 
\]
 \item[ \textbf{Hyperbolic case:}] 
 Let $v \in \R^2$ be an eigenvector of $\phi_T$. Then for any trivialization used, the family of vectors $\{ \phi_t(v) \}_{t \in [0,T]}$, rotates through angle $\pi r$ for some integer $r$.  The integer $r$ is dependent on the choice of trivialization $\Phi$, but is always even in the positive hyperbolic case and odd in the negative hyperbolic case.  We obtain
 \[
 \mu_{CZ}(\ga^k) = k r.
 \]
 \end{list}

 The following proposition shows that in dimension 3, the Conley-Zehnder index grows almost linearly and will be used in Section \ref{numerology}.  It follows immediately by considering the above Conley-Zehnder index formulas
 in Proposition \ref{CZs}.  

 \begin{proposition}\label{almostlinear}
 Let $(M, \A)$ be a nondegenerate contact 3-manifold.  Let $\ga$ be any closed Reeb orbit of $R$ and $\ga^k$ its $k$-fold cover. Then
 \begin{equation}\label{czlove}
 k\czm(\ga) - k+1 \leq \czm(\ga^k) \leq k \czm(\ga)+k-1.
 \end{equation}
\end{proposition}
\begin{proof}
In the case that  $\ga$ is positive or negative hyperbolic this follows from the preceding lemma since the Conley-Zehnder index grows linearly,
\[
\czm(\ga^k) = k \czm(\ga).
\]
Thus the desired inequality, (\ref{czlove}), is trivially true.  In the case that $\ga$ is elliptic, the above inequality is more meaningful and follows directly from the formula
\[
\czm(\ga^k) = 2 \lfloor k \vartheta \rfloor + 1,
\]
where $\vartheta \in \R \setminus \Z$.  Since $\vartheta \in \R \setminus \Z$ we can write for $r \in \Z$ and $\theta \in (0,1)$
\[
\vartheta = r + \theta.
\]
Then basic properties of the floor function yield for any integer $k$,
\[
\lfloor k (r + \theta) \rfloor = kr + \lfloor k \theta \rfloor
\]
and
\begin{equation}\label{floorlove}
0 \leq \lfloor k \theta \rfloor \leq k-1 < k.
\end{equation}
We have
\[
k\czm(\ga) = 2kr+  k,
\]
thus
\[
k\czm(\ga) -k +1 =2kr+1
\]
and
\[
k\czm(\ga) +k -1=2kr+2k -1. 
\]
Appealing to (\ref{floorlove}) yields (\ref{czlove}) since
\[
2kr+ 1 \leq 2kr +2 \lfloor k \vartheta \rfloor + 1 \leq 2kr+ 2k-1,
\]
\end{proof}
We will use the almost linear behavior of the Conley-Zehnder index to prove the following result.

\begin{proposition}\label{genericJgood}
Let $(M,\A)$ be a nondegenerate contact 3-manifold and   $\gp$, $\gm$ be closed Reeb orbits. Then after a generic choice of $\Jt$ all nontrivial cylinders $u \in \calm(\gp; \gm)$ in $(\R \times M, \Jt)$ satisfy 
\begin{equation}\label{posu}
{\mbox{\em ind}}(u) \geq 1.
\end{equation}

\end{proposition}

\begin{proof}
We know that after generic choice of $\Jt$ all immersed somewhere injective cylinders have index $\geq 1$; see Corollary \ref{bigdim}.   Moreover any somewhere injective cylinder of index $\leq 2$ is immersed; see Theorem \ref{HT2thm}.  Let $u \in \calm (\ga_+^k; \ga_-^k)$ be a $k$-fold cover of some somewhere injective  $v\in \calm(\gp;\gm)$.  We will treat the cases when $\ind(v)\geq2$ and $\ind(v) =1$ separately. If $\ind(v) \geq 2$ then Proposition \ref{almostlinear} yields
\[
\begin{array}{lcl}
\ind(u)& =& \czm(\ga_+^k)- \czm(\ga_-^k) \\
&=& k\czm(\gp) - \czm(\gm) -2k + 2 \\
&\geq& 2. \\
\end{array}
\]
If $\ind(v)=1$ we can improve the estimate of Proposition \ref{almostlinear} as one of the orbits must be hyperbolic, yielding
\[
\ind(u) = k\czm(\gp) -k\czm(\gm)-k+1 \geq 1.
\]

 \end{proof}
 
 The next result provides the regularity needed for Conditions (A) and (B). 

 \begin{corollary}
 Let $(M,\A)$ be a nondegenerate contact 3-manifold.  Then after a generic choice of $\A$-compatible $\Jt$ all nontrivial cylinders of index $\leq 2$ in the symplectization $(\R \times M, d(e^\tau \A), \Jt)$  are regular, including unbranched multiply covered cylinders.
 \end{corollary}
 
 \begin{proof}
To ensure that transversality holds, by Theorem \ref{Wtrans1} it suffices to demonstrate that the following inequality (\ref{transcyl}) holds, 
\begin{equation}\label{eqnreg}
{\ind}(u) > \# \g_0(u) -2 + 2Z(du).
\end{equation}
If $\Jt$ has been chosen generically, then $Z(du)=0$ for all finite energy cylinders of index $\leq 2$ by Lemma \ref{Z2}.  The subset $\g_0(u) \subset \g$ consists of punctures for which the asymptotic orbit has even Conley-Zehnder index, and hence $\# \g_0(u) \leq 2$.  Thus 
\[
\# \g_0(u) -2 + 2Z(du) \leq 0.
\]
 By Proposition \ref{genericJgood}, all nontrivial cylinders $u$ satisfy
\[
{\ind}(u) \geq 1,
\]
thus (\ref{eqnreg}) holds.
 \end{proof}

\subsection{Numerics of branched covers}\label{numerology}   All but the last result of this section hold for dynamically convex contact forms, as opposed to the more restrictive class of dynamically separated contact forms found in Definition \ref{taut}.  To prove Theorem \ref{conditions2} we  exclude all multiply covered curves which would obstruct the construction of a cylindrical contact chain complex, as described in Section \ref{quandaries}.  This is accomplished by obtaining a lower bound on the index of a multiply covered curve, making use of the Riemann-Hurwitz theorem and Conley-Zehnder index calculations.

Recall that we denote $\gamma^\ell_{+}$ to be the $\ell$-fold cover of a simple orbit $\gp$ and $\gamma^d_{-}$ the $d$-fold cover of a simple orbit $\gm$; see Figures \ref{argf1} and \ref{argf2}.  Depending on the multiplicities of the orbits and existence of a covering map, the curve $u \in \calm (\glp;\gdm)$ may or may not be multiply covered.  An example of a branched cover is given in Figure \ref{argf3}

\begin{figure}[ht]
\begin{minipage}[b]{0.32\linewidth}
\centering
 \includegraphics[width=.38\linewidth]{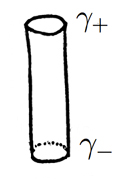}
\caption{\small A simple cylinder.}
\label{argf1}
\end{minipage}
\begin{minipage}[b]{0.33\linewidth}
\centering
\includegraphics[width=.43\linewidth]{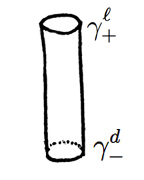}
\caption{$\mbox{\small The plot thickens.}$}
\label{argf2}
\end{minipage}
\begin{minipage}[b]{0.32\linewidth}
\centering
\includegraphics[width=.9\linewidth]{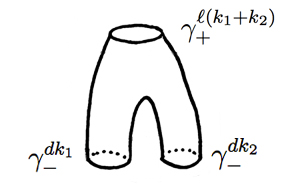}
\caption{$\mbox{\small A branched cover.}$}
\label{argf3}
\end{minipage}
\end{figure}

The definition of \textbf{dynamically convex} contact forms first appeared in 
\cite{HWZtight} and necessitates that the Conley-Zehnder index of contractible periodic orbits of the Reeb vector field be well-defined without any reference to a specific homotopy class of discs spanned by the orbits.  
If we only care about contractible loops admitting a well-defined Conley-Zehnder grading then for every map
\[
v: S^2 \to M
\]
the integer $c_1(v^*\xi)([S^2])$ must vanish.  The stipulation that $c_1(v^*\xi)([S^2])\equiv 0$ is equivalent to $\psi_\xi \equiv 0$, where $\psi_\xi$ is the natural homomorphism defined by
\begin{equation}\label{naturalhomo}
\begin{array}{crcl}
\psi_\xi: & \pi_2(M) & \to & \Z, \\
& [\sigma] & \mapsto& c_1(v^*\xi). \\
\end{array}
\end{equation}

\begin{definition}\label{dyncon}\em
Let $\A$ be a nondegenerate contact form associated to a closed 3-manifold $M$.  Assume that  the map $\psi_\xi =0$ from (\ref{naturalhomo}).  Then $\A$ is said to be \textbf{dynamically convex} whenever
\[
\czm(\ga) \geq 3
\]
for all contractible Reeb orbits $\ga$ of the Reeb vector field $R_\A$.
\end{definition}
Note that if we assume $c_1(\xi)=0$ then there is a well-defined Conley-Zehnder grading for all loops, modulo the choice of a complex volume form on $(\R \times M, \Jt)$.

Next we recall the Riemann-Hurwitz Theorem.
\begin{theorem}[Hartshorne, Corollary IV.2.4]\label{RH}
Let $\varphi:\widetilde{\dot{\Sigma}} \to \ds$ be a compact $k$-fold cover of the Riemann surface $\ds$.  Then
\[
\chi(\widetilde{\dot{\Sigma}}) =  k\chi(\ds) - \sum_{p \in\widetilde{\dot{\Sigma}}} (e(p)-1),
\]
where $e(p)-1$ is the ramification index of $\varphi$ at $p$.
\end{theorem}

We will use $b$ to keep track of the number of branch points counted with multiplicity:
\begin{equation}\label{b}
b:=\sum_{p \in \widetilde{\dot{\Sigma}}} (e(p)-1).
\end{equation}
At unbranched points $p$ we have $e(p)-1=0$, thus for any $q \in \ds$, 
\[
\sum_{p \in \varphi^{-1}(q)}e(p)=k.
\]

\begin{remark}\em
The multiplicity of the Reeb orbits of the cover of an asymptotically cylindrical curve are determined by the monodromy, which follows from Section 2 \cite{branchcourse}  with the local behavior of a curve near its punctures in \cite{MiWh, Sief}.
\end{remark}


We obtain the following result.
\begin{proposition}\label{Hurwitztentacles}
Let $(M,\A)$ be a nondegenerate contact manifold and $J$ be a generic $\A$-compatible almost complex structure.  Let $u \in \calm(\gp; \ga_0, ... \ga_s)$ be a somewhere injective curve.  Then any genus zero $k$-fold cover $\widetilde{\calc}$ of $\calc$  with 1 positive puncture must have  $1+ks+b$ negative punctures and satisfies
\begin{equation}\label{RHtentacleseqn}
{\mbox{\em ind}}(\widetilde{\calc}) \geq 2-k + 2b.
\end{equation}
\end{proposition}

\begin{proof}
Recall that the index of the underlying curve $\calc$ is
\begin{equation}\label{underlying}
\begin{array}{lcl}
{\ind}(\calc) &=& s + \czm(\gp) - \displaystyle \sum_{i=0}^s\czm(\ga_i), \\
\end{array}
\end{equation}
and that Lemma \ref{almostlinear}  yields
\begin{equation}\label{iterateineq}
k \czm(\ga) - k +1 \leq \czm(\ga^k) \leq k \czm(\ga) + k -1.
\end{equation}

From the Riemann-Hurwitz Theorem if $\widetilde{u}$ has 1 positive puncture then it must have $1+ks+b$ negative punctures.  Let $\delta_0,...,\delta_{ks+b}$ denote the Reeb orbits at which $\widetilde{u}$ has negative ends; these are covers of $\ga_0,...,\ga_s$. Moreover,
\begin{equation}\label{waytogo}
\sum_{i=0}^{ks+b}\czm(\delta_i) \leq k \sum_{i=0}^s \czm(\ga_i) + (k(s+1)-(ks+b+1))
 \end{equation}
Then (\ref{iterateineq}) and (\ref{waytogo}) yield
\[
\begin{array}{lcl}
{\ind}(\widetilde{\calc}) &=& ks+b+ \czm(\ga_+^k) - \displaystyle \sum_{i=0}^{ks+b}\czm(\delta_i) \\
& \geq & ks  + b + (\displaystyle k\czm(\gp) -k +1) - k \sum_{i= 0}^s \czm(\ga_i) - k +b +1  \\
& = & k\left(  s + \czm(\gp) -\displaystyle \sum_{i=0}^s\czm(\ga_i)\right) -2k + 2b + 2 \\
&=& k \cdot {\ind}(\calc) -2k + 2b +2 \\
& \geq & 2- k + 2b. \\
\end{array}
\]
Note that we obtain that ${\ind}(\calc) \geq 1$, since $\calc$ is somewhere injective with either $\ind(\calc) >1$ or $\ind(\calc) \leq 1$.  In the latter case the results of Proposition \ref{Z} allow us to conclude that $\calc$ is actually immersed, thus the results of \cite{HWZ3} hold given the genericity of $\Jt$. Thus (\ref{RHtentacleseqn}) is obtained as desired. 
\end{proof}

Before the reader worries that in some cases $2-k +2b \leq 0$, we note that this is not problematic because we will cap off $ks+b$ ends, each of which have index $\geq 2$; precise arguments appear in a subsequent series of lemmas.

Next we improve the preceding result when the underlying curve is a cylinder.

\begin{proposition}\label{cyl}
Let $(M,\A)$ be a nondegenerate contact manifold and $J$ be a generic $\A$-compatible almost complex structure.  Let $\calc \in \calm(\gp; \gm)$ be a nontrivial cylinder.  Then any genus zero branched $k$-fold cover $\widetilde{\calc}$ of $\calc$  with 1 positive puncture must be an element of $\calm(\ga_+^k;\ga_-^{k_1},...\ga_-^{k_n})$ where $k:=k_1+...+k_n$.  Moreover,
\begin{itemize}
\item[\em (i)]if ${\mbox{\em ind}}(\calc) \geq 2$ then ${\mbox{\em ind}}(\widetilde{u}) \geq 2n$;
\item[\em (ii)] if ${\mbox{\em ind}}(\calc) = 1$ with $\gp$ hyperbolic then ${\mbox{\em ind}}(\widetilde{u}) \geq 2n-1$;
\item[\em (iii)] if ${\mbox{\em ind}}(\calc) = 1$ with $\gm$ hyperbolic then ${\mbox{\em ind}}(\widetilde{u}) \geq n$.
\end{itemize}
\end{proposition}

\begin{proof}
\textbf{(i): Let $\ind(\calc)\geq2$.} Then \[
\begin{array}{lcl}
\ind(\widetilde{u}) &= &n-1 + \czm(\ga^k) - \displaystyle \sum_{i=1}^n \ga_-^{k_i} \\
 &\geq& n-1 + (\displaystyle k\czm(\gp) -k +1) - \left(k \sum_{i= 0}^s \czm(\ga_-) \right) - k + n,  \mbox{ by (\ref{iterateineq}) } \\
 &=& 2n-2k + k(\czm(\gp) - \czm(\ga_-)) \\
 &\geq& 2n. \\
\end{array}
\]
If $\ind(\calc)=1$ than either $\gp$ or $\gm$ must be positive hyperbolic and we can sharpen (\ref{iterateineq}) from Lemma \ref{almostlinear} to obtain the desired results in (ii) and (iii). 

\noindent \textbf{(ii): The positive end $\gp$ is hyperbolic.} Then $\czm(\ga^k_+) = k\czm(\gp),$ yields
\vspace{-.3cm}
\[
\begin{array}{lcl}
\ind(\widetilde{u}) &= & n-1 + k\czm(\gp) - \displaystyle \sum_{i=1}^n \czm(\ga_-^{k_i}) \\
&\geq& n -1 + k\czm(\gp) - \displaystyle \sum_{i=1}^n \left( k_i\czm(\gm) +k_i -1 \right),   \mbox{ by (\ref{iterateineq}) } \\
&=&2n-1+k(\czm(\gp)-\czm(\gm)) -k \\
&\geq& 2n-1. \\
\end{array}
\]
\noindent \textbf{(iii): The negative end $\gm$ is hyperbolic.} For all $i$, $\czm(\ga_-^{k_i}) = k_i\czm(\gm)$, thus
\vspace{-.3cm}
\[
\begin{array}{lcl}
\ind(\widetilde{u}) &= & n-1 + \czm(\ga_+^k) - \displaystyle \sum_{i=1}^n k_i\czm(\gm) \\
&\geq& n -1 + k\czm(\gp) - k +1 - k\czm(\gm),   \mbox{ by (\ref{iterateineq})} \\ 
&=&n + k(\czm(\gp)-\czm(\gm)) - k \\
&\geq & n. \\
\end{array}
\]
\end{proof}

Next we consider multiply covered trivial cylinders, which appears as Lemma 1.7 of \cite{HT1}, but for the sake of completeness we provide its quick proof. 

\begin{proposition}\label{trivcyl}
Let $(M,\A)$ be any nondegenerate contact 3-manifold with $\Jt$ a generic $\A$-compatible almost complex structure.  Let $\calc \in \calm(\ga; \ga)$ be a trivial cylinder.  Then any genus zero  $k$-fold cover $\widetilde{\calc}$ of $\calc$  with 1 positive puncture must either be an element of $\calm(\ga^k;\ga^k)$ or $\calm(\ga^k;\ga^{k_1},...\ga^{k_n})$ where $k:=k_1+...+k_n$.  In the former case we have 
\[
{\mbox{\em ind}}(\widetilde{u})=0 .
\]
In the latter case when $\widetilde{u} \in \calm(\ga^k;\ga^{k_1},...\ga^{k_n})$ we have
\begin{equation}\label{RHtrivial}
{\mbox{\em ind}}(\widetilde{\calc}) \geq 0.
\end{equation}
\end{proposition}

\begin{proof}
When $\calm(\ga^k;\ga^k)$ the result holds trivially.  When $\widetilde{u} \in \calm(\ga^k;\ga^{k_1},...\ga^{k_n})$ where $k:=k_1+...+k_n$ the result follows from the formulas of the Conley-Zehnder index of hyperbolic and elliptic orbits given in Proposition \ref{CZs}.  

If $\ga$ is hyperbolic then (\ref{RHtrivial}) holds trivially because the Conley-Zehnder index of $\ga$ increases linearly under iteration, thus
\[
\begin{array}{rcl}
\ind(\widetilde{u})&=&n-1 + \czm(\ga^k) - \displaystyle \sum_{i=1}^n\czm(\ga^{k_i}) \\
&=& n-1 + k\czm(\ga) -k\czm(\ga) \\
&=&n-1\\
& \geq & 1, \mbox{ since $n>1$.}
\end{array}
\]

If $\ga$ is elliptic with rotation angle $\vartheta$ then $\czm(\ga^k) = 2 \lfloor k \vartheta \rfloor + 1$, thus 
\[
\begin{array}{rcl}
\ind(\widetilde{u})&=&n-1 + \czm(\ga^k) - \displaystyle \sum_{i=1}^n \czm(\ga^{k_i}) \\
&=& n-1 +  2 \lfloor k \vartheta \rfloor + 1 -  \displaystyle \sum_{i=1}^n (2 \lfloor k_i \vartheta \rfloor - 1) \\
& \geq &   0. \\
\end{array}
\]
\end{proof}

The final step in the proof of the main result, Theorem \ref{conditions2}, is the following series of four inductive lemmas utilizing the above numerics. These results will allow us to exclude such complicated compactifications as in Figure \ref{egads}.  Before preceding, we recall the definition of a pseudoholomorphic building from \cite{BEHWZ}, which we adapt to our setting in which all curves and their limits are non-nodal and unmarked.

 \begin{figure}[h!]
  \centering
    \includegraphics[scale=2]{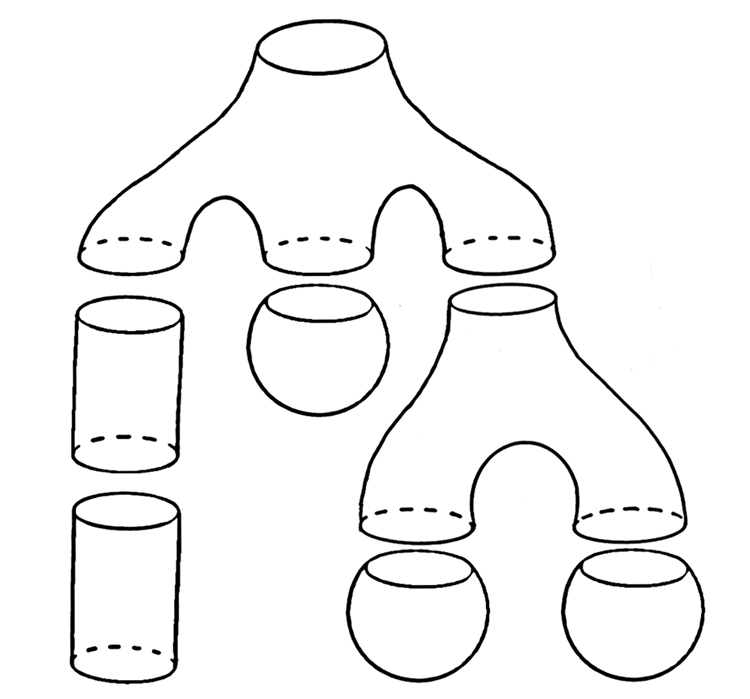}
\caption{A complicated element of $\mbar(\gp;\gm)$ best avoided.}
\label{egads}
\end{figure}

\begin{definition}\label{building} \em
We will call any asymptotically cylindrical curve $u_i=[\Sigma_i, j_i, \g_i:=\g^+_i \sqcup \g^+_i, u_i]$, with $\Sigma_i$ a possibly disconnected, \textbf{height-1 non-nodal building}, or height-1 building for short.  
Assuming there are bijections $\Phi_i: \to \g_i^- \to \g_{i+1}^+$ between the negative punctures of one curve and the positive punctures of the curve next in the sequence, a \textbf{height-$k$ non-nodal building} is defined to consist of a sequence $(u_1,...u_m)$ of $m$ height-1, non-nodal buildings and the collection $(\Phi_1,...\Phi_{m-1})$, provided the punctures identified via $\Phi_i$ have the same asymptotic limit. 
\end{definition}

\begin{remark}\label{nontrivbuilding}\em
Throughout the following lemmas we assume that each level $u_i$ of the building $(u_1,..,u_m)$  contains at least one nontrivial component, i.e. a component which is neither a trivial cylinder nor a constant map.   
\end{remark}

\begin{lemma}\label{lemma1}
Let $\calb:=(u_1,...u_m)$ be a genus 0 building with 1 positive contractible end, no negative ends, associated to the symplectization of a contact 3-manifold equipped with a nondegenerate dynamically convex contact form and generic $\A$-compatible $\Jt$.  Then 
\[
{\mbox{\em ind}}(\calb) \geq 2,
\]
and if equality holds, $\calb$ consists solely of one pseudoholomorphic plane.
\end{lemma}


\begin{proof}
This proof will be done via induction on the number of levels of $\calb:=(u_1,...,u_m)$, where the $u_i$ are levels of $\calb$ in decreasing order, e.g. $u_1$ is the top level.  For any $\calb$ with one positive end asymptotic to the Reeb orbit $\gp$ and no negative ends, 
\begin{equation}\label{helpfuless}
\ind(\calb) = \czm(\gp) -1.
\end{equation}

If $\calb$ consists of only one level then we are done since $\czm(\gp)\geq3$ by the dynamically convex hypothesis.

Suppose $m>1$ and that Lemma is true for buildings of height $m-1$.  We need to show that $\ind(\calb) > 2$. The building $(u_2,...u_m)$ is the disjoint union of $\ell$ genus 0 buildings, each having one positive end at each of the negative ends of $u_1$ and no negative ends.  By the inductive hypothesis we have
\[
\ind(\calb) \geq \ind(u_1) + 2\ell.
\]
Thus we must show that
\begin{equation}\label{lemma1ineq}
\ind(u_1) + 2\ell \geq 3.  
\end{equation}
If $u_1$ is somewhere injective then either $\ind(u_1) \geq 2$ or $u_1$ is immersed\footnote{Any somewhere injective curve $\calc$ of index $\leq 1$ is automatically immersed from Proposition \ref{Z} thus by Corollary \ref{bigdim} we know that $\ind(\calc) \geq 1$, since $\Jt$ has been chosen generically}.

If $u_1$ is a branched cover of a trivial cylinder then Proposition \ref{trivcyl} yields $\ind(u_1) \geq 0$ with $\ell >1$, thus $\ind(u_1) + 2\ell \geq 4$.

If $u_1$ is a cover of a nontrivial cylinder with $\ell$ negative punctures then Proposition \ref{cyl} yields $\ind(u_1) \geq \ell$; thus $\ind(u_1) + 2\ell \geq 3$.

If $u_1$ is a $k$-fold cover  of a somewhere injective curve $\calc \in \calm(\gp;\ga_0,...\ga_s)$ with $s\geq 1$ and $b$ branch points\footnote{Note $b$ could be 0 since the result holds if $u_1$ doesn't have any branch points.} counted with multiplicity then Proposition \ref{Hurwitztentacles} yields
\[
\ind(u_1) \geq 2-k + 2b,
\]
with $\ell=1+ ks +b$, the number of negative punctures.  Note that $\ell \geq 2,$ as $b \geq 1$.    Thus $\ind(u_1) + 2\ell\geq 4+4b+k(2s-1) > 2$.

Thus in all cases (\ref{lemma1ineq}) is satisfied.
\end{proof}

Building upon this theme we continue with the following lemma.
\begin{lemma}\label{lemma2}
Let $\calb:=(u_1,...u_m)$ be a  genus 0 building with 1 positive end and 1 negative end associated to the symplectization of a contact 3-manifold equipped with a nondegenerate dynamically convex contact form and a generic $\A$-compatible $\Jt$. Then 
\[
{\mbox{\em ind}}(\calb) \geq 1,
\]
and if equality holds, $\calb$ consists solely of one cylinder.
\end{lemma}
\begin{proof}
As before the proof will be done via induction on the number of levels of the building $\calb:=(u_1,...,u_m)$, where the $u_i$ are levels of $\calb$ in decreasing order, e.g. $u_1$ is the top level.  If $\calb$ consists of only one level then the proof is complete by Theorem \ref{conditions} in light of Remark \ref{nontrivbuilding}.

Suppose there is more than one level.   Call the top level $u_1$ and assume that the lemma is true for buildings of height $m-1$.  We need to show that $\ind(\calb) > 1$. The building $(u_2,...u_m)$ is the disjoint union of $\ell$ genus 0 buildings, each consisting of one positive end at each of the negative ends of $u_1$ and all but one, say $\calb_1$ having no negative ends.  This exceptional building, $\calb_1$, has one positive and one negative end.  By the inductive hypothesis and Lemma \ref{lemma1} we have
\[
\ind(\calb) \geq \ind(u_1) + \ind(\calb_1) + 2(\ell-1).
\]
Thus we must show that
\begin{equation}\label{lemma2ineq}
\ind(u_1) + 2(\ell-1) \geq 2,  
\end{equation}
provided $u_1$ is not itself a cylinder.   If $u_1$ is itself a cylinder then $\ind(u_1) \geq 1$ by Proposition \ref{genericJgood}.  Since there are no nontrivial cylinders of index 0, and because $\calb_1$ is not allowed to consist of a sequence of trivial cylinders the lemma holds.

If $u_1$ is the cover of a somewhere injective curve $\calc \in \calm(\gp;\ga_0,...\ga_s)$, where $s\geq 1$ then  by Lemma \ref{Hurwitztentacles}
\[
\ind(u_1) \geq 2-k + 2b,
\]
with $\ell = 1+ks+b$. Thus $\ind(u_1) + 2(\ell-1)  \geq 2 + k(2s-1) + 4b \geq 2$.

If $u_1$ is the cover of a nontrivial cylinder then Proposition \ref{cyl} yields $\ind(u_1) \geq n$ with $\ell =n$.   If $n \geq 2$ then $\ind(u_1) + 2(\ell-1) \geq 3n-1 \geq 2$.  Note if $n=1$ then $u_1$ is itself a cylinder, and this case has already been covered.

If $u_1$ is the cover of a trivial cylinder then Proposition \ref{trivcyl} yields $\ind(u_1)\geq 0$.  Since $u_1$ cannot consist solely of a trivial cylinder, $\ell=b\geq2$, thus $\ind(u_1) + 2(\ell-1) \geq 2$.  

Thus in all cases (\ref{lemma2ineq}) is satisfied.
\end{proof}

Putting the above two Lemmas together we obtain the following result for buildings in symplectizations of nondegenerate dynamically convex contact manifolds.
\begin{lemma}\label{nontrivialtentacles}
Let $\calb$ be a genus 0 building with 1 positive end and 1 negative end associated to the symplectization of a contact manifold equipped with a nondegenerate dynamically convex contact form and generic $J$.  If
\[
{\mbox{\em ind}}(\calb) =2,
\]
then $\calb$ is one of the following types,
\begin{enumerate}
\item[\em {(i)}] An unbroken cylinder of index 2;
\item[\em {(ii)}] A once broken cylinder given by a pair of cylinders, each of index 1, $(\calc_u,\calc_v) \in  \mhat(x, y) \times \mhat(y, z)$, where $\czm(y)=\czm(x)-1$;
\item[\em {(iii)}] A pair of pants in $\calm(\ga^{k_1+k_2}; \ga^{k_1}, \ga^{k_2})$ of index 0 and a holomorphic plane in $\calm(\ga^{k_i}; \emptyset)$ of index 2, for either $i=1$ or $i=2$.
\end{enumerate}
\end{lemma}

\begin{figure}[ht]
\begin{minipage}[b]{0.32\linewidth}
\centering
 \includegraphics[width=.45\linewidth]{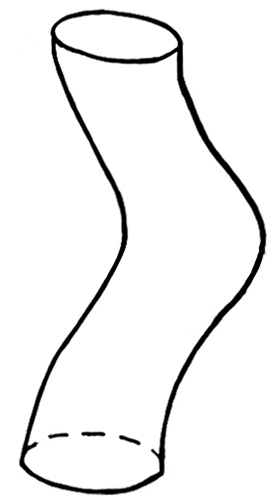}
\caption{$\calb \mbox{ of type (i)}$}
\label{f1}
\end{minipage}
\begin{minipage}[b]{0.32\linewidth}
\centering
\includegraphics[width=.65\linewidth]{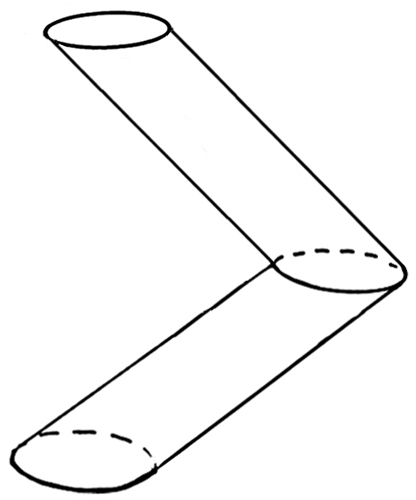}
\caption{$\calb \mbox{ of type (ii) }$}
\label{f2}
\end{minipage}
\begin{minipage}[b]{0.32\linewidth}
\centering
\includegraphics[width=.7\linewidth]{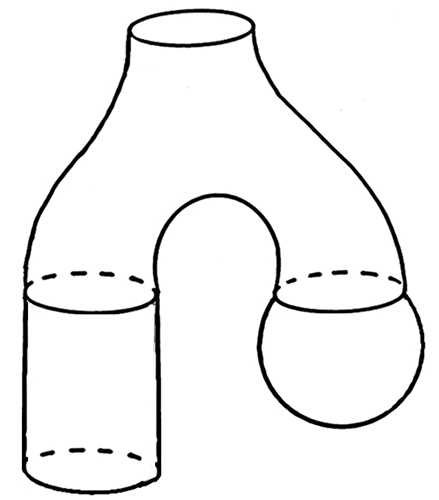}
\caption{$\calb \mbox{ of type (iii) }$}
\label{f3}
\end{minipage}
\end{figure}

\begin{remark}\em
$\calb$ of type (iii) is possible when $\ga$ is elliptic; see Example \ref{ellipsoid}.  
\end{remark}

\begin{proof}
From the numerics of Lemma \ref{lemma1} we saw that $\ind(\calb)=2$ if $\calb$ consists of one pseudoholomorphic plane of index 2. From Lemma \ref{lemma2} we saw that $\ind(\calb)=1$ if $\calb$ consists solely of an index 1 cylinder.  We also saw in Lemma \ref{lemma2} that  $\ind(\calb) =2$ whenever $\widetilde{u} \in \calm(\ga^{k_1+k_2};\ga^{k_1}, \ga^{k_2})$, with $\ind(\widetilde{u})=0$ and $\czm(\ga^{k_i})=3$ for either $i=1$ or $i=2$.  Since the index of a building is additive with respect to its components the results of Lemmas \ref{lemma1} and \ref{lemma2} imply that the only possible configurations for a building $\calb$ of index 2 are those described in (i)-(iii).
 \end{proof}

The final piece in completing the proof of Theorem \ref{conditions2} is to consider the virtual dimension of branched covers of trivial cylinders in the dynamically separated case.


\begin{lemma}\label{done}
Let $\calb$ be a genus 0 building with 1 positive end and 1 negative end associated to a nondgenerate dynamically separated contact form, with 
\[
{\mbox{\em ind}}(\calb) =2.
\]
Then $\calb$ is one of the following types,
\begin{enumerate}
\item[\em {(i)}] An unbroken cylinder of index 2
\item[\em {(ii)}] A once broken cylinder given by a pair of cylinders, each of index 1, $(\calc_u,\calc_v) \in  \mhat(x, y) \times \mhat(y, z)$, where $\czm(y)=\czm(x)-1$.
\end{enumerate}

\end{lemma}

\begin{proof}
To prove this lemma we need to exclude item (iii) of Lemma \ref{nontrivialtentacles}.  Because we are working with dynamically separated contact forms we will show for any $\widetilde{u} \in \mhat(\ga^{k_1+k_2}; \ga^{k_1}, \ga^{k_2})$ that $\ind(\widetilde{u}) \geq 2$ thus $\ind(\calb) \geq 5$.  Without loss of generality take $\ga^{k_2}$ to be the contractible orbit, thus $\ga^{k_1+k_2}$ and $\ga^{k_1}$ lie in the same free homotopy class. By the dynamically separated condition we have
\begin{equation}\label{dynsepwoo}
\czm(\ga^{k_1+k_2})-\czm(\ga^{k_1}) \geq 4.
\end{equation}
Thus for $\widetilde{u} \in \mhat(\ga^{k_1+k_2}; \ga^{k_1}, \ga^{k_2})$ we obtain
\begin{equation}\label{dynsepwoohoo}
\begin{array}{lcl}
\ind(\widetilde{u}) &=& 1 + \czm(\ga^{k_1+k_2}) - \czm(\ga^{k_1}) -\czm(\ga^{k_2}) \\
&\geq& 5 - \czm(\ga^{k_2}) \\
&\geq& 2. \\
\end{array}
\end{equation}
As $\ind(\calb)=\ind(\widetilde{u}) + \czm(\ga^{k-1}) \geq 5$ we obtain a contradiction.
\end{proof}

\begin{remark} \em
The proof of Lemma \ref{done} does not require the full definition of dynamically separated, that for all $i$, $\czm(\ga^{k_{i+1}(\baar)})=\czm(\ga^{k_{i}(\baar)})+4$.  However, this stipulation is needed to prove invariance and appears in \cite{jocompute}. 
\end{remark}

\noindent \textsc{Jo Nelson \\  Columbia University and the Institute for Advanced Study}\\
{\em email: }\texttt{nelson@math.columbia.edu}\\





\begin{thebibliography}{CFHWI}


\bibitem[ADfloer]{ADfloer}M. Audin and M. Damian, Th\'eorie de Morse et homologie de Floer. EDP Sciences, Les Ulis; CNRS ƒditions, Paris, 2010.


\bibitem[Bo02]{B02}F. Bourgeois, A Morse-Bott approach to contact homology. Ph.D. thesis, Stanford University, 2002.

\bibitem[Bo06]{Btorus} F. Bourgeois, Contact homology and homotopy groups of the space of contact structures.\emph{Math. Res. Lett.} \textbf{13} (2006), no. 1, 71-85. 

\bibitem[Bo09]{B09} F. Bourgeois,  A survey of contact homology. \emph{New perspectives and challenges in symplectic field theory}, 45-71, CRM Proc. Lecture Notes, \textbf{49}, \emph{Amer. Math. Soc.}, Providence, RI, 2009. 



\bibitem[BCE07]{BCE}F. Bourgeois, K. Cieliebak, T. Ekholm. A note on Reeb dynamics on the tight 3-sphere. \emph{J. Mod. Dyn.} \textbf{1}, no. 4, 597--613 (2007).

\bibitem[BC05]{BC} F. Bourgeois and V. Colin. Homologie de contact des vari\'et\'es toro\"idales. \emph{Geom. Topol.}, \textbf{9}: 299-313, 2005.

\bibitem[BEHWZ]{BEHWZ} F. Bourgeois, Y. Eliashberg, H. Hofer, K. Wysocki, E. Zehnder, Compactness results in symplectic field theory. \emph{Geometry and Topology} Vol. 7, 799-888 (2003).


\bibitem[BM04]{BM}F. Bourgeois and K. Mohnke,  Coherent orientations in symplectic field theory. \emph{Math. Z.} \textbf{248} (2004), no. 1, 123-146.



\bibitem[BO09a]{BOnotrans} F. Bourgeois and A. Oancea, An exact sequence for contact and symplectic homology. \emph{Invent. Math.} 175 (2009), no. 3, 611-680.

\bibitem[BO09b]{BOduke} F. Bourgeois and A. Oancea, Symplectic homology, autonomous Hamiltonians, and Morse-Bott moduli spaces, \emph{Duke Math. J.} 146 (2009), 71-174.

\bibitem[BO13]{BOSHCH} F. Bourgeois and A. Oancea, $S^1$-equivariant symplectic homology and linearized contact homology, arXiv:1212.3731.




\bibitem[CLW13]{CLW} B. Chen, A.-M. Li, and B.-L. Wang, Virtual neighborhood technique
for pseudo-holomorphic spheres, arXiv:1306.3276


\bibitem[CH05]{CH}V. Colin and K. Honda, Constructions contr\^ol\'ees de champs de Reeb et applications. Geom. Topol. 9 (2005), 2193-2226

\bibitem[CH13]{CHugh} V. Colin and K. Honda,  Reeb vector fields and open book decompositions. J. Eur. Math. Soc. (JEMS) 15 (2013), no. 2, 443-507. 


\bibitem[Dr04]{Dr} D. Dragnev, Fredholm theory and transversality for noncompact pseudoholomorphic maps in symplectizations. \emph{Comm. Pure Appl. Math.} \textbf{57}, 726-763, 2004.

\bibitem[EGH00]{EGH}Y. Eliashberg, A. Givental, and H. Hofer, 
Introduction to symplectic field theory. \emph{Geom. Funct. Anal.} (2000), Special Volume, Part II, 560-673, GAFA 2000.



\bibitem[Fl88]{Fgrad}A. Floer, The unregularized gradient flow of the symplectic action. \emph{Comm. Pure Appl. Math.}, \textbf{41}, 775-813, 1988.


 \bibitem[FH93]{FH} A. Floer and H. Hofer, Coherent orientations for periodic orbit problems in symplectic geometry. \emph{Math. Z.} \textbf{212} (1993), no. 1, 13-38.

\bibitem[FHS95]{FHS}A. Floer, H. Hofer, D. Salamon, Transversality in elliptic Morse theory for the symplectic action. \emph{Duke Math. J.} \textbf{80} (1995), no. 1, 251-292.

\bibitem[F$\mbox{O}^3$12]{confute} K. Fukaya, Y.-G. Oh, H. Ohta, and K. Ono, Technical detail on Kuranishi structure and Virtual Fundamental Chain, arXiv:1209.4410.


\bibitem[FO99]{FO}K. Fukaya and K. Ono, Arnold conjecture and Gromov-Witten invariant.\emph{Topology} \textbf{38} (1999), no. 5, 933-1048.

\bibitem[H93]{H1}H. Hofer, Pseudoholomorphic curves in symplectizations with applications to the Weinstein conjecture in dimension three. \emph{Invent. Math.,} \textbf{114} (1993), 515-563.


\bibitem[HK99]{HK} H. Hofer and M. Kriener, Holomorphic curves in contact dynamics. \emph{Proceedings of Symposia in Pure Mathematics} Vol. 66 (1999),
77-131.

\bibitem[HLS97]{HLS}H.  Hofer, V. Lizan, and J.-C. Sikorav,  On genericity for holomorphic curves in four-dimensional almost-complex manifolds. \emph{J. Geom. Anal.} \textbf{7} (1997), no. 1, 149-159.

\bibitem[HWZ99]{HWZtight}  H. Hofer, K. Wysocki, and E. Zehnder, A characterization of the tight 3-sphere. II. \emph{Comm. Pure Appl. Math.} \textbf{52} (1999), no. 9, 1139-1177. 


\bibitem[HWZI]{HWZ1} H. Hofer, K. Wysocki, and E. Zehnder, Properties of pseudoholomorphic curves in symplectisations I. Asymptotics, \emph{Ann. Inst. H. Poincar\'e Anal. Nonlin\'eaire} \textbf{13} (1996), no. 3, 337-379.

\bibitem[HWZII]{HWZ2}H. Hofer, K. Wysocki, and E. Zehnder. Properties of pseudo-holomorphic curves in symplectisations. II. Embedding controls and algebraic invariants. \emph{Geom. Funct. Anal.} \textbf{5} (1995), no. 2, 270-328.

\bibitem[HWZIII]{HWZ3} H. Hofer, K. Wysocki, and E. Zehnder, Properties of pseudoholomorphic curves in symplectizations III: Fredholm Theory, in \emph{Topics in Nonlinear Analysis}, Progr. Nonlinear Differential Equations Appl., Vol. 35, Birkh\"auser, 1999.

\bibitem[HWZIV]{HWZ4} H. Hofer, K. Wysocki, and E. Zehnder, Properties of pseudoholomorphic curves in symplectisation. IV. Asymptotics with degeneracies. \emph{Contact and symplectic geometry} (Cambridge, 1994), 78-117, Publ. Newton Inst., 8, Cambridge Univ. Press, Cambridge, 1996. 
\bibitem[HWZ02]{HWZsmallarea}H. Hofer, K. Wysocki, and E. Zehnder. Finite energy cylinders of small area. \emph{Ergodic Theory Dynam. Systems} \textbf{22} (2002), no. 5, 1451-1486. 

\bibitem[HWZ03]{HWZ03} H. Hofer, K. Wysocki, and E. Zehnder, Finite energy foliations of tight three-spheres and Hamiltonian Dynamics, \emph{Ann. of Math.}, \textbf{157}: 125-255, 2003.

\bibitem[HWZ10a]{HWZsc} H. Hofer, K. Wysocki, E. Zehnder, Sc-Smoothness, Retractions and New Models for Smooth Spaces, \emph{ Discrete Contin. Dyn. Syst.}, \textbf{28} (No 2), October 2010, 665-788.

\bibitem[HWZ10b]{HWZsec} H. Hofer, K. Wysocki, E. Zehnder, Integration Theory on the Zero Set of Polyfold Fredholm Sections, \emph{Math. Ann.} Vol 336, Issue 1 (2010), 139-198.

\bibitem[HWZgw]{HWZgw} H. Hofer, K. Wysocki, E. Zehnder, Applications of Polyfold Theory I: Gromov-Witten Theory, 2011, arXiv: 1107.2097

\bibitem[Hu09]{Huindex}M. Hutchings, The embedded contact homology index revisited. \emph{New perspectives and challenges in symplectic field theory}, 263-297, CRM Proc. Lecture Notes, 49, Amer. Math. Soc., Providence, RI, 2009. 

\bibitem[Hu10]{Hu} M. Hutchings, Taubes' proof of the Weinstein conjecture. \emph{Bull. of AMS} \textbf{47} (2010), 73-125.

\bibitem[Hu14]{Hu2} M. Hutchings, Lecture notes on embedded contact homology. \emph{Contact and Symplectic Topology}, 389-484, Bolyai Society Mathematical Studies, Vol. 26. Springer.  2014.

\bibitem[HN1]{HNdyn}M. Hutchings and J. Nelson, Cylindrical contact homology for dynamically convex contact forms in three dimensions.  To appear in \emph{J. Symplectic Geom.} arxiv:1407.2898

\bibitem[HN2]{HN2}M. Hutchings and J. Nelson, An integral lift of cylindrical contact homology for contact forms without contractible Reeb orbits.  In Preparation.

\bibitem[HN3]{HN3}M. Hutchings and J. Nelson, Invariance and an integral lift of cylindrical contact homology for dynamically convex contact forms.  In Preparation.

\bibitem[HTI]{HT1}M. Hutchings and C. Taubes, Gluing pseudoholomorphic curves along branched covered cylinders. I. \emph{J. Symplectic Geom.} 5 (2007), no. 1, 43-137.

\bibitem[HTII]{HT2}M. Hutchings and C. Taubes, Gluing pseudoholomorphic curves along branched covered cylinders. II. J. Symplectic Geom. 7 (2009), no. 1, 29-133.

\bibitem[IP]{IP} E. Ionel and T. Parker, A natural Gromov-Witten fundamental class, arXiv:1302.3472

\bibitem[IS99]{IS}S. Ivashkovich and V. Shevchishin, Structure of the moduli space in a neighborhood of a cusp-curve and meromorphic hulls. \emph{Invent. Math.} \textbf{136} (1999), no. 3, 571-602.

\bibitem[vK08]{vk}O. van Koert, Contact homology of Brieskorn manifolds. \emph{Forum Math.} \textbf{20} (2008), no. 2, 317-339.

\bibitem[La11]{LL}L. Lazzarini, Relative frames on J-holomorphic curves. \emph{J. Fixed Point Theory Appl. }9 (2011), no. 2, 213-256.

\bibitem[Lo02]{Ylong} Y. Long, \emph{Index Theory for Symplectic Paths with Applications}, Birkh\"auser (2002).

\bibitem[MNW13]{MNW} P. Massot, K. Niederkr\"uger, C. Wendl,  Weak and strong fillability of higher dimensional contact manifolds. \emph{Invent. Math.} \textbf{192} (2013), no. 2, 287-373.



\bibitem[MSintro]{MS1} D. McDuff and D. Salamon, 
\emph{Introduction to Symplectic Topology}. Oxford Univ. Press, 1995.

\bibitem[MSbig$J$]{MS2} D. McDuff and D. Salamon, 
\emph{$J$-holomorphic Curves and Symplectic Topology}, AMS
Colloquium Publications, 2004.

\bibitem[MW]{MW}D. McDuff and K. Wehrheim, Smooth Kuranishi structures with trivial isotropy. 2012. arXiv:1208.1340v1

\bibitem[McL]{mclean}M. Mclean, Reeb orbits and the minimal discrepancy of an isolated singularity. 2014. arXiv:1404.1857

\bibitem[McM]{branchcourse}C. McMullen, Course notes for ``Complex analysis on Riemann surfaces,'' \url{http://www.math.harvard.edu/~ctm/math213b/home/course/course.pdf}

\bibitem[MiWh94]{MiWh}M. Micallef and B. White,The structure of branch points in minimal surfaces and in pseudoholomorphic curves, \emph{Ann. Math.}, \textbf{139} (1994), 35-85.


\bibitem[Mo11]{Mo} A. Momin, Contact homology of orbit complements and implied existence, \emph{J. Mod. Dyn.}  \textbf{5} (2011), no. 3, 409-472.




\bibitem[Ne13]{jothesis}J. Nelson, Applications of automatic transversality in contact homology, Ph.D. thesis, University of Wisconsin - Madison, 2013.


\bibitem[Ne2]{jocompute} J. Nelson, Automatic transversality in contact homology II: Invariance and computations. In preparation.

\bibitem[Pa09]{P09} J. Pati, Contact homology of $S^1$-bundles over some symplectically reduced orbifolds. 2009. arXiv:0910.5934


\bibitem[Pa]{Pa}J. Pardon, An algebraic approach to virtual fundamental cycles on moduli spaces of $J$-holomorphic curves.	arXiv:1309.2370

\bibitem[RS93]{RS1} J. Robbin and D. Salamon,
The Maslov index for paths, \emph{Topology} \textbf{32} (1993), 827-844.



\bibitem[Sa99]{Spark}D. Salamon, Lectures on Floer homology, \emph{Symplectic geometry and topology}, 143-229, IAS/Park City Math. Ser., 7, \emph{Amer. Math. Soc.}, Providence, RI, 1999.

\bibitem[SZ92]{SZ}D. Salamon and E. Zehnder, Morse theory for periodic solutions of Hamiltonian systems and the Maslov index, \emph{Comm. Pure Appl. Math.} 45 (1992), no 10, 1303-1360. 

\bibitem[Sc95]{schwarz}M. Schwarz, Cohomology Operations from $S^1$-Cobordisms in Floer Homology, Ph.D. Thesis, ETH - Z\"urich, 1995.


\bibitem[Se00]{sgrade}P. Seidel,  Graded Lagrangian submanifolds. \emph{Bull. Soc. Math. France} \textbf{128} (2000), no. 1, 103-149.

\bibitem[Se06]{biased} P. Seidel,  A biased view of symplectic cohomology, \emph{Current Developments in Mathematics}, 2006: 211-253.

\bibitem[Si08]{Sief} R. Siefring, Relative asymptotic behavior of pseudoholomorphic half-cylinders. \emph{Comm. Pure Appl. Math.} \textbf{61} (2008), no. 12, 1631-1684. 

\bibitem[T07]{T} C.H. Taubes, The Seiberg-Witten equations and the Weinstein conjecture, \emph{Geom. Topol.} \textbf{11} (2007), 2117-2202.

\bibitem[TZ]{TZ}M. Tehrani and A. Zinger, On Symplectic Sum Formulas in Gromov-Witten Theory. 	arXiv:1404.1898


\bibitem[Us99]{U}I. Ustilovsky, Contact homology and contact structures on $S^{4m+1}$, Ph.D. thesis, Stanford University, 1999.


\bibitem[We10]{Wtrans}C. Wendl, Automatic transversality and orbifolds of punctured holomorphic curves in dimension four. \emph{Comment. Math. Helv.} \textbf{85} (2010), no. 2, 347-407. 

\bibitem[Wnotes]{Wnotes}C. Wendl, Lectures on holomorphic curves in symplectic and contact geometry, arXiv:1011.1690

\bibitem[MLY04]{ML}M.-L. Yau, Contact homology of subcritical Stein-fillable contact manifolds, Geom. Topol. 8 (2004),1243-1280.

\end{thebibliography}
\end{document}